\documentclass[11pt,a4paper]{article}

\usepackage[english]{babel}
\usepackage[T1]{fontenc}
\usepackage[utf8]{inputenc}
\usepackage[
  margin=2cm,
  includefoot,
  footskip=30pt,
]{geometry}
\usepackage{lmodern}
\usepackage{amsmath}
\usepackage{amssymb}
\usepackage{amsthm}
\usepackage{booktabs}
\usepackage{array}
\usepackage{graphicx}
\usepackage{cite}
\usepackage[dvipsnames]{xcolor}
\usepackage{tikz-cd}
\usetikzlibrary{decorations.markings,angles}
\usepackage{stmaryrd} \SetSymbolFont{stmry}{bold}{U}{stmry}{m}{n}
\usepackage{mathrsfs}
\usepackage{extarrows}
\usepackage{mathtools}
\usepackage{tensor}
\usepackage{enumerate}
\usepackage{adjustbox}
\usepackage[mathcal]{eucal}
\usepackage{hyperref}
\hypersetup{
	colorlinks,
	urlcolor={red!55!black},
	citecolor={green!60!black},
	linkcolor={red!55!black}
}
\usepackage{bm}
\usepackage{tabularx}
\usepackage{cleveref}
\usepackage[strict]{changepage}
\usepackage{tensor}

\usepackage{todonotes}
\usepackage{verbatim}

\usetikzlibrary{calc}
\tikzset{
    bvertex/.style = {
        circle,
        fill=black,
        outer sep=2.5pt,
        inner sep=2pt
    }
}
\tikzset{
    bfrvertex/.style = {
       inner sep=2.5pt, outer sep=3pt, fill=blue, draw=none}
    }

\mathchardef\mhyphen="2D
\def\on{\operatorname}

\makeatletter
\providecommand{\leftsquigarrow}{%
  \mathrel{\mathpalette\reflect@squig\relax}%
}
\newcommand{\reflect@squig}[2]{%
  \reflectbox{$\m@th#1\rightsquigarrow$}%
}
\makeatother

\definecolor{ao}{rgb}{0.0, 0.5, 0.0}

\newtheorem{theorem}{Theorem}[section]
\newtheorem{lemma}[theorem]{Lemma}

\newtheorem{proposition}[theorem]{Proposition}
\newtheorem{corollary}[theorem]{Corollary}

\theoremstyle{definition}
\newtheorem{construction}[theorem]{Construction}
\newtheorem{definition}[theorem]{Definition}
\newtheorem{notation}[theorem]{Notation}
\newtheorem{remark}[theorem]{Remark}
\newtheorem{example}[theorem]{Example}

\newcommand\noloc{%
  \nobreak
  \mspace{6mu plus 1mu}
  {:}
  \nonscript\mkern-\thinmuskip
  \mathpunct{}
  \mspace{2mu}
}

\newcommand\cocolon{%
  \nobreak
  \mspace{6mu plus 1mu}
  {:}
  \nonscript\mkern-\thinmuskip
  \mathpunct{}
  \mspace{2mu}
}

\newcommand{\rgraph}{{\bf G}}

\newcommand{\A}{\mathcal{A}}
\newcommand{\B}{\mathcal{B}}
\newcommand{\C}{\mathcal{C}}
\newcommand{\D}{\mathcal{D}}

\newcommand{\V}{\mathcal{V}}
\newcommand{\N}{\mathcal{N}}
\newcommand{\T}{\mathcal{T}}

\newcommand{\QT}{{\bf Q}_{\Delta,\ADE}}
\newcommand{\GT}{\mathscr{G}_{\Delta,\ADE}}
\newcommand{\ADE}{I}

\newcommand{\GS}{\mathscr{G}_{\rgraph,\ADE}}

\newcommand{\HH}{\on{HH}}

\newcommand{\glsec}{R^1{\bf \Gamma}}

\title{Cluster theory of topological Fukaya categories. Part II: Higher Teichm\"uller theory}
\author{Merlin Christ}
\date{\today}

\begin{document}
\maketitle

\abstract{We construct relative $3$-Calabi--Yau categories related with higher Teichm\"uller theory. We further study their corresponding cosingularity categories and the additive categorification of the corresponding cluster algebras. 

The input for our constructions is a marked surface with boundary and a Dynkin quiver $I$. In the case of the triangle, these categories have been described in recent work of Keller--Liu. For general surfaces, the categories are constructed via gluing along a perverse schober, categorifying the amalgamation of cluster varieties. The case $I=A_1$ was subject of the prequel paper. We show that the cosingularity category is equivalent to the corresponding Higgs category and to the topological Fukaya category of the marked surface valued in the $1$-Calabi--Yau cluster category of type $I$.}

\tableofcontents

\section{Introduction}

This paper and its prequel \cite{Chr22b} concern the study of additive categorifications of cluster algebras of surfaces in terms of topological Fukaya categories. The prequel focuses on the cluster algebras associated with marked surfaces that yield coordinates on the decorated Teichm\"uller space. This paper concerns more general cluster algebras giving coordinates on higher Teichm\"uller spaces arising from a marked surface and a choice of simply-laced simple Lie group $G$. The prequel thus corresponds to the case $G=\on{SL}_2/\on{PSL}_2$. 

Cluster algebras are a class of commutative algebras equipped with special generators called clusters that are related to each other via a combinatorial rule called mutation \cite{FZ02}. Cluster algebras admit a rich theory of categorification in terms of triangulated or extriangulated categories equipped with cluster tilting objects. The cluster tilting objects can be mutated and play the role of the clusters. In this categorification, the direct sum corresponds to the product in the cluster algebra, hence it is also called an additive categorification. There is also a different kind of categorification called monoidal categorification. Families of examples of additive categorifications of cluster algebras arise from triangulated cluster categories \cite{BMRRT06,Ami09} and more recently from extriangulated Higgs categories \cite{Wu21}.

In this paper, we establish an equivalence between a canonical class of Higgs categories and a class of $2$-periodic topological Fukaya categories of surfaces. The latter have been constructed by Dyckerhoff--Kapranov \cite{DK18} and arise as the global sections of a (co)sheaf of dg categories (or of $k$-linear stable $\infty$-categories). The topological Fukaya category can take values in any $2$-periodic category. The category relevant for us will be the $1$-Calabi--Yau cluster category $\C_\ADE$ of Dynkin type $\ADE$ corresponding to the Lie group $G$. The category $\C_\ADE$ can be defined as the cosingularity category 
\[ \C_\ADE\coloneqq \on{CoSing}(\Pi_2(\ADE))=\D^{\on{perf}}(\Pi_2(\ADE))/\D^{\on{fin}}(\Pi_2(\ADE))\]
 of the $2$-Calabi--Yau completion of $\ADE$. We show:

\begin{theorem}[\Cref{thm:Cosing_Fukaya,thm:GS_cluster_tilting,thm:equivHiggsCosing}]\label{introthm:equivalence}
Let ${\bf S}$ be a marked surface and $\ADE$ a Dynkin quiver. ${\bf S}$ is assumed to have non-empty boundary and no punctures. With this, we associate a relative $3$-Calabi--Yau category $\D^{\on{perf}}(\GS)$, see \Cref{def:GS} and \Cref{cor:global_sections_independence}. There exists an equivalence of stable $\infty$-categories between
\begin{enumerate}[i)]
\item the Higgs category $\mathcal{H}_{\GS}$,
\item the cosingularity category $\on{CoSing}(\GS)=\D^{\on{perf}}(\GS)/\D^{\on{fin}}(\GS)$, and
\item the $\C_\ADE$-valued topological Fukaya category $\on{Fuk}({\bf S},\C_\ADE)$.
\end{enumerate} 
\end{theorem}

We remark that the above categories, while stable, are equipped with additional $\infty$-categorical Frobenius exact structures, giving rise to extriangulated structures on their homotopy $1$-categories. The Higgs category $\mathcal{H}_{\GS}$ comes with a canonical cluster tilting object. 

The initial cluster seeds of the cluster algebras arising in higher Teichm\"uller theory are constructed via a gluing process along a triangulation of the surface, called amalgamation by Fock--Goncharov \cite{FG06a}. We categorify this amalgamation process in two ways: Firstly, we construct the above mentioned relative $3$-Calabi--Yau and extriangulated $2$-Calabi--Yau categories by gluing along the triangulation. This is formulated using a perverse sheaf of stable $\infty$-categories. On the underlying ice quivers with potentials, this yields the amalgamation. Secondly, we show that the canonical cluster tilting object in the Higgs category $\mathcal{H}_{\GS}$ arises via the gluing of the local cluster tilting objects. This uses a general gluing result for cluster tilting objects recently shown in \cite{Chr25b}. Finally, we note that a local version of \Cref{introthm:equivalence} was shown for ${\bf S}=\Delta$ a triangle in recent work of Keller--Liu \cite{KL25}.

The remainder of the introduction is structured as follows: We begin in \Cref{subsec:higherTeichmuller} with a brief summary of the relevant parts of higher Teichm\"uller theory. We then discuss in \Cref{subsec:3CY} the construction of the higher rank relative $3$-Calabi--Yau categories associated with marked surfaces. In \Cref{subsec:cosing} we discuss our results on its cosingularity category, in particular the description in terms of a $2$-periodic topological Fukaya category. Finally, we discuss in \Cref{subsec:HiggsCT} the relation with the corresponding Higgs category and the additive categorification of the corresponding cluster algebra.

\subsection{Background on higher Teichm\"uller theory and cluster varieties}\label{subsec:higherTeichmuller}

A higher Teichm\"uller space is loosely speaking a subset of a space of (potentially decorated) local systems of a simple Lie group $G$ on a topological surface. There are different versions of higher Teichm\"uller spaces and different ways to construct these. The first were the so-called Hitchin components \cite{Hit92,Lab06}. Another way to construct higher Teichm\"uller spaces uses positivity \cite{FG06}. We refer to \cite{Wie18} for an introductory survey on higher Teichm\"uller theory.

Let us consider the higher Teichm\"uller spaces related to the constructions of this paper. We let ${\bf S}$ be a closed oriented topological surface together with a set of points $M\subset {\bf S}$, consisting of marked points in the boundary, and punctures in the interior. We further choose a split semi-simple simply-laced algebraic group $G$. For instance in type $A_{n-1}$, one can choose $G=\on{SL}_n$. There is a complex algebraic variety $\mathcal{A}_{G,{\bf S}}$, called the cluster $\mathcal{A}$-variety by Fock--Goncharov (note that \cite{GS19} call it the cluster $K_2$-variety). It describes a decorated moduli space of representations of the group $G$. The variety $\mathcal{A}_{G,{\bf S}}$ has special cluster coordinates, which give its coordinate ring the structure of a cluster algebra. A full set of cluster coordinates, describing a cluster, can be associated with every triangulation of ${\bf S}$. Using the positivity of the cluster mutation rules, the decorated higher Teichm\"uller space can be defined as the subset where all cluster variables take positive values. In the case $G=\on{PSL}_2$, this recovers Penner's decorated Teichm\"uller space \cite{Pen12}.\\

\noindent {\bf Amalgamation}

The construction of the cluster coordinate systems on $\mathcal{A}_{G,{\bf S}}$ is based on the amalgamation construction of \cite{FG06a}. The amalgamation construction is of central importance to this work. When gluing two marked surfaces ${\bf S}_1,{\bf S}_2$ along boundary intervals to a marked surface ${\bf S}$, there is a corresponding restriction map $\mathcal{A}_{G,{\bf S}}\to \mathcal{A}_{G,{\bf S}_1}\times \mathcal{A}_{G,{\bf S}_2}$. To construct global coordinates, it thus suffices to construct them for for $\mathcal{A}_{G,\Delta}$ on with $\Delta$ the $3$-gon and then glue these along a choice of triangulation of ${\bf S}$. When gluing coordinates, the ice quivers of initial seeds of the corresponding cluster algebras are also glued along their frozen components. After the gluing, the frozen components along which was glued are unfrozen. This process is called (ice quiver) amalgamation.  

\begin{example}
The ice quiver of an initial seed of the cluster algebra of regular functions on $\mathcal{A}_{\on{SL}_4,\Delta}$ is depicted in \Cref{fig:A3_quiver}. We note that this ice quiver is particularly simple and corresponds to a particularly nice choice of reduced expression for $w_0$. In other Dynkin types, the ice quiver is never $\mathbb{Z}/3\mathbb{Z}$-symmetric.
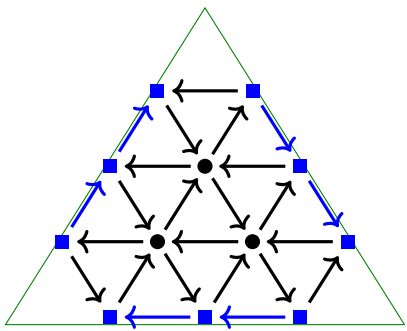
\begin{figure}[ht!]  
\begin{center}
\begin{tikzpicture}[xscale=1.25]
\draw[color=ao] (2.1,-1.1)--(0,3.1)--(-2.1,-1.1)--(2.1,-1.1);
\draw node[bfrvertex] (1) at (-0.5,2) {};
\draw node[bfrvertex] (2) at (0.5,2) {};
\draw node[bfrvertex] (3) at (-1,1) {};
\draw node[bvertex] (4) at (-0,1) {};
\draw node[bfrvertex] (5) at (1,1) {};
\draw node[bfrvertex] (6) at (-1.5,0) {};
\draw node[bvertex] (7) at (-0.5,0) {};
\draw node[bvertex] (8) at (0.5,0) {};
\draw node[bfrvertex] (9) at (1.5,0) {};
\draw node[bfrvertex] (10) at (-1,-1) {};
\draw node[bfrvertex] (11) at (0,-1) {};
\draw node[bfrvertex] (12) at (1,-1) {};

\draw[->, very thick] (2)--(1);
\draw[->, very thick] (1)--(4);
\draw[->, very thick] (4)--(2); 
\draw[->, very thick] (4)--(8);
\draw[->, very thick] (7)--(4);
\draw[->, very thick] (5)--(4);
\draw[->, very thick] (4)--(3);
\draw[->, very thick] (3)--(7);
\draw[->, very thick] (8)--(7);
\draw[->, very thick] (10)--(7);
\draw[->, very thick] (7)--(6);
\draw[->, very thick] (7)--(11);
\draw[->, very thick] (11)--(8);
\draw[->, very thick] (8)--(12);
\draw[->, very thick] (8)--(5);
\draw[->, very thick] (9)--(8);
\draw[->, very thick] (6)--(10);
\draw[->, very thick] (12)--(9);

\draw[->, blue, very thick] (3)--(1);
\draw[->, blue, very thick] (6)--(3);
\draw[->, blue, very thick] (11)--(10);
\draw[->, blue, very thick] (12)--(11);
\draw[->, blue, very thick] (2)--(5);
\draw[->, blue, very thick] (5)--(9);
\end{tikzpicture}
\caption{The ice quiver of the basic triangle for $G=\on{SL}_4$.}\label{fig:A3_quiver}\end{center}
\end{figure}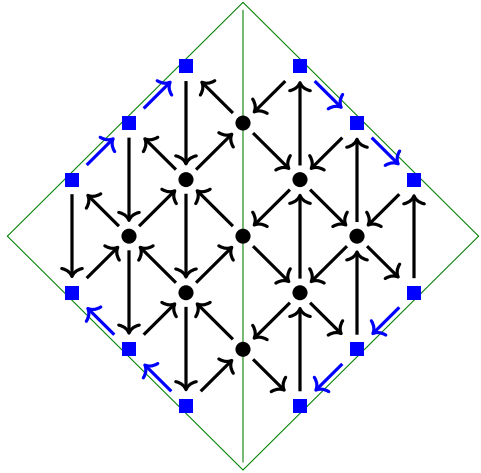
\begin{figure}[ht!]  
\begin{center}
\begin{tikzpicture}
\draw[color=ao] (0,3.1)--(-3.1,0)--(0,-3.1)--(3.1,0)--(0,3.1) (0,3)--(0,-3);
\draw node[bfrvertex] (1) at (-2.25,0.75) {};
\draw node[bfrvertex] (2) at (-1.5,1.5) {};
\draw node[bfrvertex] (3) at (-0.75,2.25) {};
\draw node[bfrvertex] (4) at (-2.25,-0.75) {};
\draw node[bfrvertex] (5) at (-1.5,-1.5) {};
\draw node[bfrvertex] (6) at (-0.75,-2.25) {};
\draw node[bvertex] (7) at (-0.75,0.75) {};
\draw node[bvertex] (8) at (-1.5,0) {};
\draw node[bvertex] (9) at (-0.75,-0.75) {};

\draw node[bvertex] (10) at (0,1.5) {};
\draw node[bvertex] (11) at (0,0) {};
\draw node[bvertex] (12) at (0,-1.5) {};

\draw[->, blue, very thick] (1)--(2);
\draw[->, blue, very thick] (2)--(3);
\draw[->, blue, very thick] (6)--(5); 
\draw[->, blue, very thick] (5)--(4);
\draw[->, very thick] (7)--(2);
\draw[->, very thick] (3)--(7);
\draw[->, very thick] (8)--(1);
\draw[->, very thick] (2)--(8);
\draw[->, very thick] (8)--(7);
\draw[->, very thick] (7)--(9);
\draw[->, very thick] (9)--(8);
\draw[->, very thick] (1)--(4);
\draw[->, very thick] (4)--(8);
\draw[->, very thick] (8)--(5);
\draw[->, very thick] (5)--(9);
\draw[->, very thick] (9)--(6);
\draw[->, very thick] (6)--(12);
\draw[->, very thick] (12)--(9);
\draw[->, very thick] (9)--(11);
\draw[->, very thick] (11)--(7);
\draw[->, very thick] (7)--(10);
\draw[->, very thick] (10)--(3);

\draw node[bfrvertex] (1') at (2.25,0.75) {};
\draw node[bfrvertex] (2') at (1.5,1.5) {};
\draw node[bfrvertex] (3') at (0.75,2.25) {};
\draw node[bfrvertex] (4') at (2.25,-0.75) {};
\draw node[bfrvertex] (5') at (1.5,-1.5) {};
\draw node[bfrvertex] (6') at (0.75,-2.25) {};
\draw node[bvertex] (7') at (0.75,0.75) {};
\draw node[bvertex] (8') at (1.5,0) {};
\draw node[bvertex] (9') at (0.75,-0.75) {};

\draw[->, blue, very thick] (2')--(1');
\draw[->, blue, very thick] (3')--(2');
\draw[->, blue, very thick] (5')--(6'); 
\draw[->, blue, very thick] (4')--(5');
\draw[->, very thick] (2')--(7');
\draw[->, very thick] (7')--(3');
\draw[->, very thick] (1')--(8');
\draw[->, very thick] (8')--(2');
\draw[->, very thick] (7')--(8');
\draw[->, very thick] (9')--(7');
\draw[->, very thick] (8')--(9');
\draw[->, very thick] (4')--(1');
\draw[->, very thick] (8')--(4');
\draw[->, very thick] (5')--(8');
\draw[->, very thick] (9')--(5');
\draw[->, very thick] (6')--(9');
\draw[->, very thick] (12)--(6');
\draw[->, very thick] (9')--(12);
\draw[->, very thick] (11)--(9');
\draw[->, very thick] (7')--(11);
\draw[->, very thick] (10)--(7');
\draw[->, very thick] (3')--(10);
\end{tikzpicture}
\caption{The amalgamation ice quiver of the $4$-gon for $G=\on{SL}_4$.}\label{fig:A3_square_quiver}
\end{center}
\end{figure}
The amalgamation ice quiver associated with the triangulated $4$-gon is depicted in \Cref{fig:A3_square_quiver}. It is the amalgamation of two copies of the ice quiver depicted in \Cref{fig:A3_quiver}.
\end{example}

To produce cluster coordinates on the variety $\mathcal{A}_{G,\Delta}$ associated with a triangle, the strategy of \cite{GS19} is to again use amalgamation, by noting that there is a more fundamental building piece than the triangle, consisting of a triangle with a short side labeled by a simple braid twist in the braid group corresponding to $G$. Amalgamating these along a reduced expression for $w_0$ yields $\mathcal{A}_{G,\Delta}$ together with an initial cluster seed. As an example, the ice quiver for the reduced expression $w_0=s_1s_2s_3s_1s_2s_1$ (with $G$ of type $A_3$) is depicted in \Cref{fig:A3_quiver}. Results of \cite{GS19} include that the cluster seeds constructed from different choices of reduced expression or choice of orientation of the triangle are mutation equivalent, and different triangulations also yield mutation equivalent cluster coordinates on $\mathcal{A}_{G,{\bf S}}$.\\

The constructions of this paper are concerned with a categorification in terms of relative $3$-Calabi--Yau and $2$-Calabi--Yau categories of the above amalgamation of $\mathcal{A}_{G,\Delta}$ along a triangulation of ${\bf S}$. We do not consider in this paper the categorification of the more fundamental building piece associated with a simple braid twist and their amalgamation, but hope to return to this in future work. We also do not allow marked surface with empty boundary or with punctures (meaning marked points in the interior), and generalizations of our results to these would be very interesting.

A version of gluing, called fusion, for wild character varieties (equivalently spaces of Stokes local systems), generalizing cluster varieties, was considered in \cite{Boa14}. While these spaces are expected to always carry cluster structures, their cluster seeds cannot always be constructed by amalgamation, see for instance the case of braid varieties \cite{CGGLSS25}. However in many cases (for instance when ${\bf S}$ is not the disc) the conjectural cluster seeds were constructed via the amalgamation of small triangles in \cite[Section 8]{GK21}.

\subsection{Higher rank \texorpdfstring{$3$}{3}-Calabi--Yau categories of marked surfaces}\label{subsec:3CY}

Let ${\bf S}$ be an oriented topological surface with non-empty boundary $\partial {\bf S}$ and $M\subset \partial {\bf S}$ a collection of marked points. Let $\ADE$ be a Dynkin quiver. We associate with ${\bf S}$ and $\ADE$ a relative $3$-Calabi--Yau category, which can be described as the derived category of a relative $3$-Calabi--Yau dg algebra. These categories have been defined and their representation theory has been well understood in the case $\ADE=A_1$, see for instance \cite{Lab09,BS15,KQ20,Chr21b,CHQ23}. In type $A_n$ with $n\geq 1$, non-relative versions of these categories appear in \cite{Abr18,Smi21}. In physics, the corresponding theories fall into 'class S', see for instance \cite{GMN13}.\\ 

\noindent {\bf The case ${\bf S}=\Delta$.}

We first suppose that ${\bf S}=\Delta$ is the triangle. We orient the triangle, meaning we distinguish one of its three boundary edges, and call the oriented triangle the basic triangle. Keller--Liu \cite{KL25} associate with the basic triangle and the Dynkin quiver $\ADE$ the relative $3$-Calabi--Yau completion of the functor 
\[
\on{proj}(\ADE)^{\times 3}\longrightarrow \on{Fun}([1],\on{proj}(\ADE))\,,\quad (X,Y,Z)\mapsto (X\to 0)\oplus (Y\to Y)\oplus (0\to Z)\,,
\]
where $[1]$ is the poset $\{0\to 1\}$ and $\on{proj}(\ADE)$ is the additive $1$-category of finitely generated projective $\ADE$-modules. We pass to the derived $\infty$-categories to obtain the functor
\[
(\tilde{D}_1,\tilde{D}_2,\tilde{D}_3)\colon \D(\Pi_2(\ADE))^{\times 3}\longrightarrow \D(\GT)
\]
where $\Pi_2(\ADE)$ is the $2$-Calabi--Yau completion and $\GT$ is the relative $3$-Calabi--Yau completion. We note that $\D(\GT)$ does not depend (up to equivalence) on the orientation of $\ADE$, i.e.~it depends only on the Dynkin type. 

The three right adjoint functors $\tilde{D}_1^R,\tilde{D}_2^R,\tilde{D}_3^R$ define a constructible sheaf $\mathcal{F}_{\Delta,\ADE}$ of stable $\infty$-categories on the $3$-spider embedded in the basic triangle (using the exit path description of constructible sheaves). We depict $\mathcal{F}_{\Delta,\ADE}$ in \Cref{fig:3gonschober}.

\begin{figure}[ht!]
\begin{center}
\begin{tikzpicture}[scale=1.3]
\draw[very thick, color=ao] (-30:2)--(90:2)--(210:2);
\draw[very thick, color=ao, dashed](210:2)--(-30:2);
\fill[orange] (-30:2) circle(0.1);
\fill[orange] (90:2) circle(0.1);
\fill[orange] (210:2) circle(0.1);
\fill[black] (0,0) circle (0.1);
\draw[very thick, ->] (0,0)--(30:1);
\draw[very thick, ->] (0,0)--(150:1);
\draw[very thick, ->] (0,0)--(270:1);
\node() at (0.6,-0.25){$\D(\mathscr{G}_{\Delta,\ADE})$};
\node() at (22:1.65){$\D(\Pi_2(\ADE))$};
\node() at (158:1.65){$\D(\Pi_2(\ADE))$};
\node() at (270:1.3){$\D(\Pi_2(\ADE))$};
\node() at (54:0.65){$\tilde{D}_3^R$};
\node() at (174:0.65){$\tilde{D}_1^R$};
\node() at (245:0.65){$\tilde{D}_2^R$};
\end{tikzpicture}
\end{center}
\caption{The perverse schober $\mathcal{F}_{\Delta,\ADE}$ on the basic triangle, parametrized by the $3$-spider. The distinguished bottom edge is dashed.} \label{fig:3gonschober}
\end{figure}
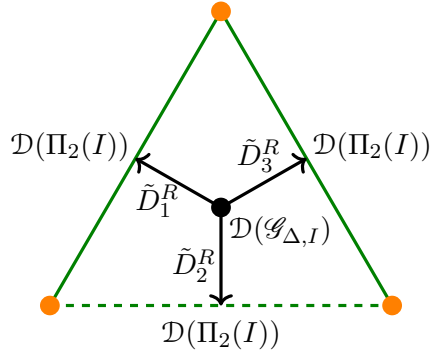

\begin{theorem}[\Cref{thm:3gonschober}]
The constructible sheaf $\mathcal{F}_{\Delta,\ADE}$ defines a perverse schober parametrized by the $3$-spider in the sense of \cite{Chr22b,CHQ23}.
\end{theorem}

The proof involves showing that the inverse dualizing bimodule (also known as the inverse Serre functor) of $\GT$ is invertible, which we prove by showing that any stable $\infty$-category with a categorical compactification has this property, see \Cref{prop:invertibleSerre}.

Due to the relative $3$-Calabi--Yau structure, the inverse dualizing bimodule describes, up to shift, the cotwist functor of the adjunction $(\tilde{D}_1,\tilde{D}_2,\tilde{D}_3)\colon \D(\Pi_2(\ADE))^{\times 3}\leftrightarrow \D(\GT) \noloc (\tilde{D}_1^R,\tilde{D}_2^R,\tilde{D}_3^R)$. Using the action of this autoequivalence of $\D(\GT)$ on $\mathcal{F}_{\Delta,\ADE}$, we can show the independence (up to equivalence) of $\mathcal{F}_{\Delta,\ADE}$ the choice of orientation of the basic triangle, see \Cref{prop:Z6symmetry}.\\

\noindent {\bf Gluing for arbitrary ${\bf S}$.}

We next allow ${\bf S}$ to be an arbitrary marked surface. We choose a triangulation of ${\bf S}$, dual to a trivalent spanning ribbon graph $\rgraph$, and an orientation of each triangle as above. We define a $\rgraph$-parametrized perverse schober $\mathcal{F}_{\rgraph,\ADE}$, meaning a constructible sheaf of stable $\infty$-categories on $\rgraph$ with local properties, categorify the properties of perverse sheaves. We require that $\mathcal{F}_{\rgraph,\ADE}$ restricts on each triangle to the perverse schober $\mathcal{F}_{\Delta,\ADE}$. The definition of $\mathcal{F}_{\rgraph,\ADE}$ thus amounts to specifying the local identifications at the edges of the ribbon graph. There, an involution $\sigma$ of the $2$-Calabi--Yau completion $\Pi_2(\ADE)$, see \Cref{def:ADEinvolution}, plays an important role, which is inserted along every edge. In type $A_n$ with the linear orientation, $\sigma$ arises from the reflection symmetry of the $2$-Calabi--Yau completion $\Pi_2(A_n)$, exchanging the vertices $i$ and $n+1-i$.  

The formalism of parametrized perverse schobers allows to prove the following:

\begin{theorem}[\Cref{cor:global_sections_independence} and \Cref{prop:global_sections_GS,prop:rel3CY}]~
\begin{enumerate}[(i)]
\item The stable $\infty$-category of global sections $\glsec(\rgraph,\mathcal{F}_{\rgraph,\ADE})$ of $\mathcal{F}_{\rgraph,\ADE}$ is independent on the choice of the ideal triangulation (and thus the choice of $\rgraph$) up to equivalence.
\item There exists an equivalence of $\infty$-categories $\glsec(\rgraph,\mathcal{F}_{\rgraph,\ADE})\simeq \D(\GS)$ with $\GS$ a smooth connective relative left $3$-Calabi--Yau dg category. 
\end{enumerate}
\end{theorem}

\noindent {\bf Ice quivers with potential}

Keller--Liu \cite{KL25} sketch that the relative Calabi--Yau completion $\GT$ of the basic triangle is Morita equivalent to the relative Ginzburg dg category of an ice quiver with potential. The underlying ice quiver is furthermore expected to coincide with the ice quiver describing the initial cluster seed of the corresponding cluster algebra of regular functions on the cluster $\mathcal{A}$-variety of the basic triangle. In type $A$, it is however clear that the two ice quivers coincide.

The ice quiver of a general triangulated marked surface is obtained via amalgamation in the sense of Fock--Goncharov \cite{FG06a}. We show in \Cref{subsec:icequiveramalgamation} a general gluing result (specifically a homotopy pushout square) for relative Ginzburg dg categories arising as the amalgamation of two ice quivers with potential, see \Cref{thm:Ginzburg_amalgamation}. Applied to amalgamation of the ice quivers with potential associated with the triangles of a triangulated marked surfaces, we obtain the following:

\begin{proposition}[\Cref{prop:Ginzburg_alg}]
There exists an equivalence of $\infty$-categories between $\D(\GS)$ and the derived $\infty$-category of the relative Ginzburg dg category of the amalgamation ice quiver with potential defined in \Cref{def:surface_ice_quiver}.
\end{proposition}

\subsection{Cosingularity categories}\label{subsec:cosing}

Let $\mathscr{G}$ be a smooth dg category. Then the derived category of finite dg $\mathscr{G}$-modules $\D^{\on{fin}}(\mathscr{G})$ is contained in the perfect derived category $\D^{\on{perf}}(\mathscr{G})$. The cosingularity category is defined as the Verdier quotient $\on{CoSing}(\mathscr{G})=\D^{\on{perf}}(\mathscr{G})/\D^{\on{fin}}(\mathscr{G})$. Cosingularity categories of (absolute!) $3$-Calabi--Yau categories were considered by Amiot \cite{Ami09} under the name of generalized cluster categories. Note that given absolute Calabi--Yau structures, the cosingularity category and the singularity category are exchanged by Koszul duality \cite{GS20}.

We study the cosingularity categories of the relative $3$-Calabi--Yau categories $\glsec(\rgraph,\mathcal{F}_{\rgraph,\ADE})\simeq \D(\GS)$, generalizing the results in the case $\ADE=A_1$ of the prequel \cite{Chr22b}. The result in loc.~cit.~was that the cosingularity category is equivalent to the topological Fukaya category of the surface valued in the $1$-periodic derived category. The $1$-periodic derived category can be seen as the $1$-Calabi--Yau cluster category $\C_{A_1}$ of type $A_1$, see below.\\ 

\noindent{\bf The $1$-Calabi--Yau cluster category $\C_\ADE$ of Dynkin type}

The $1$-Calabi--Yau cluster category $\C_\ADE$ can be defined as the cosingularity category of the $2$-Calabi--Yau completion $\Pi_2(\ADE)$, also known as the derived preprojective algebra. $\C_\ADE$ can also be described as the derived orbit category $\D(\ADE)/\tau$ of $\D(\ADE)$ by the Auslander--Reiten translation functor $\tau\simeq U[-1]$, with $U$ the Serre functor. The $1$-Calabi--Yau cluster category can furthermore be described via matrix factorizations of the type $\ADE$ simple surface singularity, and is thus $2$-periodic. 

We study $\C_\ADE$ in \Cref{sec:1CYclustercat}. The main purpose of that section is to describe the suspension functor $[1]$ of $\C_\ADE$. The main result, is the following: 

\begin{proposition}[\Cref{prop:shift}]\label{introprop:shift}
The explicit dg isomorphism $\sigma\colon \Pi_2(\ADE)\to \Pi_2(\ADE)$ from \Cref{def:ADEinvolution} induces the suspension functor $[1]\colon \on{CoSing}(\Pi_2(\ADE))\to \on{CoSing}(\Pi_2(\ADE))$. 
\end{proposition}
 
For the proof of \Cref{prop:shift}, we introduce a novel dg category $\widetilde{\Pi_2(\ADE)}$, which interpolates between $\D^{\on{perf}}(\ADE)$ and $\D^{\on{perf}}(\Pi_2(\ADE))$, see also \Cref{fig:squareofautomorphisms} for a summary of their relation.

In type $A$, the involution $\sigma$ acts by reflecting the quiver along its middle (and reversing the signs of the degree $1$ loops). In types $D_{2n},E_7,E_8$, the involution acts trivially on the vertices of $\Pi_2(\ADE)$ (but changes the signs of some morphisms). In types $D_{2n-1},E_6$, $\sigma$ acts as a partial reflection. Identifying the vertices of $\Pi_2(\ADE)$ with the positive roots, the involution $\sigma$ induces a well known involution on the positive roots, given by the action of $-w_0$. This involution is also used in the amalgamation og cluster $\mathcal{A}$-varieties of triangles and in the same way in the definition of the the perverse schober $\mathcal{F}_{\rgraph,\ADE}$.\\  

\noindent {\bf The $\C_\ADE$-valued topological Fukaya category}

Given an oriented marked surface ${\bf S}$ and a $2$-periodic dg category $C$, Dyckerhoff--Kapranov \cite{DK18} show that there is a corresponding topological Fukaya category $\on{Fuk}({\bf S},C)$, defined uniquely up to a contractible choice. Its perfect derived $\infty$-category arises as the global sections of a perverse schober with generic stalk $\D^{\on{perf}}(C)$, no singularities and trivial monodromy local system in the sense of \cite{Chr23}. We review this relation in \Cref{subsec:2_periodic_top_Fukaya}.

The passage to the ($\on{Ind}$-completed) cosingularity category commutes with the gluing along the perverse schober $\mathcal{F}_{\rgraph,\ADE}$. The cosingularity category $\on{CoSing}(\GT)$ can thus be described as the global sections of a quotient perverse schober $\mathcal{F}^{\on{clst}}_{\rgraph,\ADE}$. The generic stalk of $\mathcal{F}_{\rgraph,\ADE}$ is $\D(\Pi_2(\ADE))$, and the generic stalk of the quotient $\mathcal{F}^{\on{clst}}_{\rgraph,\ADE}$ is thus given by the ($\on{Ind}$ completion of the) $1$-Calabi--Yau cluster category $\C_\ADE=\on{CoSing}(\Pi_2(\ADE))$. Further $\mathcal{F}^{\on{clst}}_{\rgraph,\ADE}$ has no singularities. Using \Cref{introprop:shift}, we show:

\begin{theorem}[\Cref{thm:global_sections_=_Fukaya}]\label{introthm:Fuk}
The monodromy local system of the quotient perverse schober $\mathcal{F}_{\rgraph,\ADE}^{\on{clst}}$ of $\mathcal{F}_{\rgraph,\ADE}$ is trivial. Its global sections thus describe the ($\on{Ind}$-completed) $\C_\ADE$-valued topological Fukaya category:
\[
\glsec(\rgraph,\mathcal{F}_{\rgraph,\ADE}^{\on{clst}})\simeq \on{Ind}\on{Fuk}({\bf S},\C_\ADE)\,.
\]
\end{theorem}


Another variant of higher rank topological Fukaya categories have been considered in \cite{HKS21} (in relation with stability conditions). There, they take values in $\D^{\on{perf}}(\ADE)$. The $\C_\ADE$-valued topological Fukaya category $\on{Fuk}({\bf S},\C_\ADE)$ is equivalent to the orbit $\infty$-category of $\on{Fuk}({\bf S},\D^{\on{perf}}(\ADE))$ by the autoequivalence induced by the 'local' Auslander--Reiten translation functor $\tau$ on $\D^{\on{perf}}(\ADE))$. For example in the case $\ADE=A_1$, $\on{Fuk}({\bf S},\D^{\on{perf}}(\ADE))$ describes the orbit category of the usual $\D(k)^{\on{perf}}$-valued topological Fukaya category by $[1]$, see also \cite{Chr25}.

\subsection{The Higgs category and cluster tilting theory}\label{subsec:HiggsCT}

\noindent {\bf Cluster categories and Higgs categories}

As mentioned above, the cluster category of a smooth connective (absolute) $3$-Calabi--Yau dg category $\mathscr{G}$ is given by cosingularity category $\on{CoSing}(\mathscr{G})$, see \cite{Ami09}. Under mild assumptions, the cosingularity category $\on{CoSing}(\mathscr{G})$ is triangulated $2$-Calabi--Yau and the image of a connective generator of $\D(\mathscr{G})$ in $\on{CoSing}(\mathscr{G})$ defines a cluster tilting object. Such $2$-Calabi--Yau triangulated categories with cluster tilting objects can be used for the additive categorification of cluster algebras, see also below. These cluster algebras however have no coefficients (meaning no frozen cluster variables). 

For the additive categorification of cluster algebras with coefficients, one can use Frobenius extriangulated categories in the sense of \cite{NP19}. The indecomposable injective--projective objects in the Frobenius extriangulated category appear in every cluster tilting object and correspond to the frozen cluster variables. The analog of the cluster category is this context is the Higgs category, introduced by Yilin Wu \cite{Wu21}. Its construction takes as input a suitable connective dg category $\mathscr{G}$ together with a relative $3$-Calabi--Yau dg functor $B\to \mathscr{G}$. The homotopy cofiber of $B\to \mathscr{G}$ defines an absolute $3$-Calabi--Yau dg category $\mathscr{G}^\circ$. Instead of passing to the cosingularity category, one passes to the so-called relative cluster category, defined as the Verdier quotient $\D^{\on{perf}}(\mathscr{G})/\D^{\on{fin}}(\mathscr{G}^{\circ})$ by the derived category of finite $\mathscr{G}^\circ$-modules. The Higgs category $\mathcal{H}_\mathscr{G}$ arises as a certain extension closed subcategory of $\D^{\on{perf}}(\mathscr{G})/\D^{\on{fin}}(\mathscr{G}^{\circ})$ and thus inherits an extriangulated structure (which is even Frobenius). The image of $\mathscr{G}$ in the relative cluster category lies in $\mathcal{H}_{\mathscr{G}}$ and a generator of $\mathscr{G}$ is mapped to a cluster tilting object. Further, $\mathcal{H}_{\mathscr{G}}$ is extriangulated $2$-Calabi--Yau.\\ 

\noindent {\bf Higgs categories for higher Teichm\"uller theory}

For any Higgs category, there is a canonical functor $\mathcal{H}_{\mathscr{G}}\subset \D^{\on{perf}}(\mathscr{G})/\D^{\on{fin}}(\mathscr{G}^{\circ})\to \on{CoSing}(\mathscr{G})$ to the cosingularity category, but in general it is not an equivalence of categories. We however show that it is in the case that $\mathscr{G}=\GS$ is the relative $3$-Calabi--Yau dg category from above. The proof of this follows the same strategy as the proof of the statement in the case $\ADE=A_1$ given in the prequel \cite{Chr22b}:

Firstly, as in \cite{Chr22b}, we equip $\on{CoSing}(\GS)$ with an $\infty$-categorical Frobenius exact structure. This is the relative exact structure arising from the boundary restriction functor.

The second step is to show the following:

\begin{theorem}[\Cref{thm:GS_cluster_tilting}]
The exact $\infty$-category $\on{CoSing}(\GS)\simeq \on{Fuk}({\bf S},\C_\ADE)$ admits a canonical cluster tilting object.
\end{theorem}

The proof is based on two quite trivial results: firstly that $\on{CoSing}(\GT)$ admits a cluster tilting object, which was showin in \cite{KL25}. Secondly, we use the novel gluing result for cluster tilting subcategories along perverse schobers recently proven in \cite{Chr25b}. 

Thirdly, and finally, we show that the functor $\mathcal{H}_{\mathscr{G}_{\rgraph,\ADE}}\to \on{CoSing}(\mathscr{G}_{\rgraph,\ADE})$ is an exact functor mapping a cluster tilting object to a cluster tilting object, inducing an equivalence between the endomorphism algebras. This implies that the functor is an equivalence of exact $\infty$-categories. This yields the equivalence between i) and ii) in \Cref{introthm:equivalence}.\\

\noindent {\bf The additive categorifications of cluster algebras from higher Teichm\"uller theory}

For an additive categorification of a cluster algebra $A$ in terms of a triangulated or extriangulated category $\mathcal{H}$ one requires that
\begin{itemize}
\item $\mathcal{H}$ admits cluster tilting objects,
\item there is a bijection between isomorphisms classes of cluster tilting objects in $\mathcal{H}$ and the clusters of the cluster algebra $A$,
\item the above equivalence identifies the endomorphism (ice) quivers of the cluster tilting objects with the quivers of the cluster and is compatible with mutation.  
\end{itemize}

Furthermore, one asks for a so-called cluster character, which is map $\on{obj}(\mathcal{H})\to A$ sending direct sums to products and satisfying a formula related with cluster mutation. The cluster character thus gives the direct link between the categorification and the cluster algebra. A cluster character exists given a cluster tilting object, see \cite{Pal08,Pla11,KW23}. To obtain a well behaved cluster characters, it is however important that in the ice quivers of the cluster tilting objects no $2$-cycles or loops appear. This is the case if $\mathcal{H}$ arises from an ice quiver with a non-degenerate potential, which simply means that under iterated mutations of the ice quiver with potential at non-frozen vertices no $2$-cycles appear. In this case, the cluster character gives a bijection between equivalence classes of reachable rigid objects and cluster variables, see \cite{CKLP13,KW23}.

To obtain a full additive categorification of the cluster algebras arising from higher Teichm\"uller theory, the following two tasks remain to be completed:
\begin{itemize}
\item Show that the ice quivers associated with the basic triangle in \cite{GS19,KL25} agree beyond type $\ADE=A$ with the linear orientation. By amalgamation, the ice quivers associated with general triangulations then coincide.
\item Show that the ice quiver potential described in \Cref{def:surface_ice_quiver} is non-degenerate. 
\end{itemize} 

\subsection*{Acknowledgements}

I thank Bernhard Keller for many helpful discussions and specifically for suggesting the definition of the dg category $\widetilde{\Pi_2(\ADE)}$ from \Cref{def:tilde_Pi2}. I further thank Philip Boalch, Roger Casals, Mikhail Gorsky, Norihiro Hanihara, Gustavo Jasso, Mauro Porta, Yilin Wu and Qiu Yu for helpful discussions.

The author is a member of the Hausdorff Center for Mathematics at the University of Bonn (DFG GZ 2047/1, project ID 390685813). Most of this work was completed while the author was based at the IMJ-PRG in Paris. During that time, he received funding from the European Union’s Horizon 2020 research and innovation programme under the Marie Skłodowska-Curie grant agreement No 101034255.

\section{Higher categorical preliminaries}

We freely use the language of $\infty$-categories, as developed in \cite{HTT,HA,SAG,Ker,Cis}.

\subsection{Linear \texorpdfstring{$\infty$}{infinity}-categories}

We denote by $\on{St}$ the $\infty$-category of (small) stable $\infty$-categories and exact functors. We denote the $\infty$-category of presentable, stable $\infty$-categories and colimit preserving functors by $\mathcal{P}r^L_{\on{St}}$. Given an $\infty$-category $\C$, we denote by $\C^{\on{c}}$ the subcategory of compact objects. Given a small $\infty$-category $\C$, we denote its $\on{Ind}$-completion by $\on{Ind}(\C)\in \mathcal{P}r^L$. Note that if $\C$ is stable, then $\on{Ind}(\C)^{\on{c}}$ is equivalent to the idempotent completion of $\C$. 

Let $R$ be an $\mathbb{E}_\infty$-ring spectrum and $\on{Mod}_R$ its symmetric monoidal $\infty$-category of module spectra. We denote by $\on{LinCat}_R\coloneqq \on{Mod}_{\on{Mod}_R}(\mathcal{P}r^L_{\on{St}})$ the $\infty$-category of $R$-linear $\infty$-categories. Note that by definition, $R$-linear $\infty$-categories are stable and $R$-linear functors preserve colimits. We will mostly be concerned with the case $R=k$ a field in this paper. 

Given an $R$-linear $\infty$-category $\C$, and two objects $X,Y\in \C$, the morphism object $\on{Mor}_\C(X,Y)\in \on{Mod}_R$ is the essentially unique object equipped with a map $\alpha\colon \on{Mor}_\C(X,Y)\otimes X\to Y$ in $\C$ such that for every $C\in \on{Mod}_R$ the morphism between mapping spaces
\[ \on{Map}_{\on{Mod}_R}(C,\on{Mor}_{\C}(X,Y))\to \on{Map}_{\C}(C\otimes X,\on{Mor}_{\C}(X,Y)\otimes X)\xrightarrow{\alpha} \on{Map}_{\C}(C\otimes X, Y) \]
is an equivalence, see also \cite[Def.~4.2.1.28]{HA}. 

Given a ($k$-linear) dg category $C$, we denote by $\D(C)\in \on{LinCat}_k$ its derived $\infty$-category, which is defined as the $\on{Ind}$-completion of the dg nerve of the dg category $\on{Perf}(C)$ of cofibrant compact right dg $C$-modules. The passage to derived $\infty$-categories defines a functor 
\[ \D(\mhyphen)\colon \on{dgCat}\to \on{LinCat}_k\] 
with $\on{dgCat}$ the nerve of the $1$-category of dg categories. This functor further maps homotopy colimits with respect to the quasi-equivalence model structure to $\infty$-categorical colimits.

\subsection{Exact \texorpdfstring{$\infty$}{infinity}-categories and cluster tilting objects}

We recall some aspects of the theory of exact $\infty$-categories. 

\begin{definition}[$\!\!$\cite{Bar15}]
An exact $\infty$-category is a triple $(\mathcal{C},\mathcal{C}_{\dagger},\mathcal{C}^{\dagger})$, where $\mathcal{C}$ is an additive $\infty$-category and $\mathcal{C}_{\dagger},\mathcal{C}^{\dagger}\subset \mathcal{C}$ are subcategories (called subcategories of inflations and deflations), satisfying that 
\begin{enumerate}[(1)]
\item every morphism $0\rightarrow X$ in $\mathcal{C}$ lies in $\mathcal{C}_{\dagger}$ and every morphism $X\rightarrow 0$ in $\mathcal{C}$ lies in $\mathcal{C}^\dagger$.
\item pushouts in $\mathcal{C}$ along morphisms in $\mathcal{C}_{\dagger}$ exist and lie in $\mathcal{C}_{\dagger}$. Dually, pullbacks in $\mathcal{C}$ along morphisms in $\mathcal{C}^{\dagger}$ exist and lie in $\mathcal{C}^{\dagger}$. 
\item  Given a commutative square in $\mathcal{C}$ of the form
\[
\begin{tikzcd}
X \arrow[r, "a"] \arrow[d, "b"] & Y \arrow[d, "c"] \\
X' \arrow[r, "d"]               & Y'              
\end{tikzcd}
\]
the following are equivalent.
\begin{itemize}
\item The square is pullback, $c\in \mathcal{C}_{\dagger}$ and $d\in \mathcal{C}^{\dagger}$.
\item The square is pushout, $b\in \mathcal{C}_{\dagger}$ and $a\in \mathcal{C}^{\dagger}$.
\end{itemize}
\end{enumerate}
We typically abuse notation and simply refer to $\mathcal{C}$ as the exact $\infty$-category.
\end{definition}

If an $\infty$-category $\C$ is equipped with an exact structure, its homotopy $1$-category $\on{ho}\C$ inherits the structure of an extriangulated category, see \cite{Kle20,NP20}. 

\begin{definition}
An exact sequence $X\rightarrow Y\rightarrow Z$ in an exact $\infty$-category $\C$ consists of a fiber and cofiber sequence in $\mathcal{C}$
\[
\begin{tikzcd}
X \arrow[r, "a"] \arrow[dr, "\square", phantom] \arrow[d] & Y \arrow[d, "b"] \\
0 \arrow[r]                & Z               
\end{tikzcd}
\]
with $a$ an inflation and $b$ a deflation.

A functor between exact $\infty$-categories is called exact if it maps exact sequences to exact sequences.
\end{definition}

\begin{definition}\label{def:frob}
Let $\C$ be an exact $\infty$-category.
\begin{enumerate}[1)]
\item An object $P\in \mathcal{C}$ is called projective if every exact sequence $X\rightarrow Y \rightarrow P$ splits. An object $I\in \mathcal{C}$ is called injective if every exact sequence $I\rightarrow Y \rightarrow Z$ splits.
\item We say that $\mathcal{C}$ has enough projectives if for each object $X\in \mathcal{C}$ there exists an exact sequence $X\rightarrow P\rightarrow Y$ with $P$ projective. Similarly, we say that $\mathcal{C}$ has enough injectives if for each object $Y\in \mathcal{C}$ there exists an exact sequence $Y\rightarrow I \rightarrow X$ with $I$ injective.
\item We call $\mathcal{C}$ a Frobenius exact $\infty$-category if $\mathcal{C}$ has enough projectives and injectives and the classes of projective and injective objects coincide. 
\end{enumerate}
\end{definition}

\begin{remark}\label{rem:inducedexstr}
Let $F\colon \C\to \D$ be an exact functor (in the stable sense) between stable $\infty$-categories. Then there exists an exact structure on $\C$, where a fiber and cofiber sequence in $\mathcal{C}$ is exact if and only if its image under $F$ splits. We will refer to it as the exact structure on $\C$ induced by $F$.

If $F$ is spherical, then the exact structure induced by $F$ is Frobenius, see \cite{BS21,Chr22b}.
\end{remark}

\begin{example}\label{ex:schober_inducedexstr}
Let $\mathcal{F}$ be a $\rgraph$-parametrized perverse schober. Then the $\infty$-category of global sections $\glsec(\rgraph,\mathcal{F})$ inherits by \Cref{rem:inducedexstr} a Frobenius exact structure from the spherical functor \cite[Cor.~4.7]{Chr25b}
\[
\prod_{e\in \rgraph_1^\partial} \on{ev}_e \colon \glsec(\rgraph,\mathcal{F})\longrightarrow \prod_{e\in \rgraph_1^\partial} \mathcal{F}(e)\,.
\]
\end{example}

We call a subcategory $\T$ of an additive $\infty$-category an additive subcategory if it is closed under finite direct sums and direct summands.

\begin{definition}
Let $\C$ be an exact $\infty$-category and $\T\subset \C$ an additive subcategory.
\begin{enumerate}[(1)]
\item We call $\T$ rigid if all exact sequences $T\to Y\to T$ in $\C$ with $T\in \T$ and $Y\in \C$ split.
\item We say that $\T$ has the right $2$-term resolution property if for all $X\in \C$ there exists an exact sequence $X\to T_0\to T_1$ in $\C$ with $T_0,T_1\in \T$. 
\item We say that $\T$ has the left $2$-term resolution property if for all $X\in \C$ there exists an exact sequence $T_0\to T_1\to X$ in $\C$ with $T_0,T_1\in \T$.
\item We say that $\T$ has the two-sided $2$-term resolution property if it has the left and the right $2$-term resolution property.
\item We call $\T$ a cluster tilting subcategory if $T$ is rigid and has the two-sided $2$-term resolution property.
\item Suppose that $\T=\on{Add}(T)$ is the additive closure of an object $T\in \C$. We call $T$ a cluster-tilting object if $\T\subset \C$ is a cluster tilting subcategory and $T$ is basic, meaning $T$ is a finite direct sum of indecomposable objects which are pairwise non-isomorphic. 
\end{enumerate}
\end{definition}

We note that all cluster tilting subcategories appearing in this paper arise from cluster tilting objects. Note also that the property of being cluster tilting can be checked on the extriangulated homotopy $1$-category.

\subsection{\texorpdfstring{$\infty$}{infinity}-categorical group actions}

Given a group $G$, we denote by $\on{BG}$ its classifying space, which can be defined as the nerve of the $1$-category with a unique object $\ast$ with endomorphisms $G$. 

\begin{definition}~
\begin{enumerate}[(1)]
\item An action of a group $G$ on a small stable $\infty$-category $\C\in \on{St}$ is defined as a functor
\[
\rho\colon BG\longrightarrow \on{St},\ast\mapsto \C\,.
\]
\item Given an action $\rho\colon BG\to \on{St}$ of a group $G$ on a small stable $\infty$-category $\C$, the group quotient $\C_G$ is defined as the colimit $\on{colim}(\rho)\in \on{St}$. 
\end{enumerate}
Group actions and group quotients for large stable $\infty$-categories are defined similarly, replacing $\on{St}$ by $\mathcal{P}r^L_{\on{St}}$. 
\end{definition}

We note that $\on{Ind}$-completion $\on{Ind}\colon \on{St}\to \mathcal{P}r^L_{\on{St}}$ preserves colimits and thus group quotients. The forgetful functor $\on{LinCat}_k\to \on{P}r^L_{\on{St}}$ also preserves colimits.

\begin{remark}\label{rem:orbit_category}
A $\mathbb{Z}$-action $\rho\colon B\mathbb{Z}\longrightarrow \on{St} \ast\mapsto \C$ is fully determined by the autoequivalence $F=\rho(1)\colon \C\simeq \C$, see Lemma 4.3 in \cite{Chr25}. In this case, we write $\C/F=\C_G$ for the group quotient and call $\C/F$ the orbit $\infty$-category.
\end{remark}

\begin{remark}
A (not necessarily strict) $\mathbb{Z}$-action on a dg category $C$ induces a $G$-action on its $k$-linear derived $\infty$-category $\D(C)$, and the derived $\infty$-category of the orbit dg category is equivalent to the orbit $\infty$-category, see \cite[Section 4.2]{Chr25}.
\end{remark}

\begin{definition}
Let $\rho_\C,\rho_\D\colon BG\to \on{LinCat}_k$ be $G$-actions on $k$-linear $\infty$-categories $\C,\D$. A $G$-equivariant functor $F\colon \C\to\D$ consists of a functor
\[
\rho_\ast \colon \Delta^1\times BG\longrightarrow \on{LinCat}_k
\]
such that $\rho_0=\rho_\C$ and $\rho_1=\rho_\D$. 
\end{definition}

\begin{lemma}\label{lem:orbit_cat_induced_functor}
Let $C,D$ be dg categories with strict $\mathbb{Z}$-actions, with the actions of $1\in \mathbb{Z}$ given by the dg functors $F_C\colon C\to C$ and $F\colon D\to D$. Let $T\colon C\to D$ be a $\mathbb{Z}$-equivariant dg functor. 
\begin{enumerate}
\item[(1)] Then $T$ induces a dg functor 
\[ T/F_\ast\colon C/F_C\longrightarrow D/F_D\,.\]
\end{enumerate}
Let $\rho_{\D(C)},\rho_{\D(D)}\colon B\mathbb{Z}\to \on{LinCat}_k$ be the induced $\mathbb{Z}$-actions on the derived $\infty$-categories and $\D(T)\colon \D(C)\to \D(D)$ the induced $\mathbb{Z}$-equivariant functor. Passing to the colimit over $BG$ defines a functor 
\[ \D(T)/\D(F_{\ast})\colon \D(C)/\D(F_C)\longrightarrow \D(D)/\D(F_D)\,.\]
\begin{enumerate}
\item[(2)] There exists an equivalence of functors 
\[ \D(T)/\D(F_{\ast})\simeq \D(T/F_\ast)\,.\]  
\end{enumerate}
\end{lemma}

\begin{proof}
Part (1) follows from \cite[Prop~3.8]{FKQ24}. Part (2) follows from the observation that the dg functor $T/F_\ast$ appears in the restriction of a diagram $\Delta^1\times B\mathbb{Z}^\triangleright\to \on{dgCat}_k$ to $\Delta^1\times \ast'$, with $\ast'$ the cone point. Passing to derived categories, $\D(T/F_\ast)$ thus arises as the tip of a morphism between the colimit cones of $\rho_{\D(C)},\rho_{\D(D)}$, and hence is equivalent to $\D(T)/\D(F_{\ast})$ by the universal property of the colimit cones.
\end{proof}

\subsection{Categorical compactifications and inverse Serre functors}

\begin{definition}
Let $\C$ be a $R$-linear $\infty$-category which is dualizable in the symmetric monoidal $\infty$-category $\on{LinCat}_R$ with dual $\C^\vee$. The identity functor $\on{id}_\C\colon \C\to \C$ induces the evaluation bimodule $\on{ev}_\C\colon\C\otimes \C^{\vee}\to \on{Mod}_R$.
\begin{enumerate}[(1)]
\item We call $\C$ smooth if $\on{ev}_\C$ admits a left adjoint $\on{ev}_\C^L$. In this case, we can obtain from $\on{ev}_\C^L$ an $R$-linear endofunctor $\on{id}_\C^!\colon \C\to \C$, called the inverse Serre functor or inverse dualizing bimodule of $\C$.
\item We call $\C$ proper if $\on{ev}_\C$ admits an $R$-linear right adjoint. If $\C$ is compactly generated (as will be all $k$-linear $\infty$-categories considered in this paper), then $\C$ is proper if and only if $\on{Mor}_\C(X,Y)\in \on{Mod}_R$ is compact for all $X,Y\in \C^{\on{c}}$. The right adjoint of $\on{ev}_e$ corresponds to an endofunctor $\on{id}_\C^*\colon \C\to \C$. If $\C$ is compactly generated, then $\on{id}_\C^*$ is a Serre functor, see \cite[Lem.~2.22]{Chr23}.
\end{enumerate}
\end{definition}

We note that if $\C$ is smooth and proper, then $\on{id}_\C^*$ and $\on{id}_\C^!$ are inverse autoequivalences, see \cite{Chr23}.

\begin{definition}
Let $\C$ be a smooth $R$-linear $\infty$-category. A categorical compactification of $\C$ consists of a smooth and proper $R$-linear $\infty$-category $\hat{\C}$ together with a compact objects preserving $R$-linear localization functor $\pi\colon\hat{\C}\twoheadrightarrow \C$, satisfying that the Serre functor of $\hat{\C}$ preserves the kernel of $\pi$.  
\end{definition}

For $R=k$ a field, not every smooth $k$-linear $\infty$-category admits a categorical compactification, see \cite{Efi20}. Note also that for the categorical compactifications considered in \cite{Efi20}, the additional condition on the Serre functor preserving the kernel is not included. This latter condition on the Serre functor also appears in \cite{KS23}.

\begin{example}\label{ex:FS}
Let $\C=\D\mathcal{W}(X)$ be the $k$-linear derived $\infty$-category of the wrapped Fukaya category of a Liouville manifold $X$. Suppose that $X$ is equipped with a Lefschetz fibration $f\colon X\to \mathbb{D}$ with regular fiber $f^{-1}(1)\subset X$. Let $\hat{C}=\D\on{FS}(f)$ be the derived $\infty$-category of the Fukaya--Seidel category. There is a pushout diagram in $\mathcal{P}r^L$, see \cite{GPS18}, as follows:
\[
\begin{tikzcd}
\D\mathcal{W}(f^{-1}(1)) \arrow[d] \arrow[r, "F"] & \hat{\C} \arrow[d, "\pi"] \\
0 \arrow[r]                                             & \C                      
\end{tikzcd}
\]
We suppose that $\D\mathcal{W}(f^{-1}(1))$ admits a left Calabi--Yau structure, which holds for instance under mild assumptions if the fiber $f^{-1}(1)$ is Weinstein, see \cite{Gan13}. As shown in \cite{Chr23}, in this situation, the Serre functor of $\hat{\C}$ preserves the kernel of $\pi$, hence $\pi$ defines a categorical compactification.

Choosing $X$ to be Milnor fibre of ADE type $\ADE$ in complex dimension $2$, we obtain a categorical compactification of $\D\mathcal{W}(X)\simeq \D(\Pi_2(\ADE))$, see \cite{LU21}.
\end{example}

I thank Mauro Porta for a private communication in which the following statement was obtained. 

\begin{proposition}\label{prop:invertibleSerre}
Let $\C$ be a smooth $R$-linear $\infty$-category which admits a categorical compactification. Then the inverse Serre functor $\on{id}_\C^!$ is invertible.
\end{proposition}

\begin{proof}
Let $\pi\colon \hat{\C} \to \C$ be a categorical compactification. Since $\on{id}_{\hat{\C}}^*$ and $\on{id}_{\hat{\C}}^!\simeq (\on{id}_{\hat{\C}}^*)^{-1}$ preserve $\on{ker}(\pi)$, we find that they induce inverse functors on the quotient $\C\simeq \hat{\C}/\on{ker}(\pi)$. The autoequivalence of $\C$ induced by $\on{id}_\C^!$ is given by $\pi\circ \on{id}_{\hat{\C}}^!\circ \pi^{R}$, with $\pi \dashv \pi^R$. Note that $\pi\circ \on{id}_{\hat{\C}}^!\circ \pi^{R}\simeq \on{id}_{\C}^!$, see Lemma 2.33.(1) in \cite{Chr23}. 
\end{proof}

\begin{lemma}\label{lem:resolutionofSOD}
Let $\C$ be a smooth $R$-linear $\infty$-category with an admissible semiorthogonal decomposition $(\A,\B)$, meaning that the semiorthogonal decomposition has an $R$-linear gluing functor $F\colon \A\to \B$ in the sense of \cite{DKSS21}. Denote by $G$ the right adjoint of $F$. Suppose that 
\begin{itemize}
\item $\hat{\A}\twoheadrightarrow \A$ and $\hat{\B}\twoheadrightarrow \B$ are categorical compactifications and that $\hat{\A},\hat{\B}$ are compactly generated,
\item that $F$ lifts to a compact objects preserving $k$-linear functor $\hat{F}\colon \hat{A}\to \hat{B}$ with right adjoint $\hat{G}$,
\item the following diagrams commute:
\begin{equation}\label{eq:adjointablity}
\begin{tikzcd}
\hat{\A} \arrow[r, "\hat{F}"] \arrow[d, two heads] & \hat{\B} \arrow[d, two heads] \\
\A \arrow[r, "F"]                                  & \B                           
\end{tikzcd} 
\quad\quad\quad
\begin{tikzcd}
\hat{\B} \arrow[r, "\hat{G}"] \arrow[d, two heads] & \hat{\A} \arrow[d, two heads] \\
\B \arrow[r, "G"]                                  & \A                           
\end{tikzcd}
\end{equation} 
\end{itemize}
Then $\hat{\C}=\hat{A}\times^{\rightarrow}_{\hat{F}}\hat{B}\twoheadrightarrow \C$ is a categorical compactification. 
\end{lemma}

\begin{proof}
Denote by $U_{\hat{\A}},U_{\hat{\B}},U_{\hat{\C}}$ the Serre functors of $\hat{\A},\hat{\B},\hat{\C}$.
We first note that $\hat{G}$ automatically preserves compact objects: since $\hat{\A},\hat{\B}$ are smooth and proper and compactly generated, we find that the right adjoint of $\hat{G}$ is given by $U_{\hat{\A}}\circ F\circ U_{\hat{\B}}^{-1}$.

The lax limit $\hat{\C}$ arises as the pullback
\[
\begin{tikzcd}
\hat{\C} \arrow[r] \arrow[d] \arrow[rd, "\lrcorner", phantom] & \hat{\A} \arrow[d] \\
{\on{Fun}(\Delta^1,\hat{\B})} \arrow[r, "\on{ev}_0"]          & \hat{\B}          
\end{tikzcd}
\]
in both $\on{LinCat}_R$ and, since the above functors and $\infty$-categories all dualizable, also in the $\infty$-category $\on{LinCat}_R^{\on{dual}}$ of dualizable $R$-linear $\infty$-categories. By \cite[Cor.~3.13]{Chr23}, it follows that $\hat{\C}$ is proper. 

The right adjoint diagram defines a pushout diagram in $\on{LinCat}_R^{\on{dual}}$, hence $\hat{\C}$ is also smooth by \cite[Cor.~3.11]{Chr23}.

It remains to show that the Serre functor preserves the kernel of $\hat{\C}\twoheadrightarrow \C$. We identify objects of $\hat{\C}$ with triples $(a,b,\eta)$, with $a\in \hat{\A}$, $b\in \hat{\B}$ and $\eta\colon F(A)\to b$. There are fully faithful functors $\iota_{\hat{\A}}\colon \hat{\A}\hookrightarrow \hat{\C}$ and $\iota_{\hat{\B}}\colon \hat{\B}\hookrightarrow \hat{\C}$ and adjunctions $\iota_{\hat{\A}}^{LL}\dashv \iota_{\hat{\A}}^L\dashv \iota_{\hat{\A}}$ and $\iota_{\hat{\B}}\dashv \iota_{\hat{\B}}^R\dashv \iota_{\hat{\B}}^{RR}$, where the functors act on objects as follows: 
\begin{itemize}
\item $\iota_{\hat{\A}}(a)=(a,0,0)$, $\iota_{\hat{\A}}^L(a,b,\eta)=a$, and $\iota_{\hat{\A}}^{LL}(a)=(a,F(a),F(a)=F(a))$.
\item $\iota_{\hat{\B}}(b)=(0,b,0)$, $\iota_{\hat{\B}}^R(a,b,\eta)=b$ and $\iota_{\hat{\B}}^{RR}(b)=(G(b),b,\on{counit}\colon FG(b)\to b)$.
\end{itemize}
There are equivalences as follows:
\[
\iota_{\hat{A}}^L\circ U_{\hat{\C}}^{-1}\circ \iota_{\hat{\A}}\simeq U_{\hat{\A}}^{-1}
\]
\[
\iota_{\hat{B}}^R\circ U_{\hat{\C}}^{-1}\circ \iota_{\hat{\A}}\simeq \hat{F} \circ U_{\hat{\A}}^{-1}
\]
\[
\iota_{\hat{B}}^R\circ U_{\hat{\C}}\circ \iota_{\hat{\B}}\simeq U_{\hat{\B}}
\]
\[
\iota_{\hat{A}}^L\circ U_{\hat{\C}}\circ \iota_{\hat{\B}}\simeq \hat{G}\circ U_{\hat{B}}
\]
The kernel of the functor $\hat{\C}\twoheadrightarrow \C$ is stably generated by the images of the kernels of $\hat{\A}\twoheadrightarrow \A$ and $\hat{\B}\twoheadrightarrow \B$ under $\iota_{\hat{\A}}$ and $\iota_{\hat{\B}}$. Both kernels are mapped by $U_{\hat{\C}}$ to the kernel, as follows from the above equivalences and the commutativity of the diagrams \eqref{eq:adjointablity}.  
\end{proof}

\section{The \texorpdfstring{$1$}{1}-cluster categories of Dynkin type}\label{sec:1CYclustercat}

Let $k$ be the base field. Let $\ADE$ be a Dynkin quiver. The $2$-Calabi--Yau completion $\Pi_2(\ADE)$, see \cite{Kel11}, can be defined as the dg category with
\begin{itemize}
\item objects the vertices of $\ADE$,
\item the set of morphisms freely generated by the morphisms $a\colon x\to y$ and $a^\dagger\colon y\to x$ in degree $0$ with $a\in \ADE_1$ any arrow in $\ADE$, and the endomorphisms $l_x\colon x\to x$ in degree $1$ (in the homological grading convention) with $x\in\ADE_0$ any vertex, and
\item the differential determined on the generators by $d(a)=d(a^\dagger)=0$ for $a\in \ADE_1$ and $d(l_x)=\sum_{a\in \ADE_1} \on{id}_x (aa^\dagger-a^\dagger a) \on{id}_x$.
\end{itemize}

We remark that $\Pi_2(\ADE)$ is independent of the orientation of the quiver $\ADE$ up to dg isomorphism.

\begin{definition}
The $1$-cluster category $\C_\ADE$ of type $\ADE$ is defined as the cosingularity category 
\[ \C_\ADE\coloneqq \on{CoSing}(\Pi_2(\ADE))=\D^{\on{perf}}(\Pi_2(\ADE))/\D^{\on{fin}}(\Pi_2(\ADE)\,,\]
with $\D^{\on{fin}}(\Pi_2(\ADE))\subset \D^{\on{perf}}(\Pi_2(\ADE))$ the subcategory of objects whose underlying $k$-module is perfect.  
\end{definition}

In \Cref{subsec:orbit} we will describe $\C_\ADE$ and $\D^{\on{perf}}(\Pi_2(\ADE))$ as orbit categories. We then describe a collection of fiber and cofiber sequences in $\C_\ADE$ arising from the Auslander--Reiten quiver of the triangulated perfect derived category $D^{\on{perf}}(\ADE)$ in \Cref{subsec:ARsequences}.

The $1$-cluster category $\C_\ADE$ is $2$-periodic, by which we mean here an equivalence of $k$-linear endofunctors $[2]\simeq \on{id}_{\C_\ADE}$. We next record a novel description of the involution $[1]\colon \C_\ADE\to\C_\ADE$ in terms of an involution $\sigma$ of $\Pi_2(\ADE)$, that we prove in \Cref{subsec:involution}. We note that the $2$-periodicity of $\C_\ADE$ is well known fact. On the level of homotopy categories, it follows for instance from the equivalence between $\on{ho}\C_\ADE$ and the $2$-periodic category of matrix factorizations of the corresponding simple surface singularity, see for instance \cite{AIR15} and \cite[Thm.~3.3]{Han22}. An enhanced version of the $2$-periodicity is proven in \cite[Prop.~4.10]{HI22}.

\begin{definition}\label{def:ADEinvolution}
Let $\ADE$ be a Dynkin quiver, with an orientation chosen as below. We define an involution $\sigma\colon \Pi_2(\ADE)\to \Pi_2(\ADE)$ on generators as follows:
\begin{itemize}
\item In type 
\[
A_n= \begin{tikzcd}
1 \arrow[r, "a_1"] & 2 \arrow[r, "a_2"] & \dots \arrow[r, "a_{n-2}", no head] & n-1 \arrow[r, "a_{n-1}"] & n
\end{tikzcd}
\] we set 
\[ \sigma(i)=n-i+1 \in \Pi_2(\ADE)\]
and 
\begin{align*} 
\sigma(a_i)& =a_{n-i}^\dagger\\
\sigma(a_i^\dagger)&=a_{n-i}\\ 
\sigma(l_i)&=-l_{\sigma(i)}\,.
\end{align*} 
\item In type 
\[ D_n =\begin{tikzcd}
1 \arrow[r, "a_1"] & 2 \arrow[r, "a_2"] & \dots \arrow[r, "a_{n-3}"] & n-2  \arrow[r, "a_{n-2}"]\arrow[d, "a_{n-1}"] & n-1 \\
                   &                    &                                     & n                                                                 &    
\end{tikzcd}\]
with $n\geq 4$, we distinguish between $n$ even and $n$ odd. 

If $n$ is even, we set $\sigma(i)=i$ and $\sigma(a_i)=a_i$, $\sigma(a_i^\dagger)=-a_i^\dagger$, $\sigma(l_i)=-l_i$. 

If $n$ is odd, we set 
\[ \sigma(i)=\begin{cases} i & i\not= n-1,n\\ n& i=n-1 \\ n-1& i=n\end{cases}\]
and 
\begin{align*}
\sigma(a_i)& =\begin{cases} a_i & i\not= n-2,n-1\\ a_{n-1}& i=n-2 \\ a_{n-2}& i=n-1\end{cases}\\
\sigma(a_i^\dagger)& =\begin{cases} -a_i^\dagger & i\not= n-2,n-1\\ -a_{n-1}^\dagger& i=n-2 \\ -a_{n-2}^\dagger& i=n-1\end{cases}\\
\sigma(l_i)&  = -l_{\sigma(i)}
\end{align*}
\item In type
\[
E_6= \begin{tikzcd}
1 \arrow[r, "a_1"] & 2 \arrow[r, "a_2"] & 3 \arrow[r, "a_3"] \arrow[d, "a_5"] & 4 \arrow[r, "a_4"] & 5 \\
                   &                    & 6                                   &                    &  
\end{tikzcd}
\]
we set 
\[ \sigma(i) =\begin{cases} 6-i & i\neq 6\\ 6 & i=6\end{cases}\]
and
\begin{align*}
\sigma(a_i)&= \begin{cases} a_{5-i}^\dagger & i\neq 5 \\ a_5 & i=5\end{cases}\\
\sigma(a_i^\dagger)&= \begin{cases} a_{5-i} & i\neq 5 \\ -a_5^\dagger & i=5\end{cases}\\
\sigma(l_i)&= -l_{\sigma(i)}\,.
\end{align*}
\item In types 
\[
E_7= \begin{tikzcd}
1 \arrow[r, "a_1"] & 2 \arrow[r, "a_2"] & 3 \arrow[r, "a_3"] & 4 \arrow[r, "a_4"] \arrow[d, "a_6"] & 5 \arrow[r, "a_5"] & 6 \\
                   &                    &                    & 7                                   &                    &  
\end{tikzcd}
\]
and
\[
E_8=\begin{tikzcd}
1 \arrow[r, "a_1"] & 2 \arrow[r, "a_2"] & 3 \arrow[r, "a_3"] & 4 \arrow[r, "a_4"] & 5 \arrow[r, "a_5"] \arrow[d, "a_7"] & 6 \arrow[r, "a_6"] & 7 \\
                   &                    &                    &                    & 8                                   &                    &  
\end{tikzcd}
\]
 we set $\sigma(i)=i$ and $\sigma(a_i)=a_i$, $\sigma(a^\dagger_i)=-a_i^\dagger$ and $\sigma(l_i)=-l_i$. 
\end{itemize}
We note that in each case $\sigma$ commutes with the differential in $\Pi_2(\ADE)$ and thus indeed defines a dg functor.
\end{definition}

\begin{proposition}\label{prop:shift}
The following diagram of $k$-linear $\infty$-categories commutes
\[
\begin{tikzcd}
\D^{\on{perf}}(\Pi_2(\ADE)) \arrow[r, "\D^{\on{perf}}(\sigma)"] \arrow[d] & \D^{\on{perf}}(\Pi_2(\ADE)) \arrow[d] \\
{\C_\ADE} \arrow[r, "{[1]}"]           & {\C_\ADE}  
\end{tikzcd}
\]
with the vertical functors given by the quotient functor $\D^{\on{perf}}(\Pi_2(\ADE))\twoheadrightarrow \on{CoSing}(\Pi_2(\ADE))=\C_\ADE$. In other words, $\D^{\on{perf}}(\sigma)$ induces the suspension functor $[1]$ on $\C_\ADE$. 
\end{proposition}

\begin{remark}\label{rem:involution_longest_element}
The objects of $\Pi_2(\ADE)$, or equivalently the vertices of the Dynkin quiver $\ADE$, are in bijection with the simple roots. The involution $\sigma$ from \Cref{def:ADEinvolution} induces a well-known involution of the simple roots, given by the formula $\alpha_i\mapsto -w_0(\alpha_i)$.

In representation theory, on the level of objects, the involution $\sigma$ appears already in \cite{Gab80}, used for the description of the Nakayama permutation on the Auslander--Reiten quiver of $k\ADE$ (i.e.~the action of the Serre functor). The formulas for the action of $\sigma$ on morphisms also appear in \cite[Def.~4.6]{BBK02} in a description of the Nakayama automorphism of the module category of the preprojective algebra $H_0\Pi_2(\ADE)$. The involution there is also described for an arbitrary orientation of $\ADE$. Note that the category of finitely generated projective modules $\on{proj}(H_0\Pi_2(\ADE))$ is equivalent to the additive homotopy $1$-category of $\on{CoSing}(\Pi_2(\ADE))$, where the Nakayama automorphism and the suspension functor thus induce the same autoequivalence. This also follows from \Cref{prop:shift} using the fact that $\C_\ADE$ admits a right $1$-Calabi--Yau structure, see \cite{KL23}.
\end{remark}

\subsection{Orbit categories}\label{subsec:orbit}

We fix a Dynkin quiver $\ADE$. The $k$-linear $\infty$-category $\D(\ADE)$ admits a Serre functor $U$, and the action of $U[-1]$ induces a $\mathbb{Z}$-action on $\D(\ADE)$. Note that $U[-1]\simeq \tau$ acts as Auslander--Reiten translation on the Auslander--Reiten quiver. We denote by $\D^{\on{perf}}(\ADE)/U[-1]$  the orbit $\infty$-category, see \Cref{rem:orbit_category}. In the dg setting, we denote the Serre functor of $\on{Perf}(\ADE)$ by $U^{\on{dg}}$. The dg orbit category was introduced in \cite{Kel05}, see also \cite{FKQ24} for a further treatment.

\begin{proposition}\label{prop:1-cluster_orbit}
The $1$-cluster category is equivalent to the orbit $\infty$-category, as well as to the derived $\infty$-category of the orbit dg category $\on{Perf}(\ADE)/U^{\on{dg}}[-1]$:
\[
\C_\ADE\simeq \D^{\on{perf}}(\ADE)/U[-1]\simeq \D^{\on{perf}}(\on{Perf}(\ADE)/U^{\on{dg}}[-1])\,.
\]
\end{proposition}

\begin{proof}
By \cite[Theorem 3.3]{Han22}, the dg cosingularity category is equivalent to the orbit dg category. The derived $\infty$-categories of the orbit dg category is equivalent to the orbit $\infty$-category by \cite[Prop.~4.5]{Chr25}.
\end{proof}

We next introduce the dg category $\widetilde{\Pi_2(\ADE)}$ arising from an infinite quiver, which reduces to both $\D^{\on{perf}}(\Pi_2(\ADE))$ and $\D^{\on{perf}}(\ADE)$ via the passage to the orbit $\infty$-category, or the cosingularity category, respectively.

\begin{definition}\label{def:tilde_Pi2}
We define the dg category $\widetilde{\Pi_2(\ADE)}$ as follows:
\begin{itemize}
\item Objects of $\widetilde{\Pi_2(\ADE)}$ are pairs $(i,x)$ with $i\in \mathbb{Z}$ and $x\in \ADE_0$ a vertex. 
\item The morphisms of $\widetilde{\Pi_2(\ADE)}$ are freely generated by the following morphisms:
\begin{itemize}
\item $a_i\colon (i,x)\to (i,y)$ in degree $0$ for each arrow $a\colon x\to y$ in $\ADE$. 
\item $a^\dagger_i\colon (i,y)\to (i+1,x)$ in degree $0$ for each arrow $a\colon x\to y$ in $\ADE$.
\item $l_{i,x}\colon (i,x)\to (i+1,x)$ in degree $1$.
\end{itemize}
If $a=a_j$ in the notation from \Cref{def:ADEinvolution}, we also simply write $a_{i,j},a_{i,j}^\dagger$ for $(a_j)_i, (a_j)_i^\dagger$. 
\item The differential is determined on the generators by 
\[ d(a_{i})=d(a^\dagger_{i})=0\,,\quad\quad\quad d(l_{i,x})=\sum_{a\in \ADE_1, i\in \mathbb{Z}}\on{id}_{i+1,x}(a_{i+1}a^\dagger_i-a^\dagger_ia_{i})\on{id}_{i,x}\,.\]
\end{itemize}
\end{definition}

The underlying ungraded quiver of $\widetilde{\Pi_2(\ADE)}$ is given by the Auslander--Reiten quiver of the triangulated category $D^{\on{perf}}(\ADE)$. The differentials of the degree $1$ arrows give the mesh relations in the Auslander--Reiten quiver, so that by \cite[Prop.~4.6]{Hap87} we can define a dg functor $\pi^{\on{dg}}\colon \widetilde{\Pi_2(\ADE)}\to \on{Perf}(\ADE)$, mapping each object $(i,x)$ to the corresponding indecomposable object appearing in the Auslander--Reiten quiver. This dg functor gives rise to a $k$-linear functor between $\infty$-categories $\pi=\D^{\on{perf}}(\pi^{\on{dg}})\colon  \D^{\on{perf}}(\widetilde{\Pi_2(\ADE)})\to \D^{\on{perf}}(\ADE)$. We will show below in \Cref{prop:derivedcategoryasCoSing} that $\pi$ exhibits $\D^{\on{perf}}(\ADE)$ as the cosingularity category of $\D^{\on{perf}}(\widetilde{\Pi_2(\ADE)})$. We remark that a related construction appears in \cite[Section 4]{FKQ24}, with a similar description of $\on{Perf}(\ADE)$ as a cosingularity category, see \cite[Thm.~4.4]{FKQ24}, and it would be interesting to clarify the precise relation with the results below. These constructions are further related by Koszul duality to a description of Happel \cite{Hap87} of the triangulated category $D^{\on{perf}}(I)$ as the singularity category of the so-called repetitive algebra, see \cite[Section 4.3]{FKQ24}.

There is an apparent strict $\mathbb{Z}$-action on $\widetilde{\Pi_2(\ADE)}$, such that the action of $1\in \mathbb{Z}$ is given by translation $T\colon\widetilde{\Pi_2(\ADE)}\to \widetilde{\Pi_2(\ADE)},\ (i,x)\mapsto (i+1,x)$. The dg functor $F^{\on{dg}}\colon \widetilde{\Pi_2(\ADE)}\to \Pi_2(\ADE)$, given by the assignments $(i,x)\mapsto x$, $a_i\mapsto a$, $a_i^\dagger\mapsto a^\dagger$, $l_{i,x}\mapsto l_x$ induces a dg isomorphism between the dg orbit category $\widetilde{\Pi_2(\ADE)}/T$ and $\Pi_2(\ADE)$. We denote by $F\colon \D(\widetilde{\Pi_2(\ADE)})\to \D(\Pi_2(\ADE))$ the functor obtained from $F^{\on{dg}}$ by passing to derived $\infty$-categories.

The $\mathbb{Z}$-action on $\widetilde{\Pi_2(\ADE)}$ induces a $\mathbb{Z}$-action on the perfect derived $\infty$-category $\D^{\on{perf}}(\widetilde{\Pi_2(\ADE)})$. By \cite[Prop.~4.5]{Chr25}, the orbit $\infty$-category $\D^{\on{perf}}(\widetilde{\Pi_2(\ADE)})/T$ is equivalent to $\D^{\on{perf}}(\Pi_2(\ADE))$. 

We depict the appearing $k$-linear $\infty$-categories, together with compatible autoequivalences (we define $\tilde{\sigma}$ in \Cref{subsec:involution} below), in a commutative diagram in \Cref{fig:squareofautomorphisms}.

\begin{figure}[ht]
\begin{center}
\begin{tikzcd}
                                                                                                       & \D^{\on{perf}}(\widetilde{\Pi_2(\ADE)}) \arrow[rd, "\on{colim}_{B\mathbb{Z}}"', "F"] \arrow[ld, two heads, "\pi"', "{\on{CoSing}}"] \arrow["{\D^{\on{perf}}(\tilde{\sigma})}"', loop, distance=2em, in=125, out=55] &                                                                                               \\
\D^{\on{perf}}(\ADE) \arrow[rd, "\mathbb{Z}\on{-orbit} = \on{colim}_{B\mathbb{Z}}"'] \arrow["{[1]}"', loop, distance=2em, in=215, out=145] &                                                                                                                                               & \D^{\on{perf}}(\Pi_2(\ADE)) \arrow[ld, two heads, "\on{CoSing}"] \arrow["\D^{\on{perf}}(\sigma)"', loop, distance=3em, in=35, out=325] \\
                                                                                                       & \on{CoSing}(\Pi_2(\ADE)) \arrow["{[1]\simeq \on{CoSing}(\sigma)}"', loop, distance=2em, in=305, out=235]                                    &                                                                                              
\end{tikzcd}
\caption{The the $1$-Calabi--Yau category $\C_\ADE=\on{CoSing}(\Pi_2(\ADE))$ arises both as the cosingularity category of an orbit category and as an orbit category of the cosingularity category of $\D^{\on{perf}}(\widetilde{\Pi_2(\ADE)})$. The diagram commutes by the fact that colimits commute with colimit. Depicted are also compatible automorphisms, induced by $\D^{\on{perf}}(\tilde{\sigma})$.}\label{fig:squareofautomorphisms}
\end{center}
\end{figure}
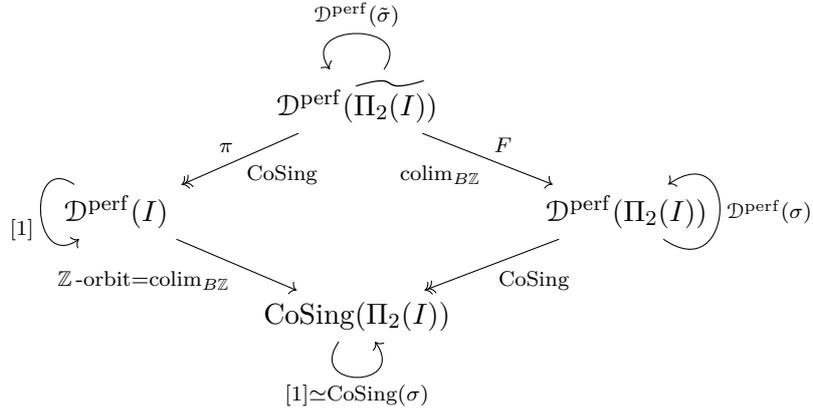

Towards the proof of \Cref{prop:derivedcategoryasCoSing}, we first identify the finite objects in $\D^{\on{perf}}(\widetilde{\Pi_2(\ADE)})$. 

\begin{definition}
An object $X\in \D^{\on{perf}}(\widetilde{\Pi_2(\ADE)})$ is called op-finite if 
\[ \on{Mor}_{\D(\widetilde{\Pi_2(\ADE)})}(X,Y)\in \D^{\on{perf}}(k)\] 
for all $Y\in \widetilde{\Pi_2(\ADE)}$. We denote by $\D^{\on{op-fin}}(\widetilde{\Pi_2(\ADE)})\subset \D^{\on{perf}}(\widetilde{\Pi_2(\ADE)})$ the full subcategory consisting of op-finite objects.

We similarly define the full subcategory $\D^{\on{op-fin}}(\Pi_2(\ADE))\subset \D^{\on{perf}}(\Pi_2(\ADE))$. 
\end{definition}

\begin{lemma}
There is an equality of full subcategories 
\[ \D^{\on{fin}}(\Pi_2(\ADE))=\D^{\on{op-fin}}(\Pi_2(\ADE))\subset \D^{\on{perf}}(\Pi_2(\ADE))\,.\]
\end{lemma}

\begin{proof}
The smooth $k$-linear $\infty$-category $\D(\Pi_2(\ADE))$ is left $2$-Calabi--Yau, meaning that $\on{id}_{\D(\Pi_2(\ADE))}^!\simeq [-2]$. For $X\in \D^{\on{fin}}(\Pi_2(\ADE))$ and $Y\in \D^{\on{perf}}(\Pi_2(\ADE))$ there is thus an equivalence
\[
\on{Mor}_{\D(\Pi_2(\ADE))}(Y,X)\simeq \on{Mor}_{\D(\Pi_2(\ADE))}(X,Y)^*[-2]\,,
\]
see \cite[Lem.~2.25]{Chr23}. This shows that $\D^{\on{fin}}(\Pi_2(\ADE))\subset \D^{\on{op-fin}}(\Pi_2(\ADE))$. The converse inclusion follows from the anti self-equivalence $\D^{\on{perf}}(\Pi_2(\ADE))^{\on{op}}\simeq \D^{\on{perf}}(\Pi_2(\ADE))$.
\end{proof}

\begin{lemma}
An object $X\in \D^{\on{perf}}(\widetilde{\Pi_2(\ADE)})$ lies in $\D^{\on{op-fin}}(\widetilde{\Pi_2(\ADE)})$ if and only if $F(X)\in \D^{\on{perf}}(\Pi_2(\ADE))$ lies in $\D^{\on{op-fin}}(\Pi_2(\ADE))$.
\end{lemma}

\begin{proof}
A description of the right adjoint $G\colon \D(\Pi_2(\ADE))\to \D(\widetilde{\Pi_2(\ADE)})$ of $\on{Ind}(F)$ follows directly from \cite[Section 2.2]{Chr25}, it satisfies $G(\Pi_2(\ADE))\simeq \widetilde{\Pi_2(\ADE)}=\prod_{Y\in \widetilde{\Pi_2(\ADE)}} Y$. The statement thus follows from the adjunction equivalence
\[
\on{Mor}_{\D(\Pi_2(\ADE))}(F(X),\Pi_2(\ADE))\simeq \on{Mor}_{\D(\widetilde{\Pi_2(\ADE)})}(X,G(\Pi_2(\ADE)))\,.
\]
\end{proof}

\begin{construction}
Let $(i,x)\in \widetilde{\Pi_2(\ADE)}$. Suppose first that $x\in \ADE$ is not trivalent. We define $C_{i,x}$ as the fiber totalization of the square 
\begin{equation}\label{eq:totsq1}
\begin{tikzcd}
{(i,x)} \arrow[r, "{a_{i,x}}"] \arrow[d, "{a_{i,x-1}^\dagger}"'] & {(i,x+1)} \arrow[d, "{a^\dagger_{i,x}}"] \\
{(i+1,x-1)} \arrow[r, "{a_{i+1,x-1}}"]                         & {(i+1,x)}                             
\end{tikzcd}
\end{equation}
where we set $(i,x+1)=0$ if $x=n$ and $(i+1,x-1)$ if $x=1$. The commutativity is expressed by $l_{i,x}$. We similarly define $S_{i+1,x}\simeq C_{i,x}[2]$ as the cofiber totalization of the square.

Suppose now that $x\in \ADE$ is trivalent, with incoming arrow $a_{x-1}\colon x-1\to x$ and outgoing arrows $a_{x}\colon x\to x+1$ and $b\colon x\to n$. We then define $C_{i,x}$ as the fiber totalization of the square
\begin{equation}\label{eq:totsq2}
\begin{tikzcd}[column sep=45]
{(i,x)} \arrow[r, "{(a_{i,x},b_i)}"] \arrow[d, "{a_{i,x-1}^\dagger}"] & {(i,x+1)\oplus (i,n)}  \arrow[d, "{(a^\dagger_{i,x},b_i^\dagger)}"] \\
(i+1,x-1) \arrow[r, "{a_{i+1,x-1}}"]                                                   & {(i+1,x)}                                                                            
\end{tikzcd}
\end{equation}
whose commutativity is expressed by $l_{i,x}$. We again define $S_{i+1,x}\simeq C_{i,x}[2]$ as the cofiber totalization of the square.
\end{construction}

\begin{lemma}
Let $i,j\in \mathbb{Z}$ and let $1\leq x,y\leq n$ be vertices of $\ADE$. There are equivalences in $\D(k)$
\[
\on{Mor}_{\D(\widetilde{\Pi_2(\ADE)})}(C_{i,x},(j,y))\simeq \begin{cases} k & i=j, x=y\\ 0 & \text{else}\end{cases}
\]
and
\[
\on{Mor}_{\D(\widetilde{\Pi_2(\ADE)})}((j,y),S_{i,x})\simeq \begin{cases} k & i=j, x=y\\ 0 & \text{else}\end{cases}
\]
\end{lemma}

\begin{proof}
We only prove the former equivalence, the proof of the latter is analogous.

The mapping chain complex in the dg category $\widetilde{\Pi_2(\ADE)}$ counts paths in the quiver underlying $\widetilde{\Pi_2(\ADE)}$. Furthermore, for $(i',x')\in \widetilde{\Pi_2(\ADE)}$, there is an equivalence in $\D(k)$ 
\[ \on{Mor}_{\D(\widetilde{\Pi_2(\ADE)})}((i',x'),(j,y))\simeq \on{Map}_{\widetilde{\Pi_2(\ADE)}}((i',x'),(j,y))\,.\] 

In the case that $j<i$ or $i=j$ and that there are no paths in $\ADE$ from $x$ to $y$, we thus have $\on{Mor}_{\D(\widetilde{\Pi_2(\ADE)})}(C_{i,x},(j,y))\simeq 0$. 

In the case $i=j,x=y$, the equivalence follows from 
\[ \on{Map}_{\widetilde{\Pi_2(\ADE)}}((j,y),(j,y))\simeq k\text{ and }\on{Map}_{\widetilde{\Pi_2(\ADE)}}((i',x'),(j,y))\simeq 0\] 
for $(i',x')=(i,x+1),(i+1,x-1),(i+1,x)$ and if $x$ is trivalent also $(i',x')=(i,n)$. 

Suppose thus that $i>j$ or that $i=j$ and that there are paths from $x\to y$. Applying the exact functor $\on{Mor}_{\D(\widetilde{\Pi_2(\ADE)})}(\mhyphen,(j,y))$ to the square \eqref{eq:totsq1} or \eqref{eq:totsq2} yields a square whose cofiber totalization is equivalent to  $\on{Mor}_{\D(\widetilde{\Pi_2(\ADE)})}(C_{i,x},(j,y))$. The terms in the square can be computed by a simple, though somewhat lengthy, case distinction which we leave to the reader. For instance, in the simplest case that $\ADE=A_n$, all entries in the square are equivalent to $k$ if $j>i$ and either exactly two or all four entries do not vanish and are equivalent to $k$ if $j=i$. In each case, the cofiber totalization of the square vanishes.
\end{proof}

By stable generators of a stable $\infty$-category $\C$, we mean a collection of objects $\mathcal{X}\subset \C$, such that the smallest stable subcategory of $\C$ containing $\mathcal{X}$ is given by $\C$. This is the case if and only if the objects in $\C$ are generated from the additive hull of $\mathcal{X}$ by forming iterated fibers and cofibers. 

\begin{lemma}\label{lem:simplegenerators}~
\begin{enumerate}[(1)]
\item The objects $\{C_{i,x}\}_{(i,x)\in \widetilde{\Pi_2(\ADE)}}$ stably generate $\D^{\on{op-fin}}(\widetilde{\Pi_2(\ADE)})$.
\item The objects $ \{S_{i,x}\}_{(i,x)\in \widetilde{\Pi_2(\ADE)}}$ stably generate $\D^{\on{fin}}(\widetilde{\Pi_2(\ADE)})$. 
\item The two full subcategories $\D^{\on{fin}}(\widetilde{\Pi_2(\ADE)}),\D^{\on{op-fin}}(\widetilde{\Pi_2(\ADE)})\subset \D^{\on{perf}}(\widetilde{\Pi_2(\ADE)})$ coincide.
\end{enumerate} 
\end{lemma}

\begin{proof}
Part (3) immediately follows from parts (1) and (2) since the collections of stable generators only differ only by suspensions. 

We only prove part (1), the proof of part (2) is analogous. Let $X\in \D^{\on{op-fin}}(\widetilde{\Pi_2(\ADE)})$. We prove that $X$ is in the stable hull of $\{C_{i,x}\}_{(i,x)}$ by an induction over the dimension of the morphisms objects $\prod_{(i,x)\in \widetilde{\Pi_2(\ADE)})}\on{Mor}_{\D(\widetilde{\Pi_2(\ADE)})}(X,(i,x))$. In the case that the dimension is $1$, we find that $X\simeq C_{i,x}[j]$ for some $(i,x)\in \widetilde{\Pi_2(\ADE)}$ and $j\in \mathbb{Z}$.

We proceed with the induction step. The directedness of the arrows in $\widetilde{\Pi_2(\ADE)}$ equips the objects of $\widetilde{\Pi_2(\ADE)}$ with a poset structure and there exists a maximal $(j,y)\in \widetilde{\Pi_2(\ADE)}$ such that $\on{Mor}_{\D(\widetilde{\Pi_2(\ADE)})}(X,(j,y))\not \simeq 0$. We choose a non-zero morphism $\beta\colon X\to (j,x)$. Note that $\beta$ factors through the morphism $C_{j,x}\to (j,x)$ with a morphism $\beta'\colon X\to C_{j,x}$ due to the assumption that $\on{Mor}_{\D(\widetilde{\Pi_2(\ADE)})}(X,(j',y'))\simeq 0$ for $(j',y')>(j,y)$.  We define $\tilde{X}=\on{fib}(\beta')$. For each $(i,x)\in \widetilde{\Pi_2(\ADE)})$, there is a fiber and cofiber sequence
\[
\on{Mor}_{\D(\widetilde{\Pi_2(\ADE)})}(\tilde{X},(i,x))\longrightarrow 
\on{Mor}_{\D(\widetilde{\Pi_2(\ADE)})}(X,(i,x)) \longrightarrow \on{Mor}_{\D(\widetilde{\Pi_2(\ADE)})}(C_{(j,x)},(i,x)) \,.
\]
The latter term vanishes unless $(j,x)=(i,x)$ and thus the dimension of the left term is one less than the dimension of the middle term. 
\end{proof}

\begin{proposition}\label{prop:derivedcategoryasCoSing}
The functor $\pi\colon \D^{\on{perf}}(\widetilde{\Pi_2(\ADE)})\to \D^{\on{perf}}(\ADE)$  induces an equivalence of $\infty$-categories $\on{CoSing}(\widetilde{\Pi_2(\ADE)})\simeq \D^{\on{perf}}(\ADE)$.  
\end{proposition}

\begin{proof} 
Auslander--Reiten translation yields fiber and cofiber sequence in $\D^{\on{perf}}(\ADE)$. For any $(i,x)\in \widetilde{\Pi_2(\ADE)}$, these express the objects $S_{i,x}$ as in the kernel of $\pi$. By \Cref{lem:simplegenerators}, it follows that the kernel of $\pi$ contains $\D^{\on{fin}}(\widetilde{\Pi_2(\ADE)})$. There is thus an induced $\mathbb{Z}$-equivariant functor $\on{CoSing}(\widetilde{\Pi_2(\ADE)})\to \D^{\on{perf}}(\ADE)$. Passing to $\mathbb{Z}$-quotients, this functor induces the known equivalence 
\[ \on{CoSing}(\Pi_2(\ADE))\simeq \on{colim}_{B\mathbb{Z}} \D^{\on{perf}}(\ADE)= \D^{\on{perf}}(\ADE)/U[-1]\,.\]
That $\pi$ is an equivalence thus follows from \Cref{lem:quotequiv}.
\end{proof}

\begin{lemma}\label{lem:quotequiv}
Let $H$ be a group and let $\C,\D$ be stable presentable $\infty$-categories with an $H$-action. Let $\alpha\colon \C\to \D$ be an $H$-equivariant colimit preserving functor. If $\alpha$ induces an equivalence 
\[
\alpha_H=\on{colim}_{BH}(\alpha)\colon \C_H=\on{colim}_{BH} \C\xlongrightarrow{\simeq} \D_H=\on{colim}_{BH}\D
\]
on the group quotient $\infty$-categories, then $\alpha$ is already an equivalence of $\infty$-categories.
\end{lemma}

\begin{proof}
There are commutative diagrams as follows, where $G_\C,G_\D$ are the right adjoints of the functors $F_\C,F_\D$ to the colimit. This follows from the fact that $F_\C,F_\D$ arise from tensoring $\C,\D$ with a functor $\on{Sp}^{\amalg H}\to \on{Sp}$ and $G_\C,G_\D$ arise from tensoring $\C,\D$ with the right adjoint $\on{Sp}\to \on{Sp}^{\amalg H}$, see \cite[Lem.~2.10, Lem.~2.4]{Chr25} for details.
\[
\begin{tikzcd}
\C \arrow[r, "\alpha"] \arrow[d, "F_\C"']      & \D \arrow[d, "F_\D"] \\
\C_H \arrow[r, "\alpha_{H}"] & \C_{H}     
\end{tikzcd}
\quad\quad\quad
\begin{tikzcd}
\C \arrow[r, "\alpha"]                                           & \D                                \\
\C_H \arrow[r, "\alpha_{H}"] \arrow[u, "G_\C"] & \C_{H} \arrow[u, "G_\D"']
\end{tikzcd}
\]
Let $X,X'\in \C$. Then there is a commutative diagram
\[
\begin{tikzcd}
{\on{Mor}_\C(X,X')} \arrow[r, hook] \arrow[d, "\alpha"'] \arrow[rr, "F_\C", bend left=10] & {\prod_{h\in H} \on{Mor}_\C(X,h.X')} \arrow[r, "\simeq"] \arrow[d, "\simeq"] & {\on{Mor}_{\C_H}(F_\C(X),F_\C(X'))} \arrow[d, "\alpha_{H}"] \arrow[d, "\simeq"'] \\
{\on{Mor}_\D(\alpha(X),\alpha(X'))} \arrow[rr, "F_\D"', bend right=10] \arrow[r, hook]     & {\prod_{h\in H} \on{Mor}_\D(\alpha(X),h.\alpha(X'))} \arrow[r, "\simeq"]     & {\on{Mor}_{\D_H}(F_\D\alpha(X),F_\D\alpha(X'))}                                          
\end{tikzcd}
\]
where $h.$ denotes the action of $h\in H$. The horizontal equivalences arise from equivalences $G_\C F_\C\simeq \prod_{h\in H}h.(\mhyphen)$, see \cite[Lem.~2.11]{Chr25}. The equivalence $\alpha\circ G_\C F_\C \simeq G_\D F_\D \circ \alpha$ is compatible with these decompositions. This shows that the left vertical morphism is an equivalence as a direct summand of the middle vertical morphism. We thus conclude that $\alpha$ is fully faithful.

It remains to show that the image of $\alpha$ generates $\D$ under colimits. Let $X\in \D$. Then $X$ is a direct summand:
\[ X\subset \coprod_{h\in H} h.X \simeq G_\D F_\D(X)\simeq \alpha(G_\C(\alpha_{H}^{-1}(F_\D(X))\,.\] 
\end{proof}

\subsection{Fiber and cofiber sequences from squares in the Auslander--Reiten quiver}\label{subsec:ARsequences}

It is well known that the mesh relations in the Auslander--Reiten quiver give rise to distinguished triangles in the triangulated category $D^{\on{perf}}(\ADE)\simeq \on{ho}\D^{\on{perf}}(\ADE)$ and thus to fiber and cofiber sequences in $\D^{\on{perf}}(\ADE)$. We can restate this fact as follows:

We label the vertices of the Auslander--Reiten quiver of $D^{\on{perf}}(\ADE)$ by pairs $(x,i)$ with $x\in \ADE_0$ and $i\in \mathbb{Z}$. The arrows in the Auslander--Reiten quiver can further be labeled in the same way as the degree $0$ arrows of $\widetilde{\Pi_2(\ADE)}$.

\begin{proposition}
Let $\ADE$ be a Dynkin quiver (oriented as in \Cref{def:ADEinvolution}) and $x\in \ADE_0$ a vertex and $i\in \mathbb{Z}$.
\begin{enumerate}[(1)]
\item If $x$ is $2$-valent, the square in $\D^{\on{perf}}(\ADE)$ 
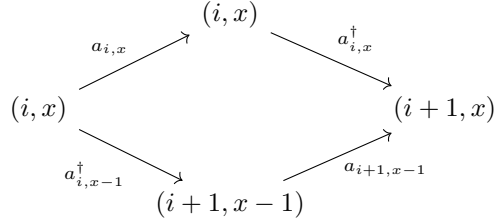
\[
\begin{tikzcd}
                                                             & {(i,x)} \arrow[rd, "{a_{i,x}^\dagger}"]      &           \\
{(i,x)} \arrow[ru, "{a_{i,x}}"] \arrow[rd, "{a_{i,x-1}^\dagger}"'] &                                          & {(i+1,x)} \\
                                                             & {(i+1,x-1)} \arrow[ru, "{a_{i+1,x-1}}"'] &          
\end{tikzcd}
\]
appearing in the Auslander--Reiten quiver is biCartesian. 

If $x$ is $1$-valent, the square is also biCartesian, when setting $(i,x+1)=0$ if $x=n$ and $(i+1,x-1)$ if $x=1$.
\item Suppose that $x$ is $3$-valent with incoming arrow $a_{x-1}\colon x-1\to x$ and outgoing arrows $a_{x}\colon x\to x+1$ and $b\colon x\to n$. Then the square in $\D^{\on{perf}}(\ADE)$
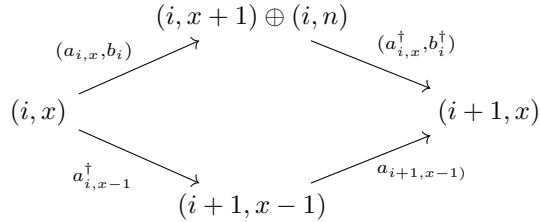
\[
\begin{tikzcd}
                                                                   & {(i,x+1)\oplus (i,n)} \arrow[rd, "{(a^\dagger_{i,x},b_{i}^\dagger)}"] &           \\
{(i,x)} \arrow[ru, "{(a_{i,x},b_i)}"] \arrow[rd, "{a_{i,x-1}^\dagger}"'] &                                                         & {(i+1,x)} \\
                                                                   & {(i+1,x-1)} \arrow[ru, "{a_{i+1,x-1)}}"']               &          
\end{tikzcd}
\]
is biCartesian.
\end{enumerate}
\end{proposition}

\begin{example}\label{ex:A3AR}
In type $A_3$, there is a commutative diagram as follows in $\D(\widetilde{\Pi_2(A_3)})$:
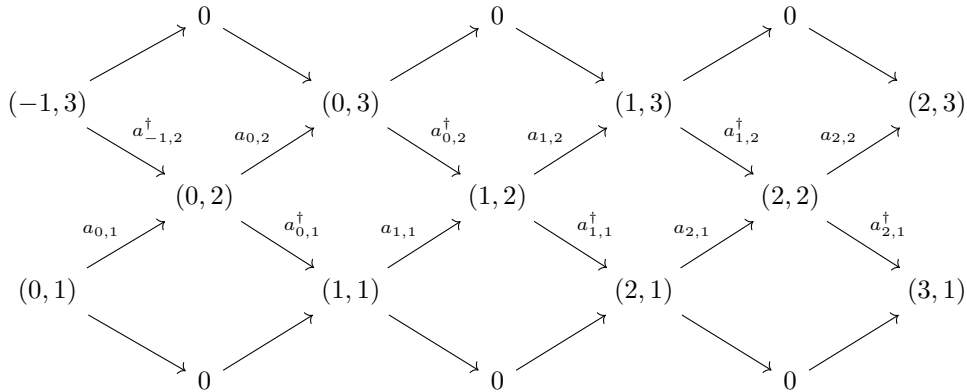
\[
\begin{tikzcd}
                                               & 0 \arrow[rd]                                          &                                            & 0 \arrow[rd]                                          &                                            & 0 \arrow[rd]                                          &         \\
{(-1,3)} \arrow[rd, "a_{-1,2}^\dagger"] \arrow[ru] &                                                       & {(0,3)} \arrow[ru] \arrow[rd, "a_{0,2}^\dagger"] &                                                       & {(1,3)} \arrow[ru] \arrow[rd, "a_{1,2}^\dagger"] &                                                       & {(2,3)} \\
                                               & {(0,2)} \arrow[ru, "a_{0,2}"] \arrow[rd, "a_{0,1}^\dagger"] &                                            & {(1,2)} \arrow[ru, "a_{1,2}"] \arrow[rd, "a_{1,1}^\dagger"] &                                            & {(2,2)} \arrow[ru, "a_{2,2}"] \arrow[rd, "a_{2,1}^\dagger"] &         \\
{(0,1)} \arrow[ru, "a_{0,1}"] \arrow[rd]       &                                                       & {(1,1)} \arrow[ru, "a_{1,1}"] \arrow[rd]   &                                                       & {(2,1)} \arrow[ru, "a_{2,1}"] \arrow[rd]   &                                                       & {(3,1)} \\
                                               & 0 \arrow[ru]                                          &                                            & 0 \arrow[ru]                                          &                                            & 0 \arrow[ru]                                          &        
\end{tikzcd}
\]
Each square in the diagram has the property that its totalization lies in $\D^{\on{fin}}(\widetilde{\Pi_2(A_3)})$. Hence, the image of each square in $\D^{\on{perf}}(\ADE)\simeq \on{CoSing}(\widetilde{\Pi_2(A_3)})$ is biCartesian.
\end{example}

\begin{corollary}\label{cor:Anrectangles}
Let $\ADE$ be of type $A$ and consider the diagram in $\D^{\on{perf}}(\ADE)$ obtained by adding a rows of $0$'s above and below the Auslander--Reiten quiver (as in \Cref{ex:A3AR}). Then all rectangles in the diagram (i.e.~those arising by composing squares) are biCartesian. 
\end{corollary}

\begin{proof}
This is a direct consequence of the pasting laws for pushouts and pullbacks. 
\end{proof}

The following description of the objects of $\C_\ADE$ is well known, since the additive $1$-category $\on{ho}(\C_\ADE)$ is equivalent to the $1$-category of projective $H_0(\Pi_2(\ADE))$-modules, see for instance \cite[Thm.~9.3.4]{Ami09}.

\begin{lemma}\label{lem:indecomposables}
Let $\ADE$ be a Dynkin quiver with $n$ vertices. Then $\C_\ADE$ has $n$ indecomposable objects up to equivalence. These arise as the images of the projective $k\ADE$ modules under the functor $\D^{\on{perf}}(\ADE)\to \D^{\on{perf}}(\ADE)/U[-1]\simeq \C_\ADE$. 
\end{lemma}

\begin{notation}
As justified by \Cref{lem:indecomposables}, we label the indecomposable objects of $\C_\ADE$ by the integers $1,\dots,n\in \ADE_0$. 
\end{notation}

\begin{remark}\label{rem:genmorphisms}
The functor $\widetilde{\Pi_2(\ADE)}$ maps each object $(i,x)$, with $i\in \mathbb{Z}$ and $x\in \ADE_0$, to $x\in \C_\ADE$. The degree $0$ morphisms in $\C_\ADE$ are generated by the morphisms in $\Pi_2(\ADE)$ labeled $a$ and $a^\dagger$, $a\in \ADE_1$, and we can label the morphisms in $\C_\ADE$ between the indecomposables by their unique lifts to $\Pi_2(\ADE)$.
\end{remark}

\begin{proof}[Proof of \Cref{lem:indecomposables}]
Consider the additive closure $\mathcal{X}=\on{Add}(\{1,\dots,n\})$. Clearly, every morphism in $\mathcal{X}$ admits a lift to $\D^{\on{perf}}(\ADE)$. The fiber or cofiber of any morphism in $\mathcal{X}$ thus lies in the image of the exact functor $\D^{\on{perf}}(\ADE)\to \C_\ADE$, which is given exactly by $\mathcal{X}$. This shows that $\mathcal{X}$ is a stable subcategory. Since each of the objects $1,\dots,n$ is indecomposable, it follows that $\mathcal{X}$ is idempotent complete. Since $1\oplus \dots\oplus n$ is a compact generator of $\C_\ADE$, it follows that $\mathcal{X}$ coincides with $\C_\ADE$. The indecomposable objects in $\C_\ADE$ are thus exactly $1,\dots,n$. 
\end{proof}

\begin{example}
In the setting of \Cref{ex:A3AR}, the image of the rectangle in $\widetilde{\Pi_2(A_3)}$
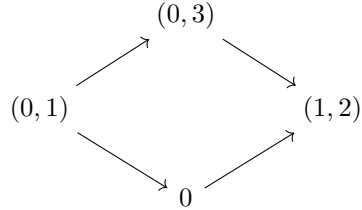
\[
\begin{tikzcd}
                               & {(0,3)} \arrow[rd]       &         \\
{(0,1)} \arrow[ru] \arrow[rd] &                    & {(1,2)} \\
                               & {0} \arrow[ru] &        
\end{tikzcd}
\]
in $\D^{\on{perf}}(A_3)$ is biCartesian. It in turn yields a fiber and cofiber sequence $1\to 3\to 2$ in $\C_{A_3}$.

Similarly, the rectangle
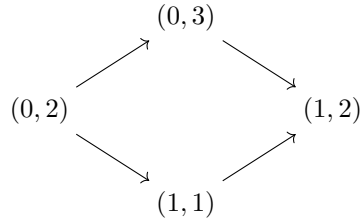
\[
\begin{tikzcd}
                               & {(0,3)} \arrow[rd]       &         \\
{(0,2)} \arrow[ru] \arrow[rd] &                    & {(1,2)} \\
                               & {(1,1)} \arrow[ru] &        
\end{tikzcd}
\]
gives rise to a fiber and cofiber sequence $2\to 1\oplus 3\to 2$ in $\C_{A_3}$. 
\end{example}

\begin{definition}\label{def:length}
Let $n\geq 1$. Given a non-invertible morphism $(i,a)\to (j,b)\in \widetilde{\Pi_2(A_n)}$, we say that it has length $l\geq 1$ if it is given by a $k$-linear sum of composites of at least $l$ generating morphisms from \Cref{rem:genmorphisms}.
\end{definition}

For $\C_{A_n}$, \Cref{cor:Anrectangles} implies the following:

\begin{proposition}~
\begin{enumerate}[(1)]
\item Let $1\leq i<j\leq n$. The cofiber of the morphism $a_{j-1}\dots a_i\colon i\to j$ in $\C_{A_n}$ is equivalent to $j-i\in \C_{A_n}$. 
\item Let $1\leq j<i\leq n$. The cofiber of the morphism $a_{j}^\dagger\dots a_{i-1}^\dagger\colon i\to j$ in $\C_{A_n}$ is equivalent to $(i-j)[1]=n+1+j-i\in \C_{A_n}$. 
\end{enumerate}
Note that the above morphisms $i\to j$ are of minimal length $|j-i|$. 
\begin{enumerate}
\item[(3)] Let $i\to j\to k$ be a fiber and cofiber sequence in $\C_{A_n}$ consisting of morphisms of minimal length between them. Then either $i+j[1]+k=n+1$ or $i+j[1]+k=2n+2$. 
\end{enumerate}
\end{proposition}

\begin{proof}
Parts (1) and (2) follow directly from \Cref{cor:Anrectangles}. For part (3), note that $j[1]=n+1-j$. Thus if $i<j$, then $k=j-i$ and $i+j[1]+k=i+n+1-j+(j-i)=n+1$. If $j>i$, then $k=n+1+j-i$ and thus $i+j[1]+k=2(n+1)$. 
\end{proof}

\subsection{Description of the suspension functor}\label{subsec:involution}

We fix a Dynkin quiver $\ADE$. The goal of this section is to prove the description given in \Cref{prop:shift} of the suspension functor of $\C_\ADE=\on{CoSing}(\Pi_2(\ADE))$ in terms of the involution $\sigma$ of the dg category $\Pi_2(\ADE)$ from \Cref{def:ADEinvolution}.  We do this by showing the compatibility of the automorphisms in the commutative square in \Cref{fig:squareofautomorphisms}.


\begin{definition}
We define the dg isomorphism $\tilde{\sigma}\colon \widetilde{\Pi_2(\ADE)} \simeq \widetilde{\Pi_2(\ADE)}$.

In type $\ADE=A_n$, labeling the arrows as in \Cref{def:ADEinvolution}, we set 
\[
\tilde{\sigma}((i,x))=(i+x,\sigma(x))\,,
\]
where $\sigma(x)=n+1-x$ as in \Cref{def:ADEinvolution}, and for all $1\leq j\leq n-1$, $1\leq x\leq n$, and $i\in \mathbb{Z}$
\begin{align*}
\tilde{\sigma}(a_{i,j})&=a_{i,j}^\dagger\\
\tilde{\sigma}(a_{i,j}^\dagger)&=a_{i+j+1,j}\\
\tilde{\sigma}(l_{i,x})& =-l_{i+x,\sigma(x)}\,.
\end{align*}

In type $\ADE=D_n$, we set 
\[
\tilde{\sigma}((i,x))=(i+n-1,\sigma(x))\,.
\]
For $n$ even, we set for $1\leq j\leq n-1$, $1\leq x\leq n$ and $i\in \mathbb{Z}$ 
\begin{align*} 
\tilde{\sigma}(a_{i,j})&= a_{i+n-1,j}\\
\tilde{\sigma}(a_{i,j}^\dagger)&= a_{i+n-1,j}^\dagger\\
\tilde{\sigma}(l_{i,x})&= l_{i+n-1,x}\,.\\
\end{align*}
For $n$ odd, we set for $1\leq j\leq n-1$, $1\leq x\leq n$ and $i\in \mathbb{Z}$ 
\begin{align*} 
\tilde{\sigma}(a_{i,j})&=\begin{cases} a_{i+n-1,j}& j\not = n-2,n-1\\  a_{i+n-1,n-1} & j=n-2\\ a_{i+n-1,n-2}& j=n-1 \end{cases}\\
\tilde{\sigma}(a_{i,j}^\dagger)&=\begin{cases} a_{i+n-1,j}^\dagger& j\not = n-2,n-1\\  a^\dagger_{i+n-1,n-1}& j=n-2\\ a_{i+n-1,n-2}^\dagger& j=n-1 \end{cases}\\
\tilde{\sigma}(l_{i,x})&= l_{i+n-1,\sigma(x)}\,.\\
\end{align*}

In type $\ADE=E_6$, we set
\[
\tilde{\sigma}((i,x))=\begin{cases} (i+3+x,\sigma(x)) & i\not=6 \\ (i+6,x) &i=6\,. \end{cases}
\]
We further set for $1\leq j\leq 5$, $1\leq x\leq 6$ and $i \in \mathbb{Z}$
\begin{align*}
\tilde{\sigma}(a_{i,j})&=\begin{cases} a_{i+3+j,5-j}^\dagger & j\not = 5\\ a_{i+6,5} & j=5\end{cases}\\
\tilde{\sigma}(a_{i,j}^\dagger)&=\begin{cases} a_{i+3+j+1,5-j} & j\not = 5\\ -a_{i+6,5}^\dagger & j=5\end{cases}\\
\tilde{\sigma}(l_{i,x})&=-l_{\tilde{\sigma}((i,x))}\,.
\end{align*}

In type $\ADE=E_7$, we have $\sigma(x)=x$ for all vertices $x$ and we set
\[
\tilde{\sigma}((i,x))=(i+9,x)\,.
\]
For $x$ a vertex and $a$ an arrow in $E_7$ and $i\in \mathbb{Z}$, we set 
\begin{align*}
\tilde{\sigma}(a_i)&=a_{i+9}\\
\tilde{\sigma}(a_i^\dagger)&=a_{i+9}^\dagger\\
\tilde{\sigma}(l_{i,x})&=l_{i+9,x}\,.
\end{align*}

In type $\ADE=E_8$, we also have $\sigma(x)=x$ for all vertices $x$ and set 
\[
\tilde{\sigma}((i,x))=(i+15,x)\,.
\]
For $x$ a vertex and $a$ an arrow in $E_8$ and $i\in \mathbb{Z}$, we set 
\begin{align*}
\tilde{\sigma}(a_i)&=a_{i+15}\\
\tilde{\sigma}(a_i^\dagger)&=a_{i+15}^\dagger\\
\tilde{\sigma}(l_{i,x})&=l_{i+15,x}\,.
\end{align*}
\end{definition}


Equipping the endofunctor $\tilde{\sigma}$ with a $\mathbb{Z}$-equivariant structure amounts to specifying a natural transformation $\tilde{\sigma}\circ T \simeq T\circ \tilde{\sigma}$. Since these two dg functors strictly commute, two natural choices are the trivial identification $\tilde{\sigma}\circ T = T\circ \tilde{\sigma}$ and the negative of the trivial identification. We choose the trivial identifications in types $\ADE=A_n,E_6$ and the negative of the trivial identification in types $\ADE=D_n,E_7,E_8$, and in the following understand $\tilde{\sigma}$ as a $\mathbb{Z}$-equivariant dg functor.


\begin{lemma}\label{lem:tilde_sigma_induces_sigma}
The automorphism of the dg orbit category 
\[ \Pi_2(\ADE)\simeq \widetilde{\Pi_2(\ADE)}/T\] 
induced by the $\mathbb{Z}$-equivariant dg functor $\tilde{\sigma}$ is given by the involution $\sigma$.
\end{lemma}

\begin{proof}
We denote by $\sigma'\colon \Pi_2(\ADE)\simeq \widetilde{\Pi_2(\ADE)}/T \to \widetilde{\Pi_2(\ADE)}/T \simeq \Pi_2(\ADE)$ the dg functor induced by $\tilde{\sigma}$. 
The (sign) rules for determining $\sigma'$ from $\tilde{\sigma}$ can be found for instance in \cite[Section 2.1]{Kel05b}. 

In the types $\ADE=A_n,E_6$, where $\tilde{\sigma}$ is equipped with the trivial equivariant structure, the identification $\sigma'= \sigma$ is immediate. 

In the types $\ADE=D_n,E_7,E_8$, $\tilde{\sigma}$ is equipped with the negative trivial equivariant structure. Let $a\colon x\to y$ in $\ADE$. Then $a_0^\dagger\colon (0,y)\to (1,x)=T(0,x)$ and $F(a_0^\dagger)=a^\dagger\in \Pi_2(\ADE)$. The dg functor $\sigma'$ maps $a^\dagger$ to the composite
\[
F\circ \tilde{\sigma}(0,y)\xrightarrow{F\circ \tilde{\sigma}(a_0^\dagger)}F\circ \tilde{\sigma}(1,x)=F\circ \tilde{\sigma}\circ T(0,x)\xrightarrow{-\on{id}}F\circ T\circ \tilde{\sigma}(0,x)
\]
which amounts to $-a^\dagger$. Similar computations show that $\sigma'(a)=F(\tilde{\sigma}(a_0))$ and $\sigma'(l_i)=-l_{\sigma(i)}$, and thus $\sigma'\simeq \sigma$. 
\end{proof}

\begin{remark}\label{rem:equivariantstructureonshift}
\begin{enumerate}[(1)]
\item The $\mathbb{Z}$-action arising from the autoequivalence $\D^{\on{perf}}(T)$ of $\D^{\on{perf}}(\widetilde{\Pi_2(\ADE)})$ descends to a $\mathbb{Z}$-action on the cosingularity category $\D^{\on{perf}}(\ADE)\simeq \on{CoSing}(\widetilde{\Pi_2(\ADE)})$.  Furthermore, since $\on{CoSing}(T)\simeq U[-1]$, with $U$ the Serre functor, this $\mathbb{Z}$-action coincides with $\mathbb{Z}$-action considered in \Cref{prop:1-cluster_orbit}.
\item The suspension functor $[1]\colon \D^{\on{perf}}(\ADE)\to \D^{\on{perf}}(\ADE)$ has a canonical $\mathbb{Z}$-equivariant structure:\\
For any exact functor $f$ between stable $\infty$-categories there is a canonical equivalence $[1]\circ f\simeq f\circ [1]$ arising from applying $f$ to the fiber and cofiber sequence of endofunctors $\on{id}\to 0 \to [1]$. Choosing $f=\on{CoSing}(T)$, this induces the $\mathbb{Z}$-equivariant structure on $[1]$.\\
With its canonical $\mathbb{Z}$-equivariant structure, passing to the colimit over $B\mathbb{Z}$, the suspension of $\D^{\on{perf}}(\ADE)$ induces the suspension functor of $\D^{\on{perf}}(\ADE)/U[-1]\simeq \C_\ADE$. 
\end{enumerate}
\end{remark}

\begin{lemma}\label{lem:Zequivariantequivalence}
Denote by $\on{CoSing}(\tilde{\sigma})$ the autoequivalence of $\D^{\on{perf}}(\ADE)$ induced by the autoequivalence $\D(\tilde{\sigma})$ of $\D(\widetilde{\Pi_2(\ADE)})$. There exists a $\mathbb{Z}$-equivariant natural equivalence $\on{CoSing}(\tilde{\sigma})\simeq [1]$.
\end{lemma}

\begin{proof}
The functor $\D^{\on{perf}}(\tilde{\sigma})$ clearly preserves finite modules, and hence descends to an autoequivalence of $\on{CoSing}(\widetilde{\Pi_2(\ADE)})\simeq \D^{\on{perf}}(\ADE)$. To abstractly conclude that this is the suspension functor, it suffices to observe that $\on{CoSing}(\D(\tilde{\sigma}))$ is determined by its action on the compact generator $k\ADE\in \D^{\on{perf}}(\ADE)$ (this including the action on its endomorphisms). As we can see by inspection of the Auslander--Reiten quiver: $\on{CoSing}(\tilde{\sigma})$ maps $k\ADE$ to $k\ADE[1]$ and the generating morphisms (the $a$'s) to the generating morphisms in $k\ADE[1]$ (without signs). 

The $0$-th Hochschild cohomology $\on{HH}^0(\D(\ADE))$ of $\D(\ADE)$ is given by the center $k$ of $k\ADE$. It describes the endomorphisms of the endofunctor $\on{id}_{\D(\ADE)}$, or equivalently the endomorpisms of any autoequivalence of $\D(\ADE)$. The $\mathbb{Z}$-equivariant structure on $\on{CoSing}(\tilde{\sigma})$ amounts to an equivalence $T\circ \tilde{\sigma}\simeq \tilde{\sigma}\circ T$. There thus exist exactly two (unless $\on{char}(k)=2$) $\mathbb{Z}$-equivariant structures on $\on{CoSing}(\tilde{\sigma})$, which give rise to an involution of $\C_\ADE=\on{CoSing}(\Pi_2(\ADE))$, which are either the trivial or the negative of the trivial $\mathbb{Z}$-equivariant structure. If $\on{char}(k)=2$, these coincide, so we assume that $\on{char}(k)\neq 2$ in the following. \\

{\bf Case 1: $\ADE=A_n$.} Suppose that the negative of the trivial equivariant structure on $\on{CoSing}(\tilde{\sigma})$ gives rise to the suspension functor $[1]$ on $\C_\ADE$. The biCartesian squares in the Auslander--Reiten quiver, see \Cref{cor:Anrectangles}, give by \cite[Lem.~1.1.2.13]{HA} rise to the distinguished triangles
\[
1\xrightarrow{a_1} 2\xrightarrow{a_1^\dagger} 1\xrightarrow{{\bf a}} n
\]
and 
\begin{equation}\label{eq:2ndtriangle}
n\xrightarrow{a_{n-1}^\dagger} n-1\xrightarrow{a_{n-1}} n\xrightarrow{-{\bf a}^\dagger} 1
\end{equation}
in the triangulated category $\on{ho}\D^{\on{perf}}(\ADE)$, and thus also in $\on{ho}\C_\ADE$, with ${\bf a}$ a composite of morphisms arising from morphisms in $\ADE$ and ${\bf a}^\dagger$ a composite of dual morphisms. A crucial point here is the minus sign in the second triangle, arising from the reversed orientation of the biCartesian squares. 
We also remark that we may (and do) chose the suspension functor of $\D^{\on{perf}}(\ADE)$ such that for these biCartesian squares, the pasted biCartesian square, expressing an equivalence $[1]((0,x))\simeq (1,x)$, with $x=1,n$, expresses the identity. 

Applying the functor induced by $\on{CoSing}(\tilde{\sigma})$ with the negative trivial equivariant structure to the first biCartesian square reverses the orientation of the biCartesian square (since $\on{CoSing}(\tilde{\sigma})$ acts as a glide reflection on the Auslander--Reiten quiver) and also reverses the signs of the dual morphisms $a^\dagger$ in $\on{ho}\C_\ADE$ due to the choice of equivariant structure. There is thus a further distinguished triangle
\[
n\xrightarrow{a_{n-1}^\dagger} n-1\xrightarrow{a_{n-1}} n\xrightarrow{{\bf a}^\dagger} 1
\]
in $\on{ho}\C_\ADE$. This contradicts the existence of the triangle \eqref{eq:2ndtriangle} by the axioms of a triangulated category. Hence it must be the trivial equivariant structure on $\on{CoSing}(\tilde{\sigma})$ which makes it $\mathbb{Z}$-equivariantly equivalent to the functor $[1]$ on $\D(\ADE)$.\\

{\bf Case 2: $\ADE=D_{2n},E_7,E_8$.} The autoequivalence of $\C_\ADE$ induced by $\on{CoSing}(\tilde{\sigma})$ wit the trivial equivariant structure is the identity. Note that applying three times the rotation axioms of a triangulated category, if $A\xrightarrow{f}B\xrightarrow{g}C\xrightarrow{h}A[1]$ is a distinguished triangle, then so is $A[1]\xrightarrow{f[1]}B[1]\xrightarrow{g[1]}C[1]\xrightarrow{-h[1]}A[2]$. Therefore, the identity cannot coincide with the suspension functor $[1]$ (since $\on{char}(k)\neq 2$). Thus, with the negative trivial equivariant structure, $\on{CoSing}(\tilde{\sigma})$ is $\mathbb{Z}$-equivariantly equivalent to $[1]$.\\

{\bf Case 3: $\ADE=D_{2n+1},E_6$.} In these types, there is an apparent involution $\sigma'$ of the quiver $\ADE$, with corresponding autoequivalence $\D^{\on{perf}}(\sigma')$ of $\D^{\on{perf}}(\ADE)$, with the property that $\on{CoSing}(\tilde{\sigma})\circ \D^{\on{perf}}(\sigma')$ is equivalent to a power of the translation $\on{CoSing}(T)$.  Note that $\on{ho}\D^{\on{perf}}(\sigma')$ is equivalent to the derived functor of the exact autoequivalence of the abelian module $1$-category $\on{mod}(\ADE)$ and thus describes a triangle functor of $\on{ho}\D^{\on{perf}}(\ADE)$ whose corresponding identification $\on{ho}\D^{\on{perf}}(\sigma')\circ [1]=[1]\circ \on{ho}\D^{\on{perf}}(\sigma')$ is trivial. Equivalently, this means that $\on{ho}\D^{\on{perf}}(\sigma')$ preserves distinguished triangles (without any appearing signs). Equipping $\D^{\on{perf}}(\sigma')$ with the trivial $\mathbb{Z}$-equivariant structure, we find that the induced autoequivalence of $\on{ho}\C_\ADE$ preserves the distinguished triangles in the image of $\on{ho}\D^{\on{perf}}(\ADE)\to \on{ho}\C_\ADE$. 
Any power of $\on{CoSing}(T)$, with the trivial $\mathbb{Z}$-equivariant structure, induces the identity on $\C_\ADE$. Thus, with the trivial equivariant structure, the functor $\D^{\on{perf}}(\tilde{\sigma})$ cannot induce the suspension $[1]$ on $\C_\ADE$. 
\end{proof}

\begin{proof}[Proof of \Cref{prop:shift}.]
Combine \Cref{rem:equivariantstructureonshift} and \Cref{lem:orbit_cat_induced_functor,lem:tilde_sigma_induces_sigma,lem:Zequivariantequivalence}.
\end{proof}

\section{\texorpdfstring{$3$}{3}-Calabi--Yau perverse schobers for higher Teichm\"uller theory} 

\subsection{Marked surfaces and ideal triangulations}

\begin{definition}
A marked surface consists of a compact oriented topological surface ${\bf S}$ together with a finite subset of marked points $M\subset \partial {\bf S}$ in the boundary, such that each boundary component contains at least one marked point. 
\end{definition}

\begin{definition}
Let ${\bf S}$ be a marked surface. 
\begin{enumerate}[(1)]
\item An arc in ${\bf S}$ consists of an embedded curve $[0,1]\to {\bf S}$ with endpoints in $\partial {\bf S}\backslash M$, considered up to homotopies relative $\partial {\bf S}\backslash M$. Arcs are not allowed to be contractible. An arc is called a boundary arc if it cuts out a monogon.  
\item Two arcs in ${\bf S}$ are called compatible if they have representatives that do not intersect each other.
\item An ideal triangulation of ${\bf S}$ consists of a maximal collection of compatible arcs. Note that every ideal triangulation contains all boundary arcs.
\end{enumerate}
\end{definition}

Given a graph $\rgraph$, we denote by $\on{Exit}(\rgraph)$ its exit path category, which is defined as (the nerve of) the $1$-category with objects the vertices and edges of $\rgraph$ and morphisms going from vertices to edges by incidence. The morphisms in $\on{Exit}(\rgraph)$ can thus be identified with the halfedges of $\rgraph$.

\begin{definition}
Given a graph $\rgraph$ and a marked surface ${\bf S}$, together with an embedding of the geometric realization $\iota\colon |\on{Exit}(\rgraph)|\subset {\bf S}\backslash M$, we call $\rgraph$ a spanning graph of ${\bf S}$ if
\begin{itemize}
\item the embedding $\iota$ is a homotopy equivalence and
\item $\iota$ restricts to a homotopy equivalence $\iota^{-1}(\partial {\bf S}\backslash M)\to \partial {\bf S}\backslash M$.
\end{itemize}
\end{definition}

A ribbon graph refers to a graph $\rgraph$ together with a cyclic orientation on the set of halfedges incident to each vertex of $\rgraph$. Any spanning graph of a marked surface inherits a ribbon graph structure, via the counterclockwise cyclic ordering of halfedges.

\begin{remark}
Ideal triangulation of a marked surface ${\bf S}$ are in bijection with trivalent spanning graph of ${\bf S}$ (considered up to homotopy), by passing to the dual graph of the triangulation.
\end{remark}

\subsection{Parametrized perverse schobers}

We briefly recall the notion of a perverse schober on a marked surface parametrized by a ribbon graph, see \cite[Section 3]{CHQ23}, \cite[Sections 3,4]{Chr22}. The notion of a perverse schober as a categorified perverse sheaf was proposed by Kapranov--Schechtman in \cite{KS14}.

For $n\in\mathbb{N}_{\geq 1}$, we let $\rgraph_{n}$ be the ribbon graph with a single vertex $v$ and $n$ incident external edges. We call $\rgraph_n$ the $n$-spider. 

\begin{definition}\label{def:schobernspider}
Let $n\geq 1$. A perverse schober on the $n$-spider consists of the following data: 
\begin{enumerate}
\item[(1)] If $n=1$, a spherical adjunction between stable $\infty$-categories
\[ F\colon \mathcal{V}\longleftrightarrow \mathcal{N}\cocolon G\,,\]
meaning an adjunction whose twist functor $T_{\mathcal{V}}=\on{cof}(\on{id}_{\mathcal{V}}\xrightarrow{\on{unit}}GF)\in \on{Fun}(\mathcal{V},\mathcal{V})$ and cotwist functor $T_{\mathcal{N}}=\on{fib}(FG\xrightarrow{\on{counit}}\on{id}_{\mathcal{N}})\in \on{Fun}(\mathcal{N},\mathcal{N})$ are autoequivalences. We also call the functors $F,G$ spherical functors, as in \cite{AL17}. 
\item[(2)] If $n\geq 2$, a collection of adjunctions between stable $\infty$-categories
\[ (F_i\colon \mathcal{V}^n\longleftrightarrow \mathcal{N}_i\cocolon G_i)_{i\in \mathbb{Z}/n\mathbb{Z}}\]
satisfying that
\begin{enumerate}
    \item $G_i$ is fully faithful, i.e.~$F_iG_i\simeq \on{id}_{\mathcal{N}_i}$ via the counit,
    \item $F_{i}\circ G_{i+1}$ is an equivalence of $\infty$-categories,
    \item $F_i\circ G_j\simeq 0$ if $j\neq i,i+1$,
    \item $G_i$ admits a right adjoint $G_i^R$ and $F_i$ admits a left adjoint $F_i^L$ and
    \item $\on{fib}(G_{i+1}^R)=\on{fib}(F_{i})$ as full subcategories of $\mathcal{V}^n$.
\end{enumerate}
\end{enumerate}
We will consider a collection of functors $(F_i\colon \mathcal{V}^n\rightarrow \mathcal{N}_i)_{i\in \mathbb{Z}/n}$ as a perverse schober on the $n$-spider if there exist adjunctions $(F_i\dashv F_i^R)_{i\in \mathbb{Z}/n\mathbb{Z}}$ defining a perverse schober on the $n$-spider. 
\end{definition}

We remark that conditions (d) and (e) are equivalent to the assertion that the adjunction $\prod_{i=1}^n F_i\colon \V^n\leftrightarrow \prod_{i=1}^n \N_i\noloc G_i$ is spherical, see \cite[Lem.~3.14]{Chr25b}. 

The exit path category of the $n$-spider consists of $n+1$ objects $v,e_1,\dots,e_n$ and $n$ morphisms $v\to e_i$. A functor $\on{Exit}(\rgraph_n)\to \on{LinCat}_R$ thus amounts to a collection of functors $(F_i)_{i\in \mathbb{Z}/n}$ as in \Cref{def:schobernspider}.

\begin{definition}\label{def:schober}
Let $\rgraph$ be a ribbon graph. A functor $\mathcal{F}\colon \on{Exit}(\rgraph)\rightarrow \on{St}$ is called a $\rgraph$-parametrized perverse schober if for each vertex $v$ of $\rgraph$, the restriction of $\mathcal{F}$ to $\on{Exit}(\rgraph)_{v/}$ determines a perverse schober parametrized by the $n$-spider in the sense of \Cref{def:schobernspider}. 
\end{definition}

Most perverse schober considered in this paper will take values in presentable (and thus large) $\infty$-categories. In this case, one simply replaces the target in \Cref{def:schober} by $\mathcal{P}r^L_{\on{St}}$.

We will also need the following explicit local description of parametrized perverse schobers, derived from the relative $S_\bullet$-construction of the underlying spherical functor, see \cite{Dyc21,Chr22}. The equivalence of this model with \Cref{def:schobernspider} is shown in \cite{CHQ23}.

\begin{definition}[The local model for perverse schobers on the $n$-spider]\label{def:locmodel_schober}
Fix $n\geq 1$ and let $F\colon \V\to \N$ be a  spherical functor. 
\begin{itemize}
\item We define $\V^n_F=\V\overset{\rightarrow}{\times_F} \on{Fun}(\Delta^{n-2},\N)$ as the lax limit of $\V$ and $\on{Fun}(\Delta^{n-2},\N)$ along the functor $\V\xrightarrow{F} \N\xrightarrow{(\Delta^{\{0\}}\subset \Delta^{n-2})_*} \on{Fun}(\Delta^{n-2},\N)$, where the second functor is right Kan extension. $\V^n$ can also be concretely be described as the $\infty$-category of sections of the Grothendieck construction of the diagram $\Delta^{n-1}\to \on{St}$ of the form
\[ \V\xrightarrow{F}\N\xrightarrow{\on{id}}\dots \xrightarrow{\on{id}} \N\,.\]
Objects of $\V^n_F$ can thus be identified with diagrams $A\to B_1\to \dots B_{n-1}$, with $A\in \V$ and $B_1,\dots,B_{n-1}\in \N$.  
\item We define the functor $\varrho_1\colon \V^n_F\to \N$ as the restriction functor to the $(n-2)$-th (i.e.~last) component of $\N$ from the lax limit  cone. Thus $\varrho_1$ maps $A\to B_1\to \dots B_{n-1}$ to $B_{n-1}$. The functor $\varrho_1$ admits all repeated left and right adjoints, see \cite{Chr22}. 
\item For $2\leq i\leq n$, we recursively define $\varrho_i=(\varrho_{i-1}^{L})^{L}$ as the doubly left adjoint of $\varrho_{i-1}$. Concretely, we can describe $\varrho_i$ on objects as 
\[
\varrho_i(A\to B_1\to \dots \to B_{n-1})\simeq \begin{cases} \on{fib}(B_{n-i}\to B_{n-i+1})[i-1]& i\not = n\\ \on{fib}(F(A)\to B_1)[n-1] & i=n\,. \end{cases}
\]
\end{itemize} 
\end{definition}

\begin{remark}\label{rem:underlyingsphericalfunctor}
\begin{enumerate}[(1)]
\item The values of a $\rgraph$-parametrized perverse schober $\mathcal{F}$ at any two edges of $\rgraph$ are equivalent. We call the equivalence class of $\mathcal{F}(e)$ for any choice of edge $e$ of $\rgraph$ the generic stalk of $\mathcal{F}$ and denote it by $\N$.
\item We call the spherical functor $F\colon \V\to \N$ appearing in the local model for a perverse schober on the $n$-spider in \Cref{def:locmodel_schober} the spherical functor underlying the perverse schober at the vertex. This spherical functor is unique in an appropriate sense. If $\V\not \simeq 0$, we call the vertex a singularity of $\mathcal{F}$.  Note that the spherical functor underlying a non-singular vertex is given by $F\colon 0\to \N$ and thus $\V^n_F\simeq \on{Fun}(\Delta^{n-2},\N)$.
\item A $\rgraph$-parametrized perverse schober without singularities is called non-singular or also locally constant.
\end{enumerate}
\end{remark}

\begin{definition}\label{def:sections}
Let $\mathcal{F}$ be a $\rgraph$-parametrized perverse schober.
We define the stable $\infty$-category of global sections $\glsec(\rgraph,\mathcal{F})\coloneqq \on{lim}(\mathcal{F})$  of $\mathcal{F}$ as the limit of $\mathcal{F}$. 
\end{definition}


Finally, we discuss the notions of transport and monodromy of parametrized perverse schobers. We fix a marked surface with a spanning ribbon graph $\rgraph$ and a $\rgraph$-parametrized perverse schober $\mathcal{F}$.

\begin{definition}
Let $\mathcal{F}$ be a $\rgraph$-parametrized perverse schober. Let $\gamma\colon [0,1]\to {\bf S}\backslash (M\cup \rgraph_0)$ be any curve with endpoints on edges $\gamma(0)\in e_0, \gamma(1)\in e_1$ of $\rgraph$, considered up to homotopies that move the endpoints at most on these edges. The transport of $\mathcal{F}$ is along $\gamma$ is given by the equivalence $\mathcal{F}^{\rightarrow}(\gamma)\colon \mathcal{F}(e_0)\to \mathcal{F}(e_1)$ obtained as follows. We note that one readily checks this equivalence to be well-defined up to natural equivalence.
\begin{itemize}
\item Suppose that $\rgraph$ is the $n$-spider with edges $e_1,\dots,e_n$. Suppose that $\gamma$ goes one step counterclockwise, starting at $e_i$ and ending at $e_{i+1}$. If $n\geq 2$, we set
\[ \mathcal{F}^{\rightarrow}(\gamma)\coloneqq \mathcal{F}(v\to e_{i+1}) \circ \mathcal{F}(v\to e_i)^L \colon \mathcal{F}(e_i)\longrightarrow \mathcal{F}(e_{i+1})\,.\]
If $n=1$, we set $\mathcal{F}^{\rightarrow}(\gamma)=T_{\mathcal{F}(e_1)}^{-1}$ to be the inverse cotwist of the spherical adjunction $\mathcal{F}(v\to e_1)\dashv \mathcal{F}(v\to e_1)^R$.\footnote{Note that this inverse cotwist is equivalent to the twist of the spherical adjunction $\mathcal{F}(v\to e_1)^L\dashv \mathcal{F}(v\to e_1)$} 

Similarly, if $\gamma$ goes one step clockwise, starting at $e_{i+1}$ and ending at $e_i$, and $n\geq 2$ we set 
\[ \mathcal{F}^{\rightarrow}(\gamma)\coloneqq \mathcal{F}(v\to e_{i}) \circ \mathcal{F}(v\to e_{i+1})^R\colon  \mathcal{F}(e_{i+1})\longrightarrow \mathcal{F}(e_{i})\,.\]
If $n=1$, we set $\mathcal{F}^{\rightarrow}(\gamma)=T_{\mathcal{F}(e_1)}$ to be the cotwist of $\mathcal{F}(v\to e_1)\dashv \mathcal{F}(v\to e_1)^R$.
\item Suppose again that $\rgraph$ is the $n$-spider. Then we can obtain $\gamma$ as the composite of $m\geq 0$ minimal curves $\delta_1,\dots,\delta_m$ which each go one step clockwise or counterclockwise, as before, and set $\mathcal{F}^{\rightarrow}(\gamma)=\mathcal{F}^{\rightarrow}(\delta_m)\circ \dots \circ \mathcal{F}^{\rightarrow}(\delta_1)$. 
\item In general, we can obtain $\gamma$ as the composite of smaller curves $\delta_1,\dots,\delta_m$, each contained in a disc spanned by the $n$-spider at a vertex of $\rgraph$, and set $\mathcal{F}^{\rightarrow}(\gamma)=\mathcal{F}^{\rightarrow}(\delta_m)\circ \dots \circ \mathcal{F}^{\rightarrow}(\delta_1)$. 
\end{itemize}
\end{definition}

While the collection of transports of a perverse schober $\mathcal{F}$ along closed curves can be assembled into a local system of stable $\infty$-categories on ${\bf S}\backslash \rgraph_0$, this does not yield the correct notion of monodromy of $\mathcal{F}$. As for perverse sheaves, one would like the monodromy local system of a perverse schober to extend to the complement of the set of singularities of $\mathcal{F}$, which is a subset of $\rgraph_0$. 

To illustrate this point, consider the trivial spherical adjunction $0\leftrightarrow \N$, which we declare to be the constant perverse schober on the disc parametrized by the $1$-spider, as it categorifies the constant perverse sheaf (whose vanishing cycles are trivial). The clockwise transport along the disc is $T_{\mathcal{N}}=\on{fib}(0\to \on{id}_\N)=[-1]$, which is non-trivial.

The correct notion of monodromy of $\mathcal{F}$ is obtained by choosing a framing of ${\bf S}$ and shifting the transport equivalences by the difference in the winding numbers of the corresponding curves relative to the framing and a line field induced by the ribbon graph, see \cite[Rem.~4.29]{Chr23}. This defines the desired local system of categories, called the monodromy relative to the framing, see \cite[Prop.~4.28]{Chr23}. In the following most relevant will be the case that the generic stalk of $\mathcal{F}$ is $2$-periodic, in which case the monodromy is independent on the choice of framing, see \cite[Rem.~4.31]{Chr23}.

A perverse schober without singularities is fully encoded by its monodromy local system with respect to any given framing, see \cite[Prop.~4.34]{Chr23}.

\subsection{A perverse schober on the basic triangle}

The basic triangle refers to a triangle equipped with a choice of distinguished edge. We will always depict the distinguished edge at the bottom of the triangle. 

In this section, we discuss how to associate a perverse schober on the basic triangle with a Dynkin quiver $\ADE$, building on the recent work of Keller--Liu \cite{KL25}. 

We denote by $[1]=\{0\to 1\}$ the poset $1$-category consisting of a non-invertible morphism. Let $\on{proj}(\ADE)\subset \on{mod}(\ADE)$ be the $1$-category consisting of finitely generated projective $k\ADE$ modules. We consider the functor 
\[ (D_{1},D_2,D_3)\colon \on{proj}(\ADE)^{\times 3} \to \on{Fun}([1],\on{proj}(\ADE))\] between $k$-linear $1$-categories defined by
\[ D_{1}(X)=(X\to 0),\quad D_2(X)=(X\xrightarrow{\on{id}}X),\quad D_{3}(X)=(0\to X)\,.\]
We can also consider $(D_1,D_2,D_3)$ as a dg functor between smooth dg categories. Following Keller--Liu, we pass to the relative (undeformed) $3$-Calabi--Yau completion of this dg functor in the sense of \cite{Yeu16}. This is a relative $3$-Calabi--Yau dg functor 
\begin{equation}\label{eq:CYcompletion}
\Pi_2(\on{proj}(\ADE))^{\times 3}\longrightarrow \GT\coloneqq \Pi_3(\on{Fun}([1],\on{proj}(\ADE)),\on{proj}(\ADE)^{\times 3})\,,
\end{equation}
whose target $\GT$ describes the dg tensor category over the relative inverse dualizing bimodule of $(D_{1},D_2,D_3)$. Note that $\GT$ is an additive dg category with finitely many equivalence classes of indecomposable objects, which are in bijection with the equivalence classes of indecomposable objects in $\on{Fun}([1],\on{proj}(\ADE))$. Further, $\GT$ is smooth and connective and the dg functor \eqref{eq:CYcompletion} has a left $3$-Calabi--Yau structure. Finally, we note that $\GT$ gives rise to a silting subcategory inside its derived $\infty$-category $\D(\GT)$, that we will also denote by $\GT$.

Passing to derived $\infty$-categories, the dg functor \eqref{eq:CYcompletion} yields a $k$-linear functor
\[ (D_1',D_2',D_3')\colon \D(\Pi_2(\ADE))^{\times 3}\longrightarrow \D(\GT)\,.\]
We define 
\[ (\tilde{D}_{1},\tilde{D}_2,\tilde{D}_3)=(D_1',D_2'\circ \D(\sigma),D_3')\,,
\]
with $\sigma$ the involution from \Cref{def:ADEinvolution}. We denote the $k$-linear right adjoint of $\tilde{D}_i$ by $\tilde{D}_i^R$.

\begin{theorem}\label{thm:3gonschober}
The functors 
\begin{equation}\label{eq:3gonschober}
( \tilde{D}_{i}^R\colon \D(\GT) \longrightarrow \D(\Pi_2(\ADE)) )_{i=1,2,3}
\end{equation}
define a perverse schober parametrized by the $3$-spider, denoted $\mathcal{F}_{\Delta,\ADE}$. We orient this perverse schober on the basic triangle as in \Cref{fig:3gonschober}.
\end{theorem}

To prove \Cref{thm:3gonschober}, we make use of the following result of Keller--Liu: 

\begin{theorem}[{$\!\!$\cite[Thm.~3.3.2]{KL25}}]\label{thm:KLmonodromy}~
\begin{enumerate}[(1)]
\item For each $i=1,2,3$, the functor $\tilde{D}_i^R$ is fully faithful.
\item There are equivalences 
\[\tilde{D}_2^R\circ \tilde{D}_1\simeq \tilde{D}_3^R\circ \tilde{D}_2\simeq \tilde{D}_1^R\circ \tilde{D}_3\simeq \D(\sigma)\,,
\]
with $\sigma$ the involution from \Cref{def:ADEinvolution}.
\item $\tilde{D}_i^R\circ \tilde{D}_{i+1}\simeq 0$ for all $i\in \mathbb{Z}/3\mathbb{Z}$. 
\end{enumerate}
\end{theorem} 

\begin{proof}
This follows directly from the results of \cite[Thm.~3.3.2]{KL25} for $D_1',D_2',D_3'$ and their right adjoints, using $\tilde{D}_2\simeq D_2'\circ \D(\sigma)$ and that $\D(\sigma)$ is an involution. We furthermore use that the bimodule $\tau_{\leq 0}\Sigma^{-1}\on{RHom}(\mhyphen,\mhyphen)$ from \cite{KL25} induces $\D(\sigma)$, as follows from \Cref{lem:involution_compatibility}, using that there is a natural transformation $\tau_{\leq 0}\Sigma^{-1}\on{RHom}(\mhyphen,\mhyphen)\to \Sigma^{-1}$ whose fiber is a finite $\Pi_2(\ADE)$-bimodule which vanishes in the cosingularity category.
\end{proof}

\begin{lemma}\label{lem:involution_compatibility}
Let $S\colon \D^{\on{perf}}(\Pi_2(\ADE))\to \D^{\on{perf}}(\Pi_2(\ADE))$ be a quasi-equivalence, which maps $\Pi_2(\ADE)$ to $\Pi_2(\ADE)$ and which induces the suspension functor $[1]$ on $\C_\ADE=\on{CoSing}(\Pi_2(\ADE))$. Then there exists a natural equivalence of endofunctors $\D^{\on{perf}}(\sigma)\simeq S$. 
\end{lemma}

\begin{proof}
The functor $\pi\colon \D^{\on{perf}}(\Pi_2(\ADE))\to \C_\ADE$ induces a morphism of dg algebras
\[
\on{Mor}_{\D(\Pi_2(\ADE))}(\Pi_2(\ADE),\Pi_2(\ADE))\to \on{Mor}_{\on{Ind}\C_\ADE}(\pi(\Pi_2(\ADE)),\pi(\Pi_2(\ADE)))
\]
which is an equivalence on the connective parts. An autoequivalence of $\D^{\on{perf}}(\Pi_2(\ADE))$ mapping $\Pi_2(\ADE)$ to $\Pi_2(\ADE)$ corresponds to an autoequivalence of the dg algebra $\on{Mor}_{\D(\Pi_2(\ADE))}(\Pi_2(\ADE),\Pi_2(\ADE))$, see \cite[Cor.~4.8.5.6]{HA}. Since $S$ and $\D^{\on{perf}}$ induce by assumption the same autoequivalence of the endomorphism dg algebra $\on{Mor}_{\on{Ind}\C_\ADE}(\pi(\Pi_2(\ADE)),\pi(\Pi_2(\ADE)))$ (namely the one corresponding to $[1]$), they also induce the same autoequivalence of the connective truncation, given by $\on{Mor}_{\D(\Pi_2(\ADE))}(\Pi_2(\ADE),\Pi_2(\ADE))$.
\end{proof}

\begin{lemma}\label{lem:3gonSOD}
The stable $\infty$-category $\D(\mathscr{G}_{\Delta,\ADE})$ admits a semiorthogonal decomposition 
\[ \Big(\D(\Pi_3(\on{Aus}(\ADE),\ADE)),\on{Fun}(\Delta^1,\D(\Pi_2(\ADE))) \Big)\,,\] 
where the $k$-linear $\infty$-category $\D(\Pi_3(\on{Aus}(\ADE),\ADE))$ is the relative $3$-Calabi--Yau completion of the inclusion $k\ADE\to \on{Aus}(\ADE)$ of $k\ADE$ into the Auslander--Reiten quiver of its module $1$-category. 
\end{lemma}

\begin{remark}
The $k$-linear $\infty$-category $\D(\Pi_3(\on{Aus}(\ADE),\ADE))$ appears in \cite[Example 8.19]{Wu21}. It is the derived $\infty$-category of a finite dimensional algebra (concentrated in degree $0$). Thus, $\D(\Pi_3(\on{Aus}(\ADE),\ADE))$ is proper as a $k$-linear $\infty$-category. It is also smooth as a relative Calabi--Yau completion.
\end{remark}

\begin{proof}[Proof of \Cref{lem:3gonSOD}]
The $k$-linear subcategory of $\D(\mathscr{G}_{\Delta,\ADE})$ generated by the images of $\tilde{D}_1^R,\tilde{D}_2^R$ is equivalent to $\on{Fun}(\Delta^1,\D(\Pi_2(\ADE)))$ by \Cref{thm:KLmonodromy}. Its semiorthogonal complement is given by the quotient $\D(\mathscr{G}_{\Delta,\ADE})/\on{Fun}(\Delta^1,\D(\Pi_2(\ADE)))$, which is equivalent to the cofiber of the functor $(\tilde{D}_1^R,\tilde{D}_2^R)\colon \D(\Pi_2(\ADE))^{\times 2}\to \D(\mathscr{G}_{\Delta,\ADE})$. Using that the passage to derived $\infty$-categories maps homotopy pushouts to $\infty$-categorical pushouts, this cofiber is by \cite[Prop.~2.3.1]{KL25} equivalent to $\D(\Pi_3(\on{Aus}(kQ),kQ))$.
\end{proof}

\begin{proof}[Proof of \Cref{thm:3gonschober}.]
Passing to right adjoints, \Cref{thm:KLmonodromy} shows that the functors \eqref{eq:3gonschober} satisfy conditions (1),(2),(3) of a perverse schober on the $3$-spider.

By \cite[Lem.~3.12]{Chr25b}, it suffices to further show that the functor $(\tilde{D}_1^R,\tilde{D}_2^R,\tilde{D}_3^R)\colon \D(\mathscr{G}_{\Delta,\ADE}) \to \D(\Pi_2(\ADE))^{\times 3}$ is spherical to conclude that the functors \eqref{eq:3gonschober} define a perverse schober on the $3$-spider. For this we show that the twist functor and cotwist functor of the adjunction $(\tilde{D}_1,\tilde{D}_2,\tilde{D}_3)\dashv \tilde{D}_1,\tilde{D}_2,\tilde{D}_3)$ are invertible. The twist functor can be readily computed using \Cref{thm:KLmonodromy}, it permutes the three components of $ \D(\Pi_2(\ADE))^{\times 3}$ and acts componentwise by $\D(\sigma)$. The cotwist functor is by the relative left $3$-Calabi--Yau structure equivalent to the inverse Serre functor $\on{id}_{ \D(\mathscr{G}_{\Delta,\ADE})}^!$, which is invertible by \Cref{ex:FS}, \Cref{prop:invertibleSerre} and \Cref{lem:resolutionofSOD,lem:3gonSOD}
\end{proof}

\begin{remark}\label{rem:transport}
We can read off the clockwise transport equivalences of the perverse schober \eqref{eq:3gonschober} from \Cref{thm:KLmonodromy}. We can graphically depict them as follows:
\[
\begin{tikzcd}
\D(\Pi_2(\ADE)) \arrow[rd, "\tilde{D}_1"', shift right] \arrow[rdd, "\D(\sigma)"', bend right] &                                                                                                                                                             & \D(\Pi_2(\ADE)) \arrow[ld, "\tilde{D}_3"', shift right] \arrow[ll, "\D(\sigma)", bend right=35] \\
                                                                                            & {\D(\mathscr{G}_{\Delta,\ADE)})} \arrow[d, "\tilde{D}_2^R", shift left] \arrow[ru, "\tilde{D}_3^R"', shift right] \arrow[lu, "\tilde{D}_1^R"', shift right] &                                                                                              \\
                                                                                            & \D(\Pi_2(\ADE)) \arrow[u, "\tilde{D}_2", shift left] \arrow[ruu, "\D(\sigma)"', bend right]                                                                     &                                                                                             
\end{tikzcd}
\]
It follows that the cotwist functor of the spherical adjunction $\D(\Pi_3(\on{Aus}(\ADE),\ADE))\leftrightarrow \D(\Pi_2(\ADE))$ underlying this perverse schober on the $3$-gon in the sense of \Cref{rem:underlyingsphericalfunctor} is equivalent to $\D(\sigma)^3[-2]\simeq \D(\sigma)[-2]$.
\end{remark}

The following proposition establishes the independence of the perverse schober $\mathcal{F}_{\Delta,\ADE}$ on the orientation of the basic triangle. We expect that the action of $T_{\D(\mathscr{G}_{\Delta,\ADE})}$ induces a $\mathbb{Z}/6\mathbb{Z}$-symmetry, and note that this equivalent to $\D(\mathscr{G}_{\Delta,\ADE})$ being fractionally left Calabi--Yau. On the level of the cosingularity category (see \Cref{sec:cosing}), it is however clear that $T_{\D(\mathscr{G}_{\Delta,\ADE})}$ induces a $\mathbb{Z}/6\mathbb{Z}$-symmetry of the quotient perverse schober. We note that a corresponding cluster automorphism is described in \cite[Thm.~12.1]{GS19}.

\begin{proposition}\label{prop:Z6symmetry}
Denote by $T_{\D(\mathscr{G}_{\Delta,\ADE})}$ the cotwist functor of the spherical adjunction 
\[ (\tilde{D}_1,\tilde{D}_2,\tilde{D}_3)\colon \D(\Pi_2(\ADE))^{\times 3}\leftrightarrow  {\D(\mathscr{G}_{\Delta,\ADE}}) \noloc (\tilde{D}_1^R,\tilde{D}_2^R,\tilde{D}_3^R)\,.\]
Then for all $i\in \mathbb{Z}/3\mathbb{Z}$
\[ \tilde{D}_i^R\circ T_{\D(\mathscr{G}_{\Delta,\ADE})} \simeq \D(\sigma)\circ \tilde{D}_{i-1}^R \,.\]
\end{proposition}

\begin{proof}
This follows from \cite[Prop.~3.11]{Chr22} using \Cref{rem:transport} and \Cref{lem:locmodel}, and using that the cotwist functor of the above adjunction is inverse to the twist functor of the adjunction $(\tilde{D}_1^R,\tilde{D}_2^R,\tilde{D}_3^R)\dashv (\tilde{D}_1^{RR},\tilde{D}_2^{RR},\tilde{D}_3^{RR})$. 
\end{proof}

We denote by $F_\Delta\colon \D(\Pi_3(\on{Aus}(\ADE),\ADE))\to \Pi_2(\ADE)$ the spherical functor underlying $\mathcal{F}_{\Delta,\ADE}$. 

\begin{lemma}\label{lem:locmodel}
There is an equivalence $\D(\mathscr{G}_{\Delta,\ADE})\simeq \mathcal{V}^3_{F_\Delta}$, such that the following diagram commutes for all $i=1,2,3$, see also \Cref{def:locmodel_schober} for the notation.
\[
\begin{tikzcd}
{\mathcal{D}(\mathscr{G}_{\Delta,\ADE})} \arrow[rr, "\simeq"] \arrow[rd, "\tilde{D}_i^R"'] &                 & \mathcal{V}^3_{F_\Delta} \arrow[ld, "{\D(\sigma)^{i-1}\circ \varrho_i}"] \\
                                                                                         & \D(\Pi_2(\ADE)) &                                      
\end{tikzcd}
\]
\end{lemma}

\begin{proof}
The functors $\tilde{D}_i^R$ and $\varrho_i$ differ from each other at most by composition with autoequivalences of $\D(\Pi_2(\ADE))$. These can be determined by comparing the transport equivalences between the perverse schober $\mathcal{F}_{\Delta,\ADE}$  and the perverse schober determined by the three functors $\varrho_1,\varrho_2,\varrho_3$. There are the adjunctions $\varrho_{i+1}\dashv \varrho_{i+1}^R\dashv \varrho_{i}$ for $1\leq i\leq 2$ and $\varrho_1^R\circ \D(\sigma)\dashv \varrho_3$, shown in \cite[Section 3]{Chr22}. We thus have $\varrho_1\circ \varrho_2^R\simeq \varrho_2\circ \varrho_3^R\simeq \on{id}_{\D(\Pi_2)}$ and $\varrho_3\circ \varrho_1^R\simeq \D(\sigma)$. 
\end{proof}

We next describe the relative left $3$-Calabi--Yau structure of $\GT$ and specific signs in corresponding negative cyclic homology classes, that will be relevant for the gluing of the Calabi--Yau structures.  

\begin{lemma}\label{lem:signs_relCY}
The functor $(\tilde{D}_1,\tilde{D}_2,\tilde{D}_3)\colon \D(\Pi_2(\ADE))^{\times 3}\rightarrow \D(\GT)$ admits a left $3$-Calabi--Yau structure $\eta\colon k[3]\to \on{HH}^{S^1}(\D(\GT),\D(\Pi_2(\ADE))^{\times 3})$ which restricts on $\D(\Pi_2(\ADE))^{\times 3}$ to a triple of identical classes
\[ (\eta')^{\times 3}\colon k[2]\to \on{HH}^{S^1}(\D(\Pi_2(\ADE)))^{\times 3}\simeq \on{HH}^{S^1}(\D(\Pi_2(\ADE))^{\times 3})\,.\]
Furthermore, the functor $\D(\sigma)$ satisfies $\D(\sigma)(\eta')=-\eta'$.
\end{lemma}

\begin{proof}
The functor $F\colon \D(\Pi_2(\ADE))\to \D(\Pi_3(\on{Aus}(\ADE),\ADE))$ arises from a relative $3$-Calabi--Yau completion and thus admits a left $3$-Calabi--Yau structure $\tilde{\eta}\colon k[3]\to \on{HH}^{S^1}(\D(\GT),\D(\Pi_2(\ADE))^{\times 3})$ by \cite{Yeu16}. We denote by $\eta'\colon k[2]\to \on{HH}^{S^1}(\D(\Pi_2(\ADE)))$ the restriction of of $\eta$. 

By \cite[Prop.~5.2]{Chr23}, the functor
\[
(\varrho_1^R,\varrho_2^R,\varrho_3^R)\colon \D(\Pi_2(\ADE))^{\times 3}\longrightarrow \D(\GT)\,,
\]
given by the right adjoints of the functors $\varrho_i$ from \Cref{def:locmodel_schober}, admit a left $3$-Calabi--Yau structure, which restricts to $(-\eta',\eta',-\eta')\in \on{HH}^{S^1}(\D(\Pi_2(\ADE)))^{\times 3}$. Note that $(\varrho_1^R,\varrho_2^R,\varrho_3^R)\simeq (\varrho_3^L\circ \D(\sigma),\varrho_1^L,\varrho_2^L)\simeq (\tilde{D}_3\circ \D(\sigma(,\tilde{D}_1,\tilde{D}_2\circ \D(\sigma))$ by \Cref{lem:locmodel}, showing the desired signs.

To conclude the proof, it suffices to show that $\D(\sigma)(\eta')=-\eta'$. This follows from \Cref{lem:twist_CY_sign} and the observation that the twist functor of $F\dashv F^R$ is equivalent to $\D(\sigma)[-2]$. 
\end{proof}

\begin{lemma}\label{lem:twist_CY_sign}
Let $R$ be an $\mathbb{E}_\infty$-ring spectrum. Let $F\colon \C\to \D$ be a spherical functor between smooth $R$-linear $\infty$-categories and $\eta\colon R[n]\to \on{HH}^{S^1}(\D,\C)=\on{cof}(\on{HH}^{S^1}(F))$ be an $R$-linear relative negative cyclic homology class. Consider the restriction $\eta_\C\colon R[n-1]\to \on{HH}^{S^1}(\C)$ of $\eta$. Then the twist functor $T_\C$ of the adjunction $F\dashv F^R$ maps $\eta$ to $-\eta$.
\end{lemma}

\begin{proof}
Consider the fiber and cofiber sequence $T_\C[-1]\to \on{id}_\C\to GF$. The assertion $T_\C(\eta)=-\eta$ is equivalent to the assertion that $T_{\C}[-1](\eta)=\eta$. To conclude the latter, it suffices to show that $GF(\eta)=0$. The relative class $\eta$ amounts to the data of $\eta_\C$ together with a trivialization of the image $F(\eta)\in \HH^{S^1}(\D)$. Thus $F(\eta)$ is trivial, and so is $GF(\eta)$. 
\end{proof}

\begin{lemma}\label{lem:Ginzburgboundary}
Let $i\in \mathbb{Z}/3\mathbb{Z}$. Then $\mathcal{F}_{\Delta,\ADE}(v\to e_i)(\GT)\subset \D(\Pi_2(\ADE))$ lies in the additive hull of $\Pi_2(\ADE)$. 
\end{lemma}

\begin{proof}
We have that 
\[ \on{Mor}_{\D(\Pi_2(\ADE))}(\Pi_2(\ADE),\mathcal{F}_{\Delta,\ADE}(v\to e_i)(\GT)) \simeq \on{Mor}_{\mathcal{F}_{\Delta,\ADE}(v)}(\mathcal{F}_{\Delta,\ADE}(v\to e_i)^L(\Pi_2(\ADE)),\GT) \]
and 
\[ \on{Mor}_{\D(\Pi_2(\ADE))}(\mathcal{F}_{\Delta,\ADE}(v\to e_i)(\GT),\Pi_2(\ADE)) \simeq \on{Mor}_{\mathcal{F}_{\Delta,\ADE}(v)}(\GT,\mathcal{F}_{\Delta,\ADE}(v\to e_i)^R(\Pi_2(\ADE))) \,.\]
Using that $\mathcal{F}_{\Delta,\ADE}(v\to e_i)^L(\Pi_2(\ADE)),\mathcal{F}_{\Delta,\ADE}(v\to e_i)^R(\Pi_2(\ADE))\subset \GT$ are direct summands and that $\GT$ has connective morphism objects (derived Homs), we see that the above morphism objects are connective. 

This implies that $\mathcal{F}_{\Delta,\ADE}(v\to e_i)(\GT)\subset \on{Add}(\Pi_2(\ADE))$ lies in the coheart of the co-t-structure, see for instance \cite[Prop.~2.23]{AI12}.
\end{proof}

\subsection{The perverse schober on a triangulated marked surface}

\begin{definition}
\begin{enumerate}[(1)]
\item A direction of an ideal triangulation $\mathcal{T}$ consists of a choice of distinguished edge in each triangle of $\mathcal{T}$ and a choice of direction for each non-boundary edge of $\mathcal{T}$. 
\item A direction of a trivalent spanning graph $\rgraph$ is the corresponding dual notion of a direction of the dual ideal triangulation $\mathcal{T}$. It consists of a choice of halfedge incident to each vertex of $\rgraph$ and a choice of direction for each internal edge of $\rgraph$ (equivalently a choice of halfedge of each internal edge of $\rgraph$, namely the halfedge lying at the vertex at which the directed edge is pointed).
\end{enumerate}
\end{definition}

We fix a Dynkin quiver $\ADE$. We also fix a marked surface ${\bf S}$ with a choice of ideal triangulation $\mathcal{T}$, dual to a trivalent spanning graph $\rgraph$, and a choice of direction of $\mathcal{T}$ and correspondingly of $\rgraph$.

\begin{construction}\label{constr:Teichmullerschober}
We construct a $\rgraph$-parametrized perverse schober $\mathcal{F}_{\rgraph,\ADE}$, determined up to equivalence, as follows. 

Let $v$ of a vertex of $\rgraph$. The direction of $\rgraph$ determines an embedding of the $3$-spider in $\rgraph$ at $v$ into the basic triangle, such that the edge of the chosen halfedge intersects the distinguished edge. 

For each vertex $v\in \rgraph_0$, we let $\mathcal{F}_{v,\ADE}$ be the perverse schober  parametrized by the $3$-spider obtained from $\mathcal{F}_{\Delta,\ADE}$ by composing $\mathcal{F}_{\Delta,\ADE}(v\xrightarrow{h}e)$ with $\D(\sigma)$, whenever $h$ is a chosen halfedge of $e$ at $v$ (in the orientation of $\rgraph$).

We let $\mathcal{F}_{\rgraph,\ADE}$ be the $\rgraph$-parametrized perverse schober obtained from gluing together the perverse schobers $\mathcal{F}_{v,\ADE}$, $v\in \rgraph_0$. This means that $\mathcal{F}_{\rgraph,\ADE}$ is the unique $\rgraph$-parametrized perverse schober satisfying that its restriction to $\on{Exit}(\rgraph_3)\simeq \on{Exit}(\rgraph)_{v/}$ is given by $\mathcal{F}_{v,\ADE}$.
\end{construction}

\begin{lemma}
Up to equivalence, the perverse schober $\mathcal{F}_{\rgraph,\ADE}$ is independent on the choice of orientation of $\rgraph$. 
\end{lemma}

\begin{proof}
Since $\D(\sigma)$ is an involution, choosing different halfedges in the orientation yields an equivalent $\rgraph$-parametrized perverse schober. Choosing different edges incident to the vertices yields equivalent $\rgraph$-parametrized perverse schobers $\mathcal{F}_{v,\ADE}$ by applying \Cref{prop:Z6symmetry}.
\end{proof}

Let $\rgraph'$ be a trivalent spanning graph of ${\bf S}$ obtained as the dual graph of an ideal triangulation differing from $\mathcal{T}$ by the flip at an edge. We also say that $\rgraph$ and $\rgraph'$ are related by a flip. We choose any orientation of $\rgraph'$.

We choose a zig-zag of contractions of ribbon graphs passing from $\rgraph$ to $\rgraph'$ as in \cite[Section 6.4]{Chr22}. Pull-push along this zig-zag allows to obtain from the $\rgraph$-parametrized perverse schober $\mathcal{F}_{\rgraph,\ADE}$ a new $\rgraph'$-parametrized perverse schober $\tilde{\mathcal{F}}_{\rgraph,\ADE}$. 

\begin{proposition}\label{prop:flip_equivalence}
Given ribbon graphs $\rgraph,\rgraph'$ differing by a flip as above, there exists an equivalence of $\rgraph'$-parametrized perverse schobers $\tilde{\mathcal{F}}_{\rgraph,\ADE}\simeq \mathcal{F}_{\rgraph',\ADE}$, which induces an equivalence of global sections
\[
\glsec(\rgraph,\mathcal{F}_{\rgraph,\ADE})\simeq \glsec(\rgraph',\mathcal{F}_{\rgraph',\ADE})\,.
\]
\end{proposition}

\begin{proof}
This follows from essentially the same proof of the statement given in the case $\ADE=A_1$ in \cite[Section 6.4]{Chr22}, by setting $T=\D(\sigma)$. Note that by \Cref{lem:locmodel}, in the local model for perverse schobers on the $3$-spider, the functors $\tilde{D}_i^R$ are identified with $\D(\sigma)^{i-1}\circ \varrho_i$. The argument also uses the equivalence of perverse schobers (in the notation of loc.~cit.) between
\[
\begin{tikzcd}
                                             & {} \arrow[d, "\D(\sigma)\circ \varrho_1", no head, maps to] &                                              \\
                                             & f^*                                                & {} \arrow[l, "\varrho_3"', no head, maps to] \\
{} \arrow[ru, "\varrho_2", no head, maps to] &                                                    &                                             
\end{tikzcd}
\]
and
\[
\begin{tikzcd}
                                             & {} \arrow[d, "\varrho_2", no head, maps to] &                                              \\
                                             & f^*                                         & {} \arrow[l, "\varrho_1"', no head, maps to] \\
{} \arrow[ru, "\D(\sigma)\circ \varrho_3", no head, maps to] &                                             &                                             
\end{tikzcd}
\]
which uses \cite[Prop.~3.11]{Chr22} and the fact that $\D(\sigma)$ is an involution. 
\end{proof}

Since any two ideal triangulations can be connected by a sequence of flips, we obtain the following:

\begin{corollary}\label{cor:global_sections_independence}
Up to equivalence, the $\infty$-category of global section $\glsec(\rgraph,\mathcal{F}_{\rgraph, \ADE})$ is independent on the choice of ideal triangulation and dual ribbon graph $\rgraph$.
\end{corollary}

Finally, we describe $\glsec(\rgraph,\mathcal{F}_{\rgraph,\ADE})$ as the derived $\infty$-category of a dg category $\GS$. We defer to \Cref{subsec:icequivers_HTT} the discussion of the corresponding ice quivers with potential.

\begin{construction}\label{constr:GSfunctor}
Choose an orientation of $\rgraph$. We define a functor $\underline{\GS}\colon \on{Exit}(\rgraph)^{\on{op}}\to \on{dgCat}$ as follows:
\begin{itemize}
\item For each edge $e$ of $\rgraph$, we set $\underline{\GS}(e)=\Pi_2(\on{proj}(\ADE))$.
\item For each vertex $v$ of $\rgraph$, we set $\underline{\GS}(v)=\GT$
\item The orientation of $\rgraph$ determines a total order on the edges $e_1,e_2,e_3$ incident to any vertex $v$. We denote the three components of the dg functor \eqref{eq:CYcompletion} by $((D_1')^{\on{dg}},(D_2')^{\on{dg}},(D_3')^{\on{dg}})$. Let $\sigma^{\on{dg}}\colon \Pi_2(\on{proj}(\ADE))\to \Pi_2(\on{proj}(\ADE))$ be the dg functor corresponding to the involution $\sigma$ from \Cref{def:ADEinvolution}. 
We set 
\[ (\tilde{D}_1^{\on{dg}},\tilde{D}_2^{\on{dg}},\tilde{D}_3^{\on{dg}})=  ((D_1')^{\on{dg}},(D_2')^{\on{dg}}\circ \sigma^{\on{dg}},(D_3')^{\on{dg}})\,.\]
We set \[
\underline{\GS}(e_i\to v)=\begin{cases} \tilde{D}_i^{\on{dg}} & e_i\text{ points away from }v\\
\tilde{D}_i^{\on{dg}}\circ \sigma^{\on{dg}} & e_i\text{ points to  }v\,.  \end{cases}
\]
\end{itemize}
\end{construction}

\begin{definition}\label{def:GS}~
\begin{enumerate}[(1)]
\item The dg category $\GS$ is defined as the homotopy colimit\footnote{With respect to the quasi-equivalence model structure.} of the functor 
\[ \underline{\GS}\colon  \on{Exit}(\rgraph)^{\on{op}}\to \on{dgCat}\] 
from \Cref{constr:GSfunctor}.
\item From the colimit diagram arises the dg functor 
\[ f_{\GS}\colon \Pi_2(\on{proj}(I))^{\amalg \rgraph_1^\partial}\to \GS\,,\] 
with $\rgraph_1^\partial$ the set of external edges of $\rgraph$. 
\end{enumerate}
\end{definition}

\begin{proposition}\label{prop:rel3CY}
The dg category $\GS$ is smooth, connective and has finitely many equivalence classes of indecomposable objects. Furthermore, the dg functor $f_{\GS}$ admits a left $3$-Calabi--Yau structure.
\end{proposition}

\begin{proof}
Smoothness and connectivity are preserved under homotopy colimits with respect to the quasi-equivalence model structure. The homotopy colimit $\GS$ can be computed as the additive closure of the homotopy colimit of the subfunctor on the subcategories of indecomposable objects. Their homotopy colimit has finitely isomorphisms classes of objects. The left $3$-Calabi--Yau structure on $f_{\GS}$ can be obtained via the gluing of relative left $3$-Calabi--Yau structures as in \cite[Thm.~6.2]{BD19}. The local relative Calabi--Yau structures of $\underline{\GS}$ are compatible by the sign discussion in \Cref{lem:signs_relCY}. 
\end{proof}

\begin{proposition}\label{prop:global_sections_GS}
There exists an equivalence of $k$-linear $\infty$-categories
\[
\glsec(\rgraph,\mathcal{F}_{\rgraph,\ADE})\simeq \D(\GS)\,.
\]
\end{proposition}

\begin{proof}
Passing to derived $\infty$-categories, the diagram $\underline{\GS}$ is mapped to the left adjoint diagram of $\mathcal{F}_{\rgraph,\ADE}$. The desired equivalence follows from the facts that the passage to derived $\infty$-categories maps homotopy colimits to colimits in $\on{LinCat}_k$ and that the colimit of the diagram in $\on{LinCat}_k$ is equivalent to the limit in $\on{LinCat}_k$ of the right adjoint diagram.
\end{proof}

\begin{remark}\label{rem:projectivesof3CY}
Given a vertex $v\in \rgraph_0$, evaluation of global sections at $v\in \on{Exit}(\rgraph)$ defines a functor $\on{ev}_e\colon \glsec(\rgraph,\mathcal{F}_\ADE)\to \mathcal{F}(v)= \D(\GT)$. The functor $\on{ev}_e$ admits a left adjoint, which we call left induction, and denote by $\on{ind}_v^L$, see also \cite{Chr25b}. Under the equivalence of $\infty$-categories from $\D(\GS)\simeq \glsec(\rgraph,\mathcal{F}_{\ADE})$, the additive subcategory $\GS$ is mapped to the additive hull of $\bigcup_{v\in \rgraph_0} \on{ind}_v^L(\GT)$. This follows from the observation that in the construction of $\GS$ as a homotopy colimit, the appearing functors give rise to the left adjoints of the functors appearing in the limit diagram of $\mathcal{F}_\ADE$ (i.e.~the evaluation functors) when passing to derived $\infty$-categories.
\end{remark}

\subsection{Ice quivers with potential}\label{subsec:icequivers_HTT}

We begin by associating an ice quiver with potential $(Q_{\Delta,\ADE},F_{\Delta,\ADE},W)$ with the basic triangle, following \cite{KL25}. Consider the Auslander--Reiten quiver $\tilde{Q}$ of the category $\on{Fun}([1],\on{proj}(\ADE))$. Recall that its vertices are representatives of the equivalences classes of the indecomposable objects and the arrows are the irreducible morphisms between these. The Auslander--Reiten quiver further comes with the mesh relations (one for each Auslander--Reiten translation), which are finite sums $\sum_i \epsilon_i\alpha_i'\alpha_i=0$, with $\epsilon_i=\pm 1$ and $\alpha_i',\alpha_i$ arrows in the Auslander--Reiten quiver. 
We define $Q_{\Delta,\ADE}$ as the quiver obtained from the Auslander--Reiten quiver $\tilde{Q}$ by adding 
\begin{enumerate}[(1)]
\item for each Auslander--Reiten translation $\tau\colon i\to j$ a dual arrow $\beta\colon j\to i$, and
\item for each arrow $i\to j$ in $\on{proj}(\ADE)$ mapped to a composite of two arrows in $\tilde{Q}$ a dual arrow $\beta \colon j\to i$. 
\end{enumerate}
The frozen subquiver $F_{\Delta,\ADE}$ of $Q_{\Delta,\ADE}$ consists of the two full subquivers $\ADE$ of $Q_{\Delta,\ADE}$ on the objects $\{0\to P\}_{P\in \on{proj}(\ADE)}$ and $\{P\to 0\}_{P\in \on{proj}(\ADE)}$, as well as the subquiver $\ADE^{\on{op}}$ composed of the dual arrows from (2).

The potential $W$ is obtained by adding 
\begin{itemize}
\item for each mesh relation $\sum_i \epsilon_i\alpha_i'\alpha_i=0$ the term $\beta \sum_i \epsilon_i\alpha_i'\alpha_i=0$, with $\beta$ the corresponding dual arrow and 
\item for each arrow $i\to j$ in $\on{proj}(\ADE)$ mapped to a composite of two arrows $\alpha'\alpha$ in $\tilde{Q}$ the $3$-cycle $-\beta\alpha'\alpha$ with $\beta$ the corresponding dual arrow. 
\end{itemize}

\begin{example}
In the case $\ADE=A_3$, let $P_1,P_2,P_3\in \on{proj}(A_3)$ denote the three projective indecomposable objects. The Auslander--Reiten quiver is given as follows:
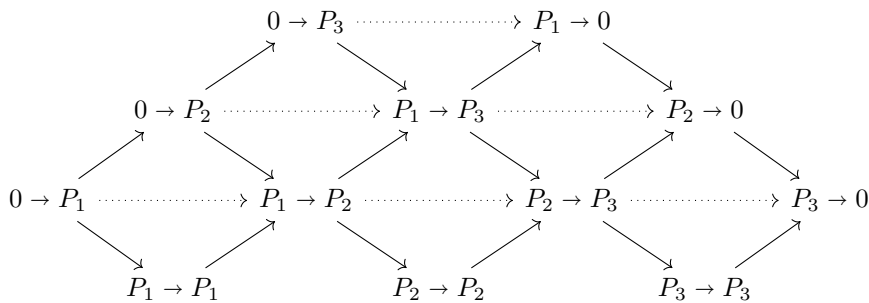
\begin{center}
\begin{tikzcd}[column sep=tiny]
                                                               &                                                                & 0\shortrightarrow P_3 \arrow[rd] \arrow[rr, dotted]              &                                                                  & P_1\shortrightarrow 0 \arrow[rd]                                 &                                    &                       \\
                                                               & 0\shortrightarrow P_2 \arrow[rd] \arrow[ru] \arrow[rr, dotted] &                                                                  & P_1\shortrightarrow P_3 \arrow[ru] \arrow[rd] \arrow[rr, dotted] &                                                                  & P_2\shortrightarrow 0 \arrow[rd]   &                       \\
0\shortrightarrow P_1 \arrow[ru] \arrow[rd] \arrow[rr, dotted] &                                                                & P_1\shortrightarrow P_2 \arrow[rd] \arrow[ru] \arrow[rr, dotted] &                                                                  & P_2\shortrightarrow P_3 \arrow[rd] \arrow[ru] \arrow[rr, dotted] &                                    & P_3\shortrightarrow 0 \\
                                                               & P_1\shortrightarrow P_1 \arrow[ru]                             &                                                                  & P_2\shortrightarrow P_2 \arrow[ru]                               &                                                                  & P_3\shortrightarrow P_3 \arrow[ru] &                      
\end{tikzcd}
\end{center}

The dotted arrows indicate the Auslander--Reiten translation. The corresponding ice quiver with potential $(Q_{\Delta,\ADE},F_{\Delta,\ADE})$ is depicted in \Cref{fig:A3_quiver}. The potential $W$ is given by the sum of the counterclockwise triangles, minus the sum of the clockwise triangles. 

The ice quiver $(Q_{\Delta,A_3},F_{\Delta,A_3})$ has an apparent $\mathbb{Z}/3\mathbb{Z}$ rotational symmetry. For most other orientations in type $A$, and in types $D$ and $E$, the ice quivers $(Q_{\Delta,\ADE})$ do not have such a symmetry, since $\ADE^{\on{op}}\not \simeq \ADE$. 
\end{example}

\begin{remark}\label{rem:Ginzburg_identification}
This ice quiver with potential was described in \cite[Rem.~3.3.1]{KL25}. It is also noted there (though the proof is only sketched) that there exist a Morita equivalence
\[ \mathscr{G}_{(Q_{\Delta,\ADE},F_{\Delta,\ADE},W)}\simeq \GT \] 
with the relative Ginzburg dg category in the sense of \Cref{def:rel_Ginzburg}. We further note without proof that the boundary functor of the relative $3$-Calabi--Yau completion \eqref{eq:CYcompletion} identifies with the boundary functor of the relative Ginzburg dg category, a model for which is given in \Cref{lem:dg_cofibration}. We will not use these facts beyond this section of the paper.
\end{remark}

\begin{remark}
The ice quiver $(Q_{\Delta,\ADE},F_{\Delta,\ADE})$ is expected to recover an ice quiver constructed in \cite[Section 11]{GS19} (used for the cluster seeds of the triangle). In type $A_n$ with the linear orientation, it is however clear that $(Q_{\Delta,\ADE},F_{\Delta,\ADE})$ arises from the construction of loc.~cit.~for the reduced expression $w_0=s_1s_2\dots s_n s_1s_2\dots s_{n-1}\dots s_1s_2s_1$.

Variants of the quiver $\QT$ previously also appeared in \cite{Fei17,Abr18,Le19}. 
\end{remark}

\begin{definition}\label{def:surface_ice_quiver}
Let ${\bf S}$ be a marked surface equipped with an ideal triangulation with dual trivalent spanning ribbon graph $\rgraph$. The ice quiver with potential $(Q_{\rgraph,\ADE},F_{\rgraph,\ADE},W)$ is obtained from $(Q_{\Delta,\ADE},F_{\Delta,\ADE},W)$ via amalgamation along the triangulation as in \Cref{def:amalgamation}. For this, we specify how the glued frozen quivers are considered as coinciding up to their orientations. When gluing two ice quivers $\ADE,\ADE$ along an edge (or analogously for $\ADE^{\on{op}},\ADE^{\on{op}}$), we use the involution $\sigma\colon \ADE_0\simeq \ADE_0$ from \Cref{def:ADEinvolution} (this determines the bijection $\ADE_1\simeq \ADE_1$). When gluing two ice quivers $\ADE,\ADE^{\on{op}}$ along an edge, we use the trivial identification $\ADE_0=\ADE_0^{\on{op}}$ such that $\ADE_1\cap (\ADE_1)^{\on{op}}=\emptyset$. 
\end{definition}

See \Cref{fig:A3_square_quiver} for an example of the amalgamation ice quiver from \Cref{def:surface_ice_quiver}.

The ice quivers in \cite{GS19} arise via the same kind of amalgamation, see \cite[Thm.~9.7]{GS19}. Note that in the amalgamation, the same involution $*=\sigma$ is used, see \Cref{rem:involution_longest_element}, and for instance \cite[Section 13.1.3]{GS19}.

\begin{proposition}\label{prop:Ginzburg_alg}
There exists a Morita equivalence 
\[
\GS\simeq \mathscr{G}_{(Q_{\rgraph,\ADE},F_{\rgraph,\ADE},W)}\,,
\]
and thus 
\[ \glsec(\rgraph,\mathcal{F}_{\rgraph,\ADE})\simeq \D(\mathscr{G}_{(Q_{\rgraph,\ADE},F_{\rgraph,\ADE},W)})\,.\]
\end{proposition}

\begin{proof}
Using \Cref{rem:Ginzburg_identification}, this follows from \Cref{thm:Ginzburg_amalgamation} as follows. We consider the functor $\underline{\GS}\colon  \on{Exit}(\rgraph)^{\on{op}}\to \on{dgCat}$ from \Cref{constr:GSfunctor}. By \Cref{ex:Dynkininvolution}, the involution $\sigma^{\on{dg}}$ of $\Pi_2(\on{proj}(\ADE))$ corresponds under the Morita equivalence $\Pi_2(\on{proj}(\ADE))\simeq \Pi_2(\ADE)$ to the composite of the involution $\psi$ with a dg isomorphism $\xi$ arising from a quiver automorphism $\ADE\simeq \ADE$. Gluing in one triangle at a time corresponds to decomposing the colimit over $\on{Exit}(\rgraph)^{\on{op}}$ via a sequence of pushouts of subdiagrams (using for instance \cite[Cor.~4.2.3.10]{HTT}). Each time a triangle is glued, we apply \Cref{thm:Ginzburg_amalgamation}, noting that the role of $\xi$ is to provide the identification of the frozen parts yielding after gluing the amalgamation ice quiver from \Cref{def:surface_ice_quiver}.
\end{proof}

\begin{remark}
The unfrozen part of the ice quiver with potential $(Q_{\rgraph,\ADE},F_{\rgraph,\ADE},W)$ appears in type $A_n$ in \cite{Abr18,Smi21} in relation with Fukaya categories of $3$-folds with a Lefschetz fibration to the surface and generic fiber the $A_n$-Milnor fibre. We expect $\D(\GS)$ to describe the derived $\infty$-category of a corresponding partially wrapped Fukaya category, and that the cosheaf version of $\mathcal{F}_{\rgraph,\ADE}$ encodes the descent for the partially wrapped Fukaya categories as in \cite{GPS18}.
\end{remark}

\section{The cosingularity category as a topological Fukaya category} \label{sec:cosing}

\subsection{Recollections on \texorpdfstring{$2$}{2}-periodic topological Fukaya categories}\label{subsec:2_periodic_top_Fukaya}

Let $C$ be a $2$-periodic dg category and ${\bf S}$ an oriented marked surface with non-empty boundary. Dyckerhoff--Kapranov give in \cite{DK18} a remarkable construction of the $C$-valued topological Fukaya dg category $\on{Fuk}({\bf S},C)$ of ${\bf S}$, which we summarize in the following. A central feature is that the resulting pretriangulated dg category $\on{Perf}(\on{Fuk}({\bf S},C))$ is determined uniquely up to a contractible space of choices. We note also that the construction of \cite{DK18} has been generalized to $2$-periodic stable $\infty$-categories in \cite{LurWaldhausen}.

The construction of \cite{DK18} involves the following steps: first they show that the Waldhausen $S_\bullet$-construction of $C$ defines a cyclic $2$-Segal object in the category of $2$-periodic dg categories, whose $n$-simplices are Morita equivalent to the dg category of representation of the $A_n$-quiver in $C$, see \cite[Thm.~5.0.1]{DK18}. By evaluating this cyclic object, they obtain a constructible cosheaf of dg categories on every spanning ribbon graph $\rgraph$ of ${\bf S}$, which assigns to an $n$-valent vertex, up to Morita equivalence, the dg category of representations of $A_{n-1}$ in $C$. We remark that this construction is stated in \cite[Section 3.4.2]{DK18} in a slightly different way, but can be interpreted as above. The properties of cyclic $2$-Segal objects guarantee that any contraction of spanning ribbon graphs induces a Morita equivalence on the homotopy colimits of these cosheaves. The space of spanning graphs of ${\bf S}$ and contractions is contractible. Thus $\on{Fuk}({\bf S},C)$ can be defined for any choice of spanning graph of ${\bf S}$ as the homotopy colimit of the corresponding cosheaf. 

As a consequence of their construction, they obtain in \cite[Cor.~3.4.7]{DK18} a (homotopy coherent) action on $\on{Fuk}({\bf S};C)$ of the mapping class group of ${\bf S}$ of homotopy classes of diffeomorphisms preserving the marked points.

The construction of \cite{DK18} of $\on{Fuk}({\bf S};C)$ has an interpretation in terms of the formalism of parametrized perverse schobers. First choose a spanning graph $\rgraph$ of ${\bf S}$, take the cosheaf of \cite{DK18} and pass to perfect derived $\infty$-categories. This yields a constructible cosheaf of stable $\infty$-categories, i.e.~a functor $\on{Exit}(\rgraph)^{\on{op}}\to \on{St}$. Passing to the right adjoint functors in this diagram yields a functor $\mathcal{F}\colon \on{Exit}(\rgraph)\to \on{St}$, which describes a constructible sheaf on $\rgraph$ and further a perverse schober in the sense of \Cref{def:schober}. The generic stalk of $\mathcal{F}$ is $\D^{\on{perf}}(C)$ and $\mathcal{F}$ has no singularities. Furthermore, the monodromy local system of $\mathcal{F}$ on ${\bf S}$ is trivial, and by \cite[Prop.~4.34]{Chr23}, $\mathcal{F}$ is uniquely characterized by these properties up to equivalence. The $\infty$-category of global sections $\glsec({\bf G},\mathcal{F})$ is equivalent to the perfect derived $\infty$-category of $\on{Fuk}({\bf S};C)$, as well as to the colimit of the cosheaf dual of $\mathcal{F}$\footnote{The fact that the global sections of the sheaf and cosheaf are equivalent is not automatic, as these are (co)sheaves of small stable $\infty$-categories, which are not presentable. One can show that the global sections of the cosheaf (the topological Fukaya category) and the global sections of the sheaf (the topological co-Fukaya category) agree by virtue of the assumption that each boundary component contains a marked point.}.

Justified by the above, we thus define:

\begin{definition}
Let $\ADE$ be a Dynkin diagram. We call the $\C_\ADE$-valued topological Fukaya category $\on{Fuk}({\bf S},\C_\ADE)$ the stable $\infty$-category of global sections of any choice of perverse schober with generic stalk $\C_\ADE$ and trivial monodromy, parametrized by any spanning graph of ${\bf S}$.
\end{definition}

\subsection{Recollections on semiorthogonal decompositions of perverse schobers}

We discuss the notion of a semiorthogonal decomposition of perverse schobers, introduced in \cite{Chr22b}.

\begin{definition}\label{def:schoberSOD}
Let $\mathcal{F},\mathcal{G}$ be $\rgraph$-parametrized perverse schobers taking values in presentable $\infty$-categories. We call a natural transformation $\eta\colon \mathcal{F}\to \mathcal{G}$ in $\on{Fun}(\on{Exit}(\rgraph),\mathcal{P}r^L_{\on{St}})$ a morphism of perverse schobers if $\eta$ is locally right adjointable, by which we mean that for every morphism $v\xrightarrow{h} e$ in $\on{Exit}(\rgraph)$ with corresponding diagram
\begin{equation}\label{eq:square}
\begin{tikzcd}
\mathcal{F}(v) \arrow[r, "\eta_v"] \arrow[d, "\mathcal{F}(v\xrightarrow{h} e)"'] & \mathcal{G}(v) \arrow[d, "\mathcal{G}(v\xrightarrow{h} e)"] \\
\mathcal{F}(e) \arrow[r, "\eta_e"]                                   & \mathcal{G}(e)                                 
\end{tikzcd}
\end{equation}
the mate 
\begin{align*}
\eta_v \circ \mathcal{F}(v\shortrightarrow e)^R \xrightarrow{\on{unit}} \ & \mathcal{G}(v\shortrightarrow e)^R\circ \mathcal{G}(v\shortrightarrow e)\circ \eta_v\circ \mathcal{F}(v\shortrightarrow e)^R\\
 \simeq \ & \mathcal{G}(v\shortrightarrow e)^R\circ \eta_e\circ \mathcal{F}(v\shortrightarrow e)\circ \mathcal{F}(v\shortrightarrow e)^R\\
 \xrightarrow{\on{counit}} \ & \mathcal{G}(v\shortrightarrow e)^R\circ \eta_e
\end{align*}
is an equivalence.
\end{definition}

The adjointability condition on the square \eqref{eq:square} may be called vertical right adjointability. There are three further similar conditions, called horizontal/vertical left/right adjointability. If all functors in \eqref{eq:square} also admit left adjoints, then
\begin{itemize}
\item right horizontal adjointability is equivalent to left vertical adjointability and
\item right vertical adjointability is equivalent to left horizontal adjointability.
\end{itemize}

Furthermore, \Cref{prop:adjointability} shows that the two right adjointability conditions are also equivalent. Hence, \Cref{def:schoberSOD} is equivalent to \cite[Def.~3.16]{Chr22b} (requiring horizontal right adjointability).

\begin{proposition}[$\!\!${\cite[Prop.~4.5.6]{CDW23}}]\label{prop:adjointability}
Consider a commutative diagram of stable $\infty$-categories
\[
\begin{tikzcd}
\V \arrow[d, "F_\V"] \arrow[r, "G"] & \N \arrow[d, "F_\N"] \\
\V' \arrow[r, "G'"]                 & \N'                 
\end{tikzcd}
\]
where $G$ and $G'$ are spherical functors. Suppose that $F_\V$ and $F_\N$ admit right adjoints $F_\V^R$ and $F_\N^R$. Then the square is square is horizontally right adjointable if and only if it is vertically right adjointable.
\end{proposition}

\begin{proof}
This follows directly from the proof of \cite[Prop.~4.5.6]{CDW23}. 
\end{proof}

\begin{definition}
An inclusion $\alpha\colon \mathcal{F}\hookrightarrow \mathcal{G}$ of $\rgraph$-parametrized perverse schobers consists of a morphism of perverse schobers such that $\alpha_x\colon \mathcal{F}(x)\to \mathcal{G}(x)$ is fully faithful for all $x\in \on{Exit}(\rgraph)$.
\end{definition}

\begin{definition}
Let $\mathcal{G}$ be a $\rgraph$-parametrized perverse schober. A a semiorthogonal decomposition $(\mathcal{F}_1,\mathcal{F}_2)$  of $\mathcal{G}$ consists of $\rgraph$-parametrized perverse schobers $\mathcal{F}_1,\mathcal{F}_2$, together with an inclusion of perverse schobers $\iota\colon \mathcal{F}_2\hookrightarrow \mathcal{G}$ and a pushout diagram in $\on{Fun}(\on{Exit}(\rgraph),\on{LinCat}_R)$ 
\[
\begin{tikzcd}
\mathcal{F}_2 \arrow[d] \arrow[r, "i", hook] \arrow[rd, "\ulcorner", phantom] & \mathcal{G} \arrow[d, "\pi"] \\
0 \arrow[r]                                                                   & \mathcal{F}_1               
\end{tikzcd}
\]
exhibiting $\mathcal{F}_1$ as the cofiber of $i$. 
\end{definition}

\begin{lemma}[$\!\!${\cite[Rem.~3.17]{Chr22b}}]\label{lem:SOD}
Let $(\mathcal{F}_1,\mathcal{F}_2)$ be a semiorthogonal decomposition of a $\rgraph$-parametrized perverse schober $\mathcal{G}$. Then there exists a semiorthogonal decomposition $(\glsec(\rgraph,\mathcal{F}_1),\glsec(\rgraph,\mathcal{F}_2))$ of $\glsec(\rgraph,\mathcal{G})$. 
\end{lemma}

\begin{lemma}\label{lem:cofiberschober}
Let $(\mathcal{F}_1,\mathcal{F}_2)$ be a semiorthogonal decomposition of a $\rgraph$-parametrized perverse schober $\mathcal{G}$. Then the cofiber map $\pi\colon \mathcal{G}\to \mathcal{F}_1$ is a morphism of perverse schobers.
\end{lemma}

\begin{proof}
Let $v\in \rgraph_0$ be a vertex with an incident halfedge $h\in e\in \rgraph_1$. Consider the following diagram of $R$-linear $\infty$-categories
\[
\begin{tikzcd}
\mathcal{F}_1(v) \arrow[r, "\iota_v", hook, shift left] \arrow[d, "\mathcal{F}_1(h)"'] & \mathcal{G}(v) \arrow[d, "\mathcal{G}(h)"] \arrow[r, "\pi_v", two heads, shift left] \arrow[l, "\iota_v^R", two heads, shift left] & \mathcal{F}_2(v) \arrow[d, "\mathcal{F}_2(h)"] \arrow[l, "\pi_v^R", hook', shift left] \\
\mathcal{F}_1(e) \arrow[r, "\iota_e", hook, shift left]                              & \mathcal{G}(e) \arrow[r, "\pi_e", shift left] \arrow[l, "\iota_e^R", two heads, shift left]                                        & \mathcal{F}_2(e) \arrow[l, "\pi_e^R", hook', shift left]                                 
\end{tikzcd}
\]
where the horizontal rightpointed morphisms define cofiber sequences and the superscript $R$ denotes right adjoints. The left square commutes in both directions, by the right adjointability of the square. The right square commutes in the horizontal right direction and by \Cref{prop:adjointability} it suffices to show that it is horizontally right adjointable. 

The image of the fully faithful functor $\pi_v^R$ is given by the kernel of $\iota_v^R$. Therefore, the image of $\mathcal{G}(h)\circ \pi_v^R$ lies in the kernel of $\iota_e^R$, and thus factors uniquely through the inclusion $\pi_e^R\colon \mathcal{F}_2(e)\hookrightarrow \mathcal{G}(e)$ via a functor $F\colon \mathcal{F}_2(v)\to \mathcal{F}_2(e)$. The equivalences
\begin{align*}
F&\simeq \pi_e \circ\pi_e^R \circ F\\
& \simeq \pi_e\circ \mathcal{G}(h)\circ \pi_v^R\\
& \simeq \mathcal{F}_2(h)\circ \pi_v\circ \pi_v^R \\
& \simeq \mathcal{F}_2(h)
\end{align*}
thus show that the right square commutes in the right vertical direction. We use this to prove that it is also right adjointable:

We must show that the natural transformation 
\begin{align*}
\mathcal{G}(h)\circ \pi_v^R& \rightarrow \pi_e^R\circ \pi_e\circ \mathcal{G}(h)\circ \pi_v^R\\
&\simeq \pi_e^R\circ \mathcal{F}_2(h)\circ \pi_v \circ \pi_v^R\\
& \rightarrow \pi_e^R\circ \mathcal{F}_2(h)
\end{align*}
is a natural equivalence. The third natural transformation above, given by the counit of $\pi_v\dashv \pi_v^R$, is an equivalence by the fully faithfulness of $\pi_v^R$. Using that $\mathcal{G}(h)\circ \pi_v^R\simeq \pi_e^R\circ \mathcal{F}_2(h)$, we see that the first natural transformation applies the unit of $\pi_e\dashv \pi_e^R$ to a functor with image contained in the image of $\pi_e^R$. Restricted to the image of $\pi_e^R$, the unit is an equivalence. 
\end{proof}

\subsection{The equivalence between the cosingularity category and the \texorpdfstring{$\C_\ADE$}{1-cluster category}-valued topological Fukaya category}

We fix a marked surface ${\bf S}$ with trivalent spanning graph $\rgraph$. We also fix a Dynkin quiver $\ADE$. The goal of this section is to describe the cosingularity category of $\D^{\on{perf}}(\GS)\simeq \glsec(\rgraph,\mathcal{F}_{\rgraph,\ADE})^{\on{c}}$, see \Cref{prop:global_sections_GS}, showing that it is equivalent to the topological Fukaya category of ${\bf S}$ valued in the $2$-periodic cosingularity category $\C_\ADE=\on{CoSing}(\Pi_2(\ADE))$. 

We begin by constructing a subschober $\mathcal{F}_{\rgraph,\ADE}^{\on{Ind-fin}}$ of $\mathcal{F}_{\rgraph,\ADE}$ whose global sections describe the Ind-completion of the subcateory of finite objects in $\D^{\on{perf}}(\GS)\simeq \glsec(\rgraph,\mathcal{F}_{\rgraph,\ADE})^{\on{c}}$.

\begin{construction}
Let $v$ be a vertex of $\rgraph$ with an incident edge $e$. The functor 
\[ \mathcal{F}_{\rgraph,\ADE}(v\to e)\colon \mathcal{F}_{\rgraph,\ADE}(v)= \D(\GT)\to \mathcal{F}_{\rgraph,\ADE}(e)=\D(\Pi_2(\ADE))\] maps ($\on{Ind}$-)finite objects to ($\on{Ind}$-)finite objects and thus restricts to a functor 
\[
\mathcal{F}^{\on{Ind-fin}}_{\rgraph,\ADE}(v\to e)\colon \on{Ind}\D^{\on{fin}}(\GT)\to \on{Ind}\D^{\on{fin}}(\Pi_2(\ADE))
\]
These functors assemble into a diagram $\mathcal{F}^{\on{Ind-fin}}_{\rgraph,\ADE}\colon \on{Exit}(\rgraph)\to \on{LinCat}_k$, together with a natural transformation $\alpha\colon \mathcal{F}^{\on{Ind-fin}}_{\rgraph,\ADE}\hookrightarrow \mathcal{F}_{\rgraph,\ADE}$, given by pointwise inclusion.
\end{construction}

\begin{lemma}
The diagram $\mathcal{F}^{\on{Ind-fin}}_{\rgraph,\ADE}$ defines a $\rgraph$-parametrized perverse schober.
\end{lemma}
 
\begin{proof}
Denote by
\[ F\colon \D(\Pi_3(\on{Aus}(kQ),kQ))\to \D(\Pi_2(\ADE))\]
the spherical functor underlying the perverse schober $\mathcal{F}_{\Delta,\ADE}$ on the $3$-spider. Then 
\[ \mathcal{F}_{\Delta,\ADE}(v)\simeq \D(\Pi_3(\on{Aus}(kQ),kQ))\times_F^{\rightarrow} \D(\Pi_2(\ADE))\times^\rightarrow_{\on{id}_{\D(\Pi_2(\ADE))}} \D(\Pi_2(\ADE))\] is equivalent to the lax limit. We note that $F$ takes values in $\on{Ind}\D^{\on{fin}}(\Pi_2(\ADE))\subset  \D(\Pi_2(\ADE))$ and denote by $F'\colon \D(\Pi_3(\on{Aus}(kQ),kQ))\to \on{Ind}\D^{\on{fin}}(\Pi_2(\ADE))$ the restriction of $F$. There is thus an embedding
\[
\D(\Pi_3(\on{Aus}(kQ),kQ))\times_{F'}^{\rightarrow} \on{Ind}\D^{\on{fin}}(\Pi_2(\ADE))\times^\rightarrow_{\on{id}} \on{Ind}\D^{\on{fin}}(\Pi_2(\ADE))\subset\mathcal{F}_{\Delta,\ADE}(v)\,. 
\]
Its image consists of the full subcategory 
\[ \on{Ind}\D^{\on{fin}}(\GT)\subset \D(\GT)\simeq \D(\Pi_3(\on{Aus}(kQ),kQ))\times_F^{\rightarrow} \D(\Pi_2(\ADE))\times^\rightarrow_{\on{id}_{ \D(\Pi_2(\ADE))}} \D(\Pi_2(\ADE))\]
which follows from the observation that an object in $\D(\GT)$ is finite if and only if it is finite in each component of the above lax limit and that all objects in $\D(\Pi_3(\on{Aus}(kQ),kQ))$ are finite.

Since $\mathcal{F}^{\on{Ind-fin}}_{\rgraph,\Delta}$ is given by restricting $\mathcal{F}_{\rgraph,\Delta}$ to the subcategories of Ind-finite objects, we thus see that $\mathcal{F}^{\on{Ind-fin}}_{\rgraph,\Delta}$ is locally described by the local model for a perverse schober on the $3$-spider of \Cref{def:locmodel_schober} arising from the spherical functor $F'$ and thus indeed defines a perverse schober. 
\end{proof}

\begin{lemma}
The natural transformation $\alpha\colon \mathcal{F}^{\on{Ind-fin}}_{\rgraph,\ADE}\to \mathcal{F}_{\rgraph,\ADE}$ defines a morphism of perverse schobers. 
\end{lemma}

\begin{proof}
The proof of the corresponding statement in type $\ADE=A_1$ in \cite[Lem.~4.1]{Chr22b} directly generalizes.
\end{proof}

\begin{proposition}\label{prop:Ind_finite_global_sections}
The image of the fully faithful functor 
\[ \glsec(\rgraph,\mathcal{F}^{\on{Ind-fin}}_{\rgraph,\ADE})\xlongrightarrow{\glsec(\rgraph,\alpha)} \glsec(\rgraph,\mathcal{F}_{\rgraph,\ADE})\simeq \D(\GS)\]
is given by $\on{Ind}\D^{\on{fin}}(\GS)$. 
\end{proposition}

\begin{proof}
Let $X\in \glsec(\rgraph,\mathcal{F}_{\rgraph,\ADE})$. Then by \Cref{rem:projectivesof3CY}
\[
\on{Mor}(\GS,X)\simeq \on{Mor}(\prod_{v\in \rgraph_0}\on{ind}^L_v(\GT),X)\simeq \prod_{v\in \rgraph_0}\on{Mor}(\GT,\on{ev}_v(X))\,.
\]
Thus a global section is finite if and only if its restrictions to all vertices of $\rgraph$ is finite.
\end{proof}

We let $\mathcal{F}^{\on{clst}}_{\rgraph,\ADE}$ be the cofiber of $\alpha\colon \mathcal{F}^{\on{Ind-fin}}_{\rgraph,\ADE}\to \mathcal{F}_{\rgraph,\ADE}$ in $\on{Fun}(\on{Exit}(\rgraph),\on{LinCat}_k)$.

\begin{proposition}\label{prop:SODofschobers}
$\mathcal{F}^{\on{clst}}_{\rgraph,\ADE}$ is a $\rgraph$-parametrized perverse schober and there is thus a semiorthogonal decomposition $(\mathcal{F}^{\on{clst}}_{\rgraph,\ADE},\mathcal{F}^{\on{Ind-fin}}_{\rgraph,\ADE})$ of the perverse schober $\mathcal{F}_{\rgraph,\ADE}$. In particular, there exists an equivalence $\glsec(\rgraph,\mathcal{F}^{\on{clst}}_{\rgraph,\ADE})\simeq \on{Ind}\on{CoSing}(\GS)$. 
\end{proposition}

\begin{proof}
Let $e$ be an edge of $\rgraph$. Then there is an equivalence $\mathcal{F}^{\on{clst}}_{\rgraph,\ADE}(e)\simeq \mathcal{F}_{\rgraph,\ADE}(e)/\mathcal{F}^{\on{Ind-fin}}_{\rgraph,\ADE}(e)\simeq \on{Ind}\C_\ADE$.

Let $v\in \rgraph_0$ be a vertex. Using the equivalences 
\[\mathcal{F}_{\rgraph,\ADE}(v)\simeq \D(\Pi_3(\on{Aus}(kQ),kQ))\times_F^{\rightarrow} \D(\Pi_2(\ADE))\times^\rightarrow_{\on{id}_{\D(\Pi_2(\ADE))}} \D(\Pi_2(\ADE))\]
and 
\[
\mathcal{F}^{\on{Ind-fin}}_{\rgraph,\ADE}(v)\simeq \D(\Pi_3(\on{Aus}(kQ),kQ))\times_{F'}^{\rightarrow} \on{Ind}\D^{\on{fin}}(\Pi_2(\ADE))\times^\rightarrow_{\on{id}_{\on{Ind}\D^{\on{fin}}}} \on{Ind}\D^{\on{fin}}(\Pi_2(\ADE))
\]
we find 
\begin{align*}
\mathcal{F}^{\on{clst}}_{\rgraph,\ADE}(v)\simeq \mathcal{F}_{\rgraph,\ADE}(v)/\mathcal{F}^{\on{Ind-fin}}_{\rgraph,\ADE}(v) \simeq 0\times^{\rightarrow} \on{Ind}\C_\ADE\times^\rightarrow_{\on{id}_{\on{Ind}\C_\ADE}}\on{Ind}\C_\ADE\simeq \on{Fun}(\Delta^1,\on{Ind}\C_\ADE)\,.
\end{align*}
For each halfedge $h$ of an edge $e$ incident to a vertex $v$, the functor $\mathcal{F}^{\on{clst}}_{\rgraph,\ADE}(v\xrightarrow{h} e)$ arises from the functor $\mathcal{F}_{\rgraph,\ADE}(v\xrightarrow{h} e)$ via the universal property of the Verdier quotient. Using this, it is straightforward to verify that $\mathcal{F}^{\on{clst}}_{\rgraph,\ADE}$ is locally equivalent to the model from \Cref{def:locmodel_schober}, and thus a $\rgraph$-parametrized perverse schober.

The latter part follows from \Cref{lem:SOD} and \Cref{prop:Ind_finite_global_sections}.
\end{proof}

\begin{theorem}\label{thm:global_sections_=_Fukaya}
The perverse schober $\mathcal{F}^{\on{clst}}_{\rgraph,\ADE}$ has generic stalk $\on{Ind}\C_\ADE$, no singularities and its monodromy local system on ${\bf S}$ is trivial. Hence there exists an equivalence of $k$-linear $\infty$-categories
\[\glsec(\rgraph,\mathcal{F}^{\on{clst}}_{\rgraph,\ADE})\simeq \on{Ind}\on{Fuk}({\bf S},\C_\ADE)\,.
\]
\end{theorem}

\begin{proof}
The generic stalk of $\mathcal{F}^{\on{clst}}_{\rgraph,\ADE}$ is given by its value at any edge $e$ of $\rgraph$, which is given by $\on{Ind}\C_\ADE$. 

Since for each vertex $v$ of $\rgraph$ there is an equivalence
\[\mathcal{F}^{\on{clst}}_{\rgraph,\ADE}(v)\simeq \on{Fun}(\Delta^1,\on{Ind}\C_\ADE)\]
we find that $\mathcal{F}^{\on{clst}}_{\rgraph,\ADE}$ has no singularities.

As $\mathcal{F}^{\on{clst}}_{\rgraph,\ADE}$ is non-singular, its monodromy local systems extends from ${\bf S}\backslash \rgraph_0$ to ${\bf S}$, see \cite[Prop.~4.28]{Chr23}. Fix a closed curve $\gamma$ in ${\bf S}$. Then $\gamma$ is homotopic to the composite of segments, embedded into the ideal triangles of ${\bf S}$ and which start at a boundary edge of the triangle and end at a different boundary edge of the triangle. We can choose these segments such that they do not hit $\rgraph_0$ and turn exactly one step clockwise or counterclockwise in the triangle punctured by the vertex of $\rgraph$. Inspecting \Cref{rem:transport} and the construction of $\mathcal{F}_{\rgraph,\ADE}$, one readily sees that the transport of $\mathcal{F}_{\rgraph,\ADE}$ along the composite of these segments is trivial. It follows that the monodromy of $\mathcal{F}^{\on{clst}}_{\rgraph,\ADE}$ along the closed curve $\gamma$ is also trivial. 
\end{proof}

\begin{theorem}\label{thm:Cosing_Fukaya}
There exists an equivalence of $k$-linear $\infty$-categories
\[
\on{CoSing}(\GS)\simeq \on{Fuk}({\bf S},\C_\ADE) \,.
\]
\end{theorem}

\begin{proof}
We note that 
\[ \on{Ind}\on{CoSing}(\GS)\simeq \glsec(\rgraph,\mathcal{F}_{\rgraph,\ADE})/\glsec(\rgraph,\mathcal{F}^{\on{Ind-fin}}_{\rgraph,\ADE})\simeq \glsec(\rgraph,\mathcal{F}^{\on{clst}}_{\rgraph,\ADE})\] by \Cref{prop:Ind_finite_global_sections}. The equivalence thus follows from \Cref{thm:global_sections_=_Fukaya}.
\end{proof}

The article \cite{Chr25b} discusses the gluing of cluster tilting subcategories along perverse schobers. Theorem 1 in \cite{Chr25b} shows that, given a cluster tilting subcategory $\mathcal{T}_v\subset \mathcal{F}_{\rgraph,\ADE}^{\on{clst}}(v)^{\on{c}}\simeq \on{CoSing}(\GT)\simeq \on{Fun}(\Delta^1,\C_\ADE)$ for every vertex $v$ of $\rgraph$, then the additive subcategory
\[
\on{Add}(\bigcup_{v\in \rgraph_0} \on{ind}^L_v(\mathcal{T}_v))\subset \glsec(\rgraph,\mathcal{F}^{\on{clst}}_{\rgraph,\ADE})^{\on{c}}\simeq \on{CoSing}(\GS)
\]
is cluster tilting as well. Here $ \on{ind}^L_v\colon \mathcal{F}^{\on{Ind-fin}}_{\rgraph,\ADE}(v)^c\to \glsec(\rgraph,\mathcal{F}^{\on{clst}}_{\rgraph,\ADE})^{\on{c}}$ denotes the left induction functor, see \Cref{rem:projectivesof3CY}. Further note that $\glsec(\rgraph,\mathcal{F}^{\on{clst}}_{\rgraph,\ADE})^{\on{c}}$ is understood to be equipped with the relative Frobenius $\infty$-categorical exact structure arising from the spherical functor 
\[\prod_{e\in \rgraph_1^\partial}\on{ev}_{e}\colon \glsec(\rgraph,\mathcal{F}^{\on{clst}}_{\rgraph,\ADE})^{\on{c}}\to \prod_{e\in \rgraph_1^\partial}\on{ev}_{e}\mathcal{F}^{\on{clst}}(e)\,.\]

The equivalence between $\on{CoSing}(\GT)$ and the Higgs category of $\GT$ proven in \cite[Section 7.3]{KL25} shows that the image of $\GT$ in $\on{CoSing}(\GT)$ defines a cluster tilting subcategory (recall that we consider $\GT$ and $\GS$ as additive subcategories of their derived categories). Combining this with the gluing of cluster tilting objects, we obtain:
 
\begin{theorem}\label{thm:GS_cluster_tilting}
The image of the additive subcategory $\GS\subset \D(\GS)$ in $\on{CoSing}(\GS)\simeq \on{Fuk}({\bf S},\C_\ADE)$ is a cluster tilting subcategory.
\end{theorem}

\begin{proof}
In the case that ${\bf S}=\Delta$ is the triangle, this is proved in \cite{KL25}. For ${\bf S}$ arbitrary, this follows from \cite[Thm.~1]{Chr25b} using \Cref{rem:projectivesof3CY} and the observation that left induction commutes with morphisms of perverse schobers. 
\end{proof}

\begin{lemma}\label{lem:endomorphismalgebra}
Denote by $\pi\colon\D(\GS)\to \on{CoSing}(\GS)$ the quotient functor. The functor $\pi$ induces an isomorphism
\[
\pi\colon \on{Ext}^{-i}_{\D(\GS)}(P,P')\xlongrightarrow{\simeq} \on{Ext}^{-i}_{\on{CoSing}(\GS)}(\pi(P),\pi(P'))
\]
for all $P,P'\in \GS$ and all $i\geq 0$.
\end{lemma}

\begin{proof}
In the case ${\bf S}=\Delta$ a triangle, \cite{KL25} show that 
\[ \tau_{\geq 0}\on{Mor}_{\on{CoSing}(\GS)}(\pi(P),\pi(P'))\simeq \on{Mor}_{\D(\GS)}(\pi(P),\pi(P'))\]
 (where the truncation is taken with respect to the homological grading convention). The general case follows from gluing using the following two observations and \Cref{lem:Ginzburgboundary}:
\begin{itemize}
\item The truncation $\tau_{\geq 0}$  in $\D(k)$ commutes with pullbacks. 
\item Consider a pullback diagram in $\on{LinCat}_k$
\[
\begin{tikzcd}
\C \arrow[r, "F_2"] \arrow[d, "F_1"] & \B_2 \arrow[d, "G_2"] \\
\B_1 \arrow[r, "G_1"]                & \A                    
\end{tikzcd}
\]
Let $X,Y\in \C$. The arising square in $\D(k)$
\[
\begin{tikzcd}
{\on{Mor}_{\C}(X,Y)} \arrow[r] \arrow[d]   & {\on{Mor}_{\B_2}(F_2(X),F_2(Y))} \arrow[d] \\
{\on{Mor}_{\B_1}(F_1(X),F_1(Y))} \arrow[r] & {\on{Mor}_{\A}(G_1F_1(X),G_1F_1(Y))}      
\end{tikzcd}
\]
is pullback. This follows for instance from the proof of \cite[Prop.~3.12]{Chr23}.
\end{itemize}
\end{proof}

\begin{remark}
Combining \Cref{prop:Ginzburg_alg} and \Cref{lem:endomorphismalgebra}, or alternatively applying \cite[Prop.~5.10]{Chr22b}, shows that the endomorphism algebra of the cluster tilting object in $\on{Fuk}({\bf S},\C_\ADE)$ arising from the cluster tilting subcategory in \Cref{thm:GS_cluster_tilting} is described by the amalgamation ice quiver $(Q_{\rgraph,\ADE},F_{\rgraph,\ADE})$ from \Cref{def:surface_ice_quiver}.
\end{remark}

\section{The cosingularity category as a Higgs category}

In this section, we prove that the Higgs category $\mathcal{H}_{\GS}$ is equivalent to the cosingularity category of the relative $3$-Calabi--Yau dg category $\GS$. We do so by a minor modification of the argument for the case $\ADE=A_1$ given in the prequel \cite{Chr22b}. The argument consists of two steps: One first observes that there is a canonical exact functor from the Higgs category $\mathcal{H}_{\GS}$ to the cosingularity category of $\GS$. Secondly, one shows that this functor maps a cluster tilting object to a cluster tilting object, preserving the endomorphisms, and deduces that the functor is an equivalence of exact $\infty$-categories.

\subsection{Recollections on Higgs categories}

In this following, we describe an $\infty$-categorical version of the Higgs dg category of \cite{Wu21}, which generalizes Amiot's construction of the cosingularity category \cite{Ami09} to the relative $3$-Calabi--Yau setting. As input for this construction serves a dg functor between smooth idempotent complete dg categories $f\colon B\to A$  such that
\begin{itemize}
\item $A$ and $B$ have finitely many isomorphisms classes of indecomposable objects,
\item all morphisms complexes in $A$ are connective,
\item $H_0(A)$ has finite dimensional Homs and
\item $f$ admits a left $3$-Calabi--Yau structure.  
\end{itemize}
Note that $A,B$ are thus Morita equivalent to dg algebras and the dg functor $f$ corresponds to a non-unital dg algebra morphism. 

A typical instance of the above arises from an ice quiver with potential $(Q,F,W)$ by setting $A=\mathscr{G}_{(Q,F,W)}$ to be the relative $3$-Calabi--Yau Ginzburg dg category and $f$ to be the Ginzburg dg functor \cite{Wu21}
\[ {\bf Gi}\colon \Pi_2(F)\to \mathscr{G}_{(Q,F,W)}\,.\]
This satisfies the above conditions if $H_0(\mathscr{G}_{(Q,F,W)})$ has finite dimensional Homs. 

Let $\D(f)\colon \D(B)\to \D(A)$ be the $k$-linear functor induced on the derived $\infty$-categories and $\D(f)^R$ its right adjoint. 

We denote by $\D^{\on{fin}}_{B}(A)\subset \D^{\on{fin}}(A)$ the full subcategory consisting of all elements in the kernel of the functor $\D(f)^R$. The relative cluster ($\infty$-)category of \cite{Wu21} may be defined as the Verdier quotient
\[ \C^{\on{rel}}_A\coloneqq \D^{\on{perf}}(A)/\D^{\on{fin}}_B(A)\,.\]
We note that $\C^{\on{rel}}_A$ is equivalent to the perfect derived $\infty$-category of the relative cluster dg category of \cite{Wu21}. The Higgs category $\mathcal{H}_{A}$ is a certain extension closed, full subcategory of $\C^{\on{rel}}_A$, which thus inherits the structure of an exact $\infty$-category. To state its definition, we first define the so-called relative fundamental domain $\mathcal{F}^{\on{rel}}\subset \D^{\on{perf}}(A)$.  Consider the set $\mathcal{P}$ of objects of $A$ lying the image of $f$.

\begin{definition}\label{def:fundamentaldom}
The relative fundamental domain $\mathcal{F}^{\on{rel}}\subset \D^{\on{perf}}(A)$ is the full subcategory spanned by objects $X$ satisfying that 
\begin{enumerate}[1)]
\item $X$ fits into a fiber and cofiber sequence $M_1\to M_0\to X$ with $M_0,M_1$ lying in the additive closure $\on{Add}(A)$ and
\item $\on{Ext}^i(P,X)\simeq \on{Ext}^i(X,P)\simeq 0$ for all $P\in \mathcal{P}$ and $i>0$.
\end{enumerate}
\end{definition}

In favorable situations, the definition of the fundamental domain can be simplified as follows:

\begin{lemma}[$\!\!${\cite[Lem.~8.2]{Chr22b}}]\label{lem:fundamentaldomain}
Suppose that the functor $\D(f)^R$ admits a colimit preserving right adjoint $\D(f)^{RR}$ satisfying that $\on{Add}(\mathcal{P})=\D(f)^{RR}(\on{Add}(B))$. Then condition 2) in \Cref{def:fundamentaldom} is equivalent to $X$ satisfying $\D(f)^R(X)\in \on{Add}(B)\subset \D^{\on{perf}}(B)$. 
\end{lemma}

The Higgs category $\mathcal{H}_A\subset \C^{\on{rel}}_A$ is defined as the image of $\mathcal{F}^{\on{rel}}$ under the quotient functor $\D^{\on{perf}}(A)\to \C^{\on{rel}}_A$. The arising functor $\mathcal{F}^{\on{rel}}\to \mathcal{H}_A$ is furthermore fully faithful on the level of homotopy categories \cite[Prop.~5.20]{Wu21}. The image of $A$ in $\C^{\on{rel}}_A$ lies in $\mathcal{H}_A$ and defines a cluster tilting subcategory in $\mathcal{H}_A$ with respect to its $\infty$-categorical exact structure, see \cite{Wu21}.

\subsection{The equivalence}

Consider the connective dg category $\GS$ associated with a marked surface ${\bf S}$ with trivalent spanning graph $\rgraph$ and Dynkin quiver $\ADE$ from \Cref{def:GS} and the corresponding relative left $3$-Calabi--Yau functor 
\[f_{\GS}\colon \Pi_2(\on{proj}(I))^{\amalg \rgraph_1^\partial}\to \GS\,.\]

Let $\mathcal{H}_{\GS}$ be the associated Higgs category. We define the functor $\tau\colon \mathcal{H}_{\GS}\to \on{CoSing}(\GS)$ as the composite 
\begin{align*}
\mathcal{H}_{\GS}&\subset \C^{\on{rel}}_{\GS}= \D^{\on{perf}}(\GS)/\D^{\on{fin}}_{\Pi_2(\on{proj}(I))^{\amalg \rgraph_1^\partial}}(\GS)\\
& \twoheadrightarrow  \on{CoSing}(\GS)=\D^{\on{perf}}(\GS)/\D^{\on{fin}}(\GS)\,.
\end{align*}

Using the equivalence $\on{CoSing}(\GS)\simeq \glsec(\rgraph,\mathcal{F}^{\on{clst}}_{\rgraph,\ADE})$ from \Cref{prop:Ind_finite_global_sections}, the stable $\infty$-category $\on{CoSing}(\GS)$ inherits a relative Frobenius exact $\infty$-structure arising from boundary restriction, see \Cref{ex:schober_inducedexstr}.

The main result of this section is the following:

\begin{theorem}\label{thm:equivHiggsCosing}
The functor 
\[
\tau\colon \mathcal{H}_{\GS}\to \on{CoSing}(\GS)
\]
is an equivalence of exact $\infty$-categories.
\end{theorem}

\begin{lemma}
The functor $\tau$ is an exact functor between exact $\infty$-categories.
\end{lemma}

\begin{proof}
By Proposition 8.3 in \cite{Chr22b}, which bases on \Cref{lem:fundamentaldomain}, it suffices to show that $\D(f_{\GS})^R$ admits a colimit preserving right adjoint $\D(f_{\GS})^{RR}$ satisfying that 
\[ \on{Add}(\mathcal{P})=\D(f_{\GS})^{RR}(\on{Add}(\Pi_2(\on{proj}(I))^{\amalg \rgraph_1^\partial}))\,.\]
Firstly, we note that the functor $\D(f_{\GS})^R\colon \D(\GS)\to \D(\Pi_2(\on{proj}(I))^{\amalg \rgraph_1^\partial}))\simeq \D(\Pi_2(\ADE))^{\times \rgraph_1^\partial}$ corresponds via the equivalence from \Cref{prop:global_sections_GS} to the spherical functor $\prod_{e\in {\bf G}_1^\partial}\on{ev}_e$, and hence admits a colimit preserving right adjoint $\D(f_{\GS})^{RR}$. The functor $\D(f_{\GS})^{RR}$ is equivalent to the composite of $\D(f_{\GS})$ with the inverse cotwist functor $T_{\D(\Pi_2(\ADE))^{\times \rgraph_1^\partial}}^{-1}$, see \cite[Cor.~2.5.16]{DKSS21}. The inverse twist functor preserves the additive hull of the projectives, see \Cref{lem:AddPi2}.
\end{proof}

\begin{lemma}\label{lem:AddPi2}
The cotwist functor $T_{\D(\Pi_2(\ADE))^{\times \rgraph_1^\partial}}$ of the spherical adjunction 
\[
\prod_{e\in {\bf G}_1^{\partial}}\on{ev}_e\colon  \glsec({\bf G},\mathcal{F}_{\rgraph,\ADE}) \longleftrightarrow \D(\Pi_2(\ADE))^{\times \rgraph_1^\partial}\noloc \prod_{e\in {\bf G}_1^{\partial}}\on{ev}_e^R
\]
preserves the additive subcategory $\on{Add}(\Pi_2(\ADE)^{\times \rgraph_1^\partial})$. 
\end{lemma}

\begin{proof}
Applying \cite[Prop.~4.10]{Chr25b}, we find that the cotwist acts via transport equivalences of $\mathcal{F}_{\rgraph,\ADE}$ along boundary arcs. These transports are given by powers of the involution $\D(\sigma)$, as follows from \Cref{rem:transport} and the construction of $\mathcal{F}_{\rgraph,\ADE}$. The involution $\D(\sigma)$ preserves the additive subcategory $\on{Add}(\Pi_2(\ADE)$.
\end{proof}

The following is a variant of \cite[Lem.~8.6]{Chr22b}:

\begin{lemma}\label{lem:detectextriangulatedequiv}
Let $\tau\colon (C,\mathbb{E},\mathfrak{s})\to (C',\mathbb{E}',\mathfrak{s}')$ be an extriangulated functor, such that
\begin{enumerate}[1)]
\item there are cluster tilting subcategories $\T\subset C$ and $\T'\subset C'$,
\item $\tau(\T)=\T'$,
\item $\tau\colon \on{Hom}_C(T,T)\to \on{Hom}_{C'}(\tau(T),\tau(T))$ is an isomorphism for every $T\in \T$ and
\item $\tau\colon \on{Hom}_C(T,T)\to \on{Hom}_{C'}(\tau(T'),\tau(T'))$ restricts to a bijection on the subsets of inflations for every $T\in \T$.
\end{enumerate}
Then $\tau$ is an extriangulated equivalence.
\end{lemma}

\begin{proof}
We first show that $\tau$ is essentially surjective. Let $X\in C'$ and choose an exact $2$-term resolution $T_1'\to T_0'\to X$ with $T_0',T_1'\in \T'$. All objects of $\T'$ lift by 2) and 3) uniquely (up to equivalence) along $\tau$ to objects of $\T$. The morphism $T_1'\to T_0'$ is an inflation and thus lifts by 4) along $\tau$ to an inflation $T_1\to T_0$ with $T_0,T_1\in \T$. The third term in the arising exact sequence $T_1\to T_0\to Y$ in $C$ gives the desired lift of $X$ along $\tau$. 

It remains to show that $\tau$ is fully faithful. This follows by the same computation as in the proof of \cite[Lem.~8.6]{Chr22b}.
\end{proof}

\begin{proof}[Proof of \Cref{thm:equivHiggsCosing}]
We first prove that the functor $\on{ho}\tau$ between the extriangulated homotopy categories is an equivalence. We apply \Cref{lem:detectextriangulatedequiv} to show this. Conditions $1)$ and $2)$ follow from the commutativity of the diagram
\[
\begin{tikzcd}
                                         & \D^{\on{perf}}(\GS) \arrow[ld] \arrow[rd] &                   \\
\C^{\on{rel}}_{\GS} \arrow[rr, two heads] &                                & \on{CoSing}(\GS)
\end{tikzcd}
\]
as the cluster tilting subcategories arise as the images of $\GS$, see also \Cref{thm:GS_cluster_tilting}. Denote by $\T\subset \mathcal{H}_{\GS}$ and $\T'\subset \on{CoSing}(\GS)$ the images of $\GS$. Condition $3)$ follows from \Cref{lem:endomorphismalgebra} (with $i=0$) and the fact that the functor $\on{ho}
\GS\subset \on{ho}\mathcal{F}^{\on{rel}}\to \on{ho}\mathcal{H}_{\GS}$ is fully faithful.

An inflation in the Higgs category $\mathcal{H}_{\GS}$ from $T$ to $T$ is by definition a morphism $\alpha\colon T\to T$ with the property that its cofiber in $\mathcal{C}^{\on{rel}}_{\GS}$ again lies in $\mathcal{H}_{\GS}$. By \Cref{lem:fundamentaldomain}, this is the case if and only if $\D(f_{\GS})^R(\alpha)$ is a split inclusion. Hence, this is the case if and only if the image of $\alpha$ in $\on{CoSing}(\GS)$ is an inflation, showing condition $4)$. 

It remains to prove that the equivalence of homotopy categories lifts to an equivalence of $\infty$-categories, which amounts to showing that $\tau$ is fully faithful. Let $X,Y\in \mathcal{H}_{\GS}$. We can find exact sequences $T_0\to T_1\to X$ and $Y\to T_0'\to T_1'$ in $\mathcal{H}_{\GS}$ with $T_0,T_0',T_1,T_1'\in \T$. These are fiber and cofiber sequences in the stable $\infty$-category $\C^{\on{rel}}_{\GT}$ and thus also in the subcategory $\mathcal{H}_{\GS}$. Using that the functor $\on{Map}_{\mathcal{H}_{\GS}}(\mhyphen,\mhyphen)$ preserves limits in both entries, we can reduce the fully faithfulness to the assertion that
\[ 
\tau\colon \on{Map}_{\mathcal{H}_{\GS}}(T,T')\to \on{Map}_{\on{CoSing}(\GS)}(\tau(T),\tau(T'))
\]
is an equivalence of spaces for all $T,T'\in \T$. This is the cases if and only if this map induces an equivalence on homotopy groups, which amount to the negative extension groups. By \cite[Lem.~6.70]{Che23}, the functor $\D(\GS)^{\on{perf}}\to \C^{\on{rel}}_{\GS}$ induces equivalences on the negative self-extension groups between objects in $\GS$ and their images in $T$. The same is true for the functor $\D(\GS)^{\on{perf}}\to \on{CoSing}(\GS)$ by \Cref{lem:endomorphismalgebra}. 
\end{proof}

\appendix

\section{Ice quivers with potential, relative Ginzburg dg categories and amalgamation}\label{subsec:icequiveramalgamation} 

The goal of this appendix is to prove a general gluing result for relative Ginzburg dg categories, providing a relative $3$-Calabi--Yau categorical analog of the amalgamation procedure of Fock--Goncharov, see \Cref{thm:Ginzburg_amalgamation}.

\begin{definition}
\begin{enumerate}[(1)]
\item A quiver $Q$ consists of a finite set of vertices $Q_0$ and a finite set of arrows $Q_1$ together with source and target functions $s,t\colon Q_1\to Q_0$. 
\item An ice quiver $(Q,F)$ consists of a quiver $Q$, together with a subquiver $F\subset Q$ of frozen vertices and arrows. Note that all frozen arrows must go between frozen vertices but there can be non-frozen arrows between frozen vertices. 
\item A quiver with potential $(Q,W)$ consists of a quiver $Q$ together with a potential, meaning  a $k$-linear sum of cycles in $Q$, considered up to cyclic equivalence. An ice quiver with potential $(Q,F,W)$ consists of an ice quiver $(Q,F)$ together with a potential $W$ for $Q$.
\item Given a quiver with potential $(Q,W)$ and an arrow $a\in Q_1$, the cyclic derivative $\partial_a W$ of $W=\sum_{i=1}^m \lambda_i c_i$, with $\lambda_i\in k$, at $a$ is defined as $\sum_{i=1}^m \lambda_i \partial_a c_i$ with 
\[ \partial_a c_i=\sum_{c_i=uav} vu\,.\]
\end{enumerate}
\end{definition}

We next define the amalgamation of ice quivers with potential, which provides a (slightly less general\footnote{In the amalgamation of cluster seeds, two frozen ice quivers along which are glued only need to have identified sets of vertices, there are no restrictions on the appearing arrows. Note however that arrows between frozen vertices have no cluster algebraic meaning.}) version of the amalgamation construction of \cite{FG06a} that includes potentials. In the amalgamation, we assume that the frozen quivers which are amalgamated coincide up to their orientations in the following sense.

\begin{definition}
We say that two quivers $Q,Q'$ coincide up to their orientations if there are choices of bijections $Q_0\simeq Q_0'$ and $Q_1\simeq Q_1'$ which are compatible with the source and target functions, except that they possibly reverse the source and target (i.e.~direction) of each arrow. We denote by $Q_1\cap Q_1'$ the set of (identified) arrows of $Q_1,Q_1'$ which are oriented the same way in $Q_1$ and $Q_1'$.
\end{definition}

\begin{definition}\label{def:amalgamation}
Let $(Q,E\amalg F,W),(Q',E'\amalg F',W')$ be two ice quivers with potential such that $E$ and $E'$ coincide up to their orientations. Let $\phi\colon E_1\simeq E_1'$ be the corresponding bijection of the sets of arrows. The amalgamation ice quiver with potential along $E$
\[
(\tilde{Q},\tilde{F},\tilde{W}+\tilde{W}')=(Q,E\amalg F,W)\coprod_{E} (Q',E'\amalg F',W')
\] 
is defined as follows:
\begin{itemize}
\item The set vertices of $\tilde{Q}$ is given by $\tilde{Q}_0=Q_0\amalg_{E_0} Q_0'$. The set of arrows of $\tilde{Q}$ is given by $\tilde{Q}_1=(Q_1\backslash E_1)\cup (Q_1'\backslash E_1')\cup (E_1\cap E_1')$
\item The frozen subquiver is $\tilde{F}=F\cup F'$.
\item We denote by $\tilde{W}$ the potential for $(\tilde{Q},\tilde{F})$ obtained from $W$ by replacing each occurrence of an edge $e\in E_1\backslash (E_1\cap E_1')$ by the cyclic derivative $\partial_{\phi(e)} W'$. We similarly denote by $\tilde{W}'$ the potential for $(\tilde{Q},\tilde{F})$ obtained from $W'$ by replacing each occurrence of an edge $e\in E_1'\backslash (E_1\cap E_1')$ by the cyclic derivative $\partial_{\phi^{-1}(e)} W$. The amalgamation potential for $(\tilde{Q},\tilde{F})$ is the sum $\tilde{W}+\tilde{W}'$.
\end{itemize}
\end{definition}

\begin{definition}\label{def:rel_Ginzburg}
Let $(Q,F,W)$ be an ice quiver with potential. We define the relative Ginzburg dg category $\mathscr{G}_{(Q,F,W)}$ as the dg category with objects the vertices of $Q$ and morphisms freely generated by the following generators:
\begin{itemize}
\item For each arrow $a\colon i\to j$ of $Q$ a morphism $a\colon i \to j$ in degree $0$.
\item For each arrow $a\colon i\to j$ of $Q$ a morphism $a^*\colon j \to i$ in degree $1$ (in the homological grading convention).
\item For each frozen arrow $a\colon i \to j$ in $F$ a morphisms $a^\dagger\colon j\to i$ in degree $0$.
\item For each vertex $i\in Q_0$ an endomorphisms $L_i\colon i\to i$ in degree $2$.
\item For each frozen vertex $i\in F_0$ an endomorphism $l_i\colon i\to i$ in degree $1$.
\end{itemize}
The differentials in $\mathscr{G}_{(Q,F,W)}$ are determined on the generators by
\begin{itemize}
\item $d(a)=0$ for each arrow $a$ of $Q$.
\item $d(a^*)=\partial_a W$ for each non-frozen arrow $a$ and $d(a^*)=\partial_a W- a^\dagger$ for each frozen arrow $a$.
\item $d(a^\dagger)=0$ for each frozen arrow $a$.
\item $d(L_i)=\on{id}_i \sum_{a\in Q_1}[a,a^*] \on{id}_i$ for each non-frozen vertex $i$.
\item $d(L_i)=l_i+ \sum_{a\in Q_1}\on{id}_i[a,a^*] \on{id}_i$ and $d(l_i)= \sum_{a\in Q_1}\on{id}_i[a,a^\dagger]\on{id}_i$ for each frozen vertex $i$. 
\end{itemize}
\end{definition}

We remark that $\mathscr{G}_{(Q,F,W)}$ is Morita equivalent to the endomorphism algebra of the direct sum of its objects, which is often called the Ginzburg dg algebra. The advantage of using the Ginzburg dg category over the Ginzburg dg algebra is that it simplifies the description of the boundary functors as well as the gluing along cofibrations. Note also that $\mathscr{G}_{(Q,F,W)}$ describes a relative deformed $3$-Calabi--Yau completion of $kQ$ in the sense of \cite{Yeu16}.

Removing the arrows $a^\dagger,a^*$ and loops $l_i,L_i$ arising from frozen arrows and vertices from $\mathscr{G}_{(Q,F,W)}$ yields a quasi-equivalent dg category. These are however very helpful when gluing, see also the following Lemma:

\begin{lemma}\label{lem:dg_cofibration}
There is a dg functor $\Pi_2(F)\to \mathscr{G}_{(Q,F,W)}$, defined by the assignments $i\mapsto i, l_i\mapsto l_i$ for $i\in F_0$ and $a\mapsto a, a^\dagger\mapsto a^\dagger$ for $a\in F_1$. This dg functor defines a cofibration with respect to the quasi-equivalence model structure on the $1$-category of dg categories.
\end{lemma}

\begin{proof}
Using that the morphisms in $\Pi_2(F)$ are freely generated by the morphisms $\{a,a^\dagger,l_i\}_{a\in F_1}$ and that the differentials in $\Pi_2(F)$ and $\mathscr{G}_{(Q,F,W)}$ match, it is clear that the assignment extends uniquely to a dg functor. 

The quasi-equivalence model structure on the $1$-category of dg categories $\on{dgCat}$ is cofibrantly generated, with generating cofibrations given by 
\begin{itemize}
\item the inclusion $\emptyset \to k$ of the empty dg category into the dg category with one object and endomorphisms $k$.
\item the inclusion of the dg category $S^{n-1}$ with two objects $1,2$ and $\on{Map}_{S^{n-1}}(i,i)=k, \on{Map}_{S^{n-1}}(1,2)=k[n-1], \on{Map}_{S^{n-1}}(2,1)=0$ into the dg category $D^n$ with objects $1,2$ and $\on{Map}_{D^n}(i,i)=k, \on{Map}_{D^n}(1,2)=\on{cone}(\on{id}_{k[n-1]}), \on{Map}_{D^n}(2,1)=0$.
\end{itemize} 
We find that $\mathscr{G}_{(Q,F,W)}$ arises from $\Pi_2(F)$ by iteratively taking pushouts along the following cofibrations: $\emptyset\to k$, one for each vertex in $Q_0\backslash F_0$, $S^{-1}\to D^0$, one for each non-frozen arrow $a$ in $Q$ adding the morphism $a$ (mapping the non-invertible morphisms in $S^{-1}$ to $0$), $S^0\to D^1$, one for each arrow $a$ in $Q$ adding the morphism $a^*$, $S^1\to D^2$, one for each vertex in $Q$, adding the morphism $L_i$.

Thus, the dg functor $\Pi_2(F)\to \mathscr{G}_{(Q,F,W)}$ is a cofibration as an iterated pushout of cofibrations.
\end{proof}

\begin{remark}\label{rem:Pi2_equivalence}
Let $E,E'$ be two quivers which coincide up to their orientation, with $\phi_0\colon E_0\simeq E_0$, $\phi_1\colon E_1\simeq E_1'$ the corresponding bijections. Then there is a dg isomorphism $\psi\colon \Pi_2(E)\simeq \Pi_2(E')$ defined by 
\begin{itemize}
\item $\psi(i)=\phi_0(i)$ for all $i\in E_0$.
\item $\psi(l_i)=-l_{\phi_0(i)}$ for all $i\in E_0$.
\item $\psi(a)=a$ and $\psi(a^\dagger)=-a^\dagger$ if $a\in E_1\cap E_1'$.
\item $\psi(a)=a^\dagger$ and $\psi(a^\dagger)=a$ if $a\in E_1\backslash (E_1\cap E_1')$. 
\end{itemize}
\end{remark}

\begin{example}\label{ex:Dynkininvolution}
Suppose that $E$ is a Dynkin quiver oriented as in \Cref{def:ADEinvolution}. Then the involution $\sigma\colon \Pi_2(E)\simeq \Pi_2(E)$ from \Cref{def:ADEinvolution} arises as the composite $\sigma=\xi\circ \psi$ of a dg isomorphism $\psi\colon \Pi_2(E)\simeq \Pi_2(E')$ from \Cref{rem:Pi2_equivalence} and an additional dg isomorphism $\xi\colon \Pi_2(E')\simeq \Pi_2(E)$ arising from a quiver isomorphism $E\simeq E'$ as follows:  
\begin{enumerate}[(1)]
\item If $E=A_n$, we set $E'=A_n^{\on{op}}$, the quiver isomorphism between $E$ and $E'$ maps the $i$-th vertex to $n+1-i$. This induces a dg isomorphism $\xi\colon \Pi_2(A_n^{\on{op}})\simeq \Pi_2(A_n)$, mapping $l_i$ to $l_{n+1-i}$, $a_i$ to $a_{n-i}^\dagger$ and $a_i^\dagger$ to $a_{n-i}$. 

We consider $E,E'$ as coinciding up to their orientations with corresponding bijection $\phi_0\colon E_0\simeq E_0'$ mapping $i$ to $i$, and such that $E_1\cap E_1'=\emptyset$. The composite dg isomorphism $\Pi_2(A_n)\xrightarrow{\psi}\Pi_2(A_n^{\on{op}})\xrightarrow{\xi} \Pi_2(\A_n)$ coincides with the involution $\sigma$ from \Cref{def:ADEinvolution}.
\item If $E=D_{n}$, with $n$ odd, there is an involution of $E$ that exchanges the vertices $n-1$ and $n$. This induces a dg isomorphism $\xi\colon \Pi_2(D_n)\simeq \Pi_2(D_n)$ with $\xi(i)=\sigma(i)$, 
\[ 
\xi(a_i)=\begin{cases} a_i& i\not=n\\ a_{n-1}& i=n-2\\ a_{n-2}& i=n-1 \end{cases}
\]
and
\[ 
\xi(a_i^\dagger)=\begin{cases} a_i^\dagger& i\not=n\\ a_{n-1}^\dagger& i=n-2\\ a_{n-2}^\dagger& i=n-1 \end{cases}
\]
We choose $E'=E$ with the trivial identifications of the vertices and arrows, so that again $\sigma=\xi \circ \psi$. 
\item If $E=E_{6}$, we set $E'=\tilde{E}_6$ to be the following quiver.
\[
\tilde{E}_6=\begin{tikzcd}
1  & 2 \arrow[l, "a_4"] & 3 \arrow[l, "a_3"] \arrow[d, "a_5"] & 4 \arrow[l, "a_2"] & 5 \arrow[l, "a_1"] \\
                   &                    & 6                                   &                    &  
\end{tikzcd}
\]
There is a quiver isomorphism between $E$ and $E'$, mapping $i$ to $\sigma(i)$, inducing the dg isomorphism $\xi$ with $\xi(l_i)=l_{\sigma(i)}$, 
\[
\xi(a_i)=\begin{cases} a_{5-i} & i\leq 4\\ a_5 & i=5\end{cases}
\]
and 
\[
\xi(a_i^\dagger)=\begin{cases} a_{5-i}^\dagger & i\leq 4\\ a_5 & i=5\end{cases}
\]
We consider $E,E'$ as coinciding up to their orientations with corresponding bijection $\phi_0\colon E_0\simeq E_0'$ mapping $i$ to $i$, and such that $E_1\cap E_1'=\{a_5\}$. 
\item If $E=D_{n}$ with $n$ even, or $E=E_7,E_8$, the quiver isomorphism $E=E'$ is chosen to be trivial, $\xi=\on{id}$, and $\sigma=\psi$.
\end{enumerate}
\end{example}

\begin{theorem}\label{thm:Ginzburg_amalgamation}
Let $(Q,E\amalg F,W),(Q',E'\amalg F',W')$ be two ice quivers with potential such that $E$ and $E'$ coincide up to their orientations. Let $\psi\colon \Pi_2(E)\simeq \Pi_2(E')$ be the arising equivalence of $2$-Calabi--Yau completions from \Cref{rem:Pi2_equivalence}. Then the two cofibrations from \Cref{lem:dg_cofibration} fit into a homotopy pushout square of dg categories as follows.
\[
\begin{tikzcd}
\Pi_2(E) \arrow[d]\arrow[rrd, "\ulcorner", phantom] \arrow[r, "\psi"]\arrow[r, "\simeq"'] & \Pi_2(E')\arrow[r]  & \mathscr{G}_{(Q',E'\amalg F',W')}\arrow[d]\\
\mathscr{G}_{(Q,E\amalg F,W)} \arrow[rr]& & \mathscr{G}_{(Q,E\amalg F,W)\coprod_{E} (Q',E'\amalg F',W')}
\end{tikzcd}
\]
Thus, passing to derived $\infty$-categories and right adjoint functors yields a pullback square in $\on{LinCat}_k$ as follows.
\[
\begin{tikzcd}
\D(\mathscr{G}_{(Q,E\amalg F,W)\coprod_{E} (Q',E'\amalg F',W')})\arrow[rr]\arrow[rrd, "\lrcorner", phantom] \arrow[d] & & \D(\mathscr{G}_{(Q,E\amalg F,W)})\arrow[d]\\
\D(\mathscr{G}_{(Q',E'\amalg F',W')}) \arrow[r] & \D(\Pi_2(E')) \arrow[r, "\simeq"] & \D(\Pi_2(E))
\end{tikzcd}
\]
\end{theorem}

\begin{proof}
By the cofibrancy of the morphisms, the homotopy pushout is given by the strict pushout. It thus suffices to note that $\mathscr{G}_{(Q,E\amalg F,W)\coprod_{E} (Q',E'\amalg F',W')}$ is quasi-equivalent to the strict pushout. The strict pushout $P$ differs from $\mathscr{G}_{(Q,E\amalg F,W)\coprod_{E} (Q',E'\amalg F',W')}$ as follows: 
\begin{itemize}
\item Each arrow $e=e'\in E_1\cap E_1'$ gives rise to three dual morphisms in $P$ (and only one in $\mathscr{G}_{(Q,E\amalg F,W)\coprod_{E} (Q',E'\amalg F',W')}$), namely two morphisms $e^*,(e')^*$ in degree $1$ and one morphism $e^\dagger=(e')^\dagger$ in degree $0$, where the latter identification comes from the arrow $e^\dagger$ in $\Pi_2(E)$.
The differentials are given by $d(e^*)=\partial_e W-e^\dagger$ and $d((e')^*)=-\partial_{e} W'+e^\dagger$, where the different sign arises from the dg isomorphism $\psi$. Up to quasi-equivalence, we can thus replace the three morphisms $e^*,(e')^*,e^\dagger$ by the unique morphisms $e^*+(e')^*$.
\item For each pair of arrows $e\in E_1\backslash (E_1\cap E_1')$  and $e'\in E_1\backslash (E_1\cap E_1')$, identified via $\phi\colon E_1\simeq E_1'$, there are additional morphisms in $P$, namely the degree $0$ morphism $e=(e')^\dagger$ and $e'=e^\dagger$, and the degree $1$ morphisms $e^*, (e')^*$. Furthermore, we have  $d(e^*)=\partial_e W - e^\dagger=\partial_e W -e'$ and $d((e')^*)=\partial_{e'} W' -(e')^\dagger=\partial_{e'} W'-e$. Up to quasi-equivalence, we can thus remove $e^*,(e')^*$ and add identifications $e=\partial_{e'}W'$ and $e'=\partial_e W$. 
\end{itemize}
The above shows that $P$ is related with $\mathscr{G}_{(Q,E\amalg F,W)\coprod_{E} (Q',E'\amalg F',W')}$ via a zig-zag of quasi-equivalences. 
\end{proof}

\bibliography{biblio} 
\bibliographystyle{alpha}

\textsc{Mathematisches Institut, Universit\"at Bonn, Endenicher Allee 60, 53115 Bonn, Germany}

\textit{Email address:} \texttt{christ@math.uni-bonn.de}

\end{document}